\documentclass[11pt]{amsart}

\usepackage{soul}
\usepackage{etoolbox}
\usepackage{amsbsy}
\usepackage{amsmath}
\usepackage[alphabetic]{amsrefs}
\usepackage{amssymb}
\usepackage{amsthm}
\usepackage{amsxtra}
\usepackage{cancel}
\usepackage{comment}
\usepackage{fullpage}
\usepackage{graphicx}
\usepackage{hyperref}
\usepackage{paralist}
\usepackage{mathrsfs}
\usepackage{stmaryrd}
\usepackage{textcomp}
\usepackage[all,2cell]{xy}
\usepackage{enumitem}
\usepackage{pxfonts}
\usepackage{fancyhdr}
\usepackage{tikz}
\usepackage{tikz-qtree}
\usetikzlibrary{cd,matrix,arrows,decorations.pathmorphing}
\usepackage{bbm}

\DeclareRobustCommand{\SkipTocEntry}[5]{}

\newtheorem{prop}{Proposition}[section]

\newtheorem{thm}[prop]{Theorem}
\newtheorem{cor}[prop]{Corollary}
\newtheorem{lem}[prop]{Lemma}

\theoremstyle{definition}
\newtheorem{defn}[prop]{Definition}

\newtheorem{cons}[prop]{Construction}

\newtheorem{notn}[prop]{Notation}

\theoremstyle{remark}

\newtheorem{app}[prop]{Application}
\newtheorem{rem}[prop]{Remark}

\numberwithin{equation}{prop}
\DeclareMathOperator{\Ab}{Ab}

\DeclareMathOperator{\Ad}{Ad}
\DeclareMathOperator{\Aff}{Aff}
\DeclareMathOperator{\Alg}{Alg}
\DeclareMathOperator{\Ann}{Ann}

\DeclareMathOperator{\Assoc}{Assoc}
\DeclareMathOperator{\bAssoc}{\textbf{Assoc}}
\DeclareMathOperator{\Aut}{Aut}
\DeclareMathOperator{\Az}{Az}
\DeclareMathOperator{\B}{B}

\DeclareMathOperator{\bfCAlg}{\mathbf{CAlg}}
\DeclareMathOperator{\bfMod}{\mathbf{Mod}}
\DeclareMathOperator{\BHom}{\mathbf{Hom}}

\DeclareMathOperator{\BM}{BM}
\DeclareMathOperator{\BMod}{BMod}
\DeclareMathOperator{\BPr}{\mathbf{Pr}}

\DeclareMathOperator{\Br}{Br}

\DeclareMathOperator{\CAlg}{CAlg}
\DeclareMathOperator{\Cart}{Cart}
\DeclareMathOperator{\Cat}{Cat}

\DeclareMathOperator{\coCart}{coCart}

\DeclareMathOperator{\Coh}{Coh}

\DeclareMathOperator{\Comm}{Comm}

\DeclareMathOperator{\const}{const}

\DeclareMathOperator{\DAff}{DAff}
\DeclareMathOperator{\dBr}{dBr}

\DeclareMathOperator{\Deraz}{Deraz}
\DeclareMathOperator{\Dg}{Dg}
\DeclareMathOperator{\Disc}{Disc}

\DeclareMathOperator{\DMod}{D-mod}
\DeclareMathOperator{\dR}{dR}
\DeclareMathOperator{\End}{End}

\DeclareMathOperator{\ev}{ev}

\DeclareMathOperator{\Exc}{Exc}

\DeclareMathOperator{\Fin}{Fin}

\DeclareMathOperator{\Fun}{Fun}

\DeclareMathOperator{\GL}{GL}

\DeclareMathOperator{\hocolim}{hocolim}

\DeclareMathOperator{\holim}{holim}

\DeclareMathOperator{\Ho}{Ho}
\DeclareMathOperator{\Hom}{Hom}

\DeclareMathOperator{\id}{id}
\DeclareMathOperator{\Ind}{Ind}

\DeclareMathOperator{\Isom}{Isom}

\DeclareMathOperator{\LFun}{LFun}

\DeclareMathOperator{\LinFun}{LinFun}
\DeclareMathOperator{\LKE}{LKE}
\DeclareMathOperator{\LM}{LM}
\DeclareMathOperator{\LMod}{\mathrm{LMod}}

\DeclareMathOperator{\LTens}{LTens}
\DeclareMathOperator{\Map}{Map}

\DeclareMathOperator{\Mod}{Mod}
\DeclareMathOperator{\Mon}{Mon}
\DeclareMathOperator{\Mor}{Mor}

\DeclareMathOperator{\Perf}{Perf}
\DeclareMathOperator{\PGL}{PGL}
\DeclareMathOperator{\Pic}{Pic}

\DeclareMathOperator{\PreStk}{PreStk}
\DeclareMathOperator{\PrL}{Pr^{\mathrm L}}

\DeclareMathOperator{\PrSt}{Pr^{\mathrm{St}}}

\DeclareMathOperator{\QC}{QC}
\DeclareMathOperator{\QCoh}{QCoh}

\DeclareMathOperator{\res}{res}

\DeclareMathOperator{\REnd}{\mathbf REnd}

\DeclareMathOperator{\RKE}{RKE}
\DeclareMathOperator{\RM}{RM}
\DeclareMathOperator{\RMod}{RMod}

\DeclareMathOperator{\SAff}{SAff}
\DeclareMathOperator{\Sch}{Sch}

\DeclareMathOperator{\Set}{Set}

\DeclareMathOperator{\sign}{sign}

\DeclareMathOperator{\SL}{SL}

\DeclareMathOperator{\Sp}{Sp}

\DeclareMathOperator{\Spec}{Spec}

\DeclareMathOperator{\St}{St}
\DeclareMathOperator{\Stk}{Stk}

\DeclareMathOperator{\Supp}{Supp}
\DeclareMathOperator{\Sys}{Sys}

\DeclareMathOperator{\TCoh}{TCoh}

\DeclareMathOperator{\Tens}{Tens}

\DeclareMathOperator{\TPerf}{TPerf}

\DeclareMathOperator{\VB}{VB}
\DeclareMathOperator{\WCat}{WCat}
\newcommand{\natmap}[5]{\xymatrix{
		#1 \ar@/^/[r]^{#3} \ar@/_/[r]_{#4} \ar@{}[r]|{\Downarrow_{#5}} & #2
}}
\newcommand{\adjoints}[4]{\xymatrix{#1 \ar@/^/[rr]^{#3} & \perp & #2 \ar@/^/[ll]^{#4}}}
\newcommand{\smalladjoints}[4]{\xymatrix@C=5pt{#1 \ar@/^/[rr]^{#3} & \perp & #2 \ar@/^/[ll]^{#4}}}
\newcommand{\xtworightarrows}[4]{\xymatrix{#1 \ar@<.5ex>[r]^{#3} \ar@<-.5ex>[r]_{#4} & #2}}
\newcommand{\xtwoleftarrows}[4]{\xymatrix{#1 &  \ar@<-.5ex>[l]_{#3} \ar@<.5ex>[l]^{#4}  #2}}
\newcommand{\xrightleftarrows}[4]{\xymatrix{#1 \ar@/^/[r]^{#3} & #2 \ar@/^/[l]^{#4}}}

\newcommand{\aft}[0]{\mathrm{aft}}

\newcommand{\cl}[0]{\mathrm{cl}}
\newcommand{\coC}[0]{\mathrm{coC}}

\newcommand{\cont}[0]{\mathrm{cont}}

\newcommand{\desc}[0]{\mathrm{desc}}
\newcommand{\dg}[0]{\mathrm{dg}}
\newcommand{\dgen}[0]{\mathrm{dgen}}
\newcommand{\dgproj}[0]{\mathrm{dg-proj}}

\newcommand{\fin}[0]{\mathrm{fin}}

\newcommand{\fpqc}[0]{\mathrm{fpqc}}
\newcommand{\fppf}[0]{\mathrm{fppf}}
\newcommand{\ft}[0]{\mathrm{ft}}

\newcommand{\gen}[0]{\mathrm{gen}}
\newcommand{\idem}[0]{\mathrm{idem}}

\newcommand{\laft}[0]{\mathrm{laft}}
\newcommand{\lft}[0]{\mathrm{lft}}
\newcommand{\op}[0]{\mathrm{op}}

\newcommand{\pr}[0]{\mathrm{pr}}

\newcommand{\qis}[0]{\mathrm{qis}}
\newcommand{\red}[0]{\mathrm{red}}

\newcommand{\univ}[0]{\mathrm{univ}}

\newcommand{\weq}[0]{\mathrm{weq}}

\theoremstyle{remark}

\theoremstyle{definition}

\theoremstyle{plain}

\DeclareMathOperator{\Crys}{Crys}

\newtheorem{innercustomthm}{Theorem}
\newenvironment{customthm}[1]
  {\renewcommand\theinnercustomthm{#1}\innercustomthm}
  {\endinnercustomthm}
  
\newtheorem{innercustomprop}{Proposition}
\newenvironment{customprop}[1]
  {\renewcommand\theinnercustomprop{#1}\innercustomprop}
  {\endinnercustomprop}
  
\newtheorem{innercustomcor}{Corollary}
\newenvironment{customcor}[1]
  {\renewcommand\theinnercustomcor{#1}\innercustomcor}
  {\endinnercustomcor}

\begin{document}

\title{A homotopical Skolem--Noether theorem}

\author{Ajneet Dhillon}
\email{adhill3@uwo.ca}
\address{Western University, Canada}

\author{P\'al Zs\'amboki}
\email{zsamboki@renyi.hu}
\address{R\'enyi Institute, Hungary}

\begin{abstract}
The classical Skolem--Noether Theorem \cite{giraud1971cohomologie} shows us (1) how we can assign to an Azumaya algebra $A$ on a scheme $X$ a cohomological Brauer class in $H^2(X,\mathbf G_m)$ and (2) how Azumaya algebras correspond to twisted vector bundles. The Derived Skolem--Noether Theorem \cite{lieblich2009compactified} generalizes this result to weak algebras in the derived 1-category locally quasi-isomorphic to derived endomorphism algebras of perfect complexes. We show that in general for a co-family $\mathscr C^\otimes$ of presentable monoidal quasi-categories with descent over a quasi-category with a Grothendieck topology, there is a fibre sequence giving in particular the above correspondences. For a totally supported perfect complex $E$ over a quasi-compact and quasi-separated scheme $X$, the long exact sequence on homotopy splits, thus showing that the adjoint action induces isomorphisms $\pi_i(\Aut_{\Perf} E,\id_E)\to\pi_i\Aut_{\Alg\Perf}({\REnd E},\id_{\REnd E})$ for $i\ge1$. Further applications include complexes in Derived Algebraic Geometry, module spectra in Spectral Algebraic Geometry and ind-coherent sheaves and crystals in Derived Algebraic Geometry in characteristic 0.

\end{abstract}

\maketitle
\tableofcontents

\section{Introduction}

\subsection{The classical theory}

\subsubsection{Skolem--Noether Theorem} \cite{giraud1971cohomologie}*{V, Lemme 4.1} Let $X$ be a scheme (or a locally ringed topos), and $n$ a positive integer. Then there exists a short exact sequence of group sheaves on $X$:
$$
1\to\mathbf G_m\to\mathrm{GL}_n\xrightarrow{\mathrm{Ad}}\mathrm{Aut}_{\mathrm{Alg}}\mathrm{End}_{\mathscr O}\mathscr O^{\oplus n}\to1.
$$
Taking deloopings, we get a fibre sequence of stacks on $X$:
\begin{equation}\tag{1}\label{eq:1-HSN}
\mathrm{LB}\to\VB_n\xrightarrow{\mathrm{End}}\mathrm{Az}_n.
\end{equation}
Here, $\mathrm{LB}$ is the stack of line bundles, $\VB_n$ is the stack of rank-$n$ vector bundles and $\mathrm{Az}_n$ is the stack of \emph{rank-$n$ Azumaya algebras on $X$}, that is forms of the endomorphism algebra $\End_{\mathscr O}\mathscr O^{\oplus n}$.

\subsubsection{From Azumaya algebras to cohomological Brauer classes} Let $A$ be a rank-$n$ Azumaya algebra on an $X$-scheme $U$. Then it is classified by a map $U\xrightarrow{\mathbf c_A}\mathrm{Az}_n$. We can take the homotopy pullback square
\begin{center}
 \begin{tikzpicture}[xscale=3,yscale=2]
	\node (C') at (0,0.5) {$\mathscr X(A)$};
	\node (D') at (1,0.5) {$\VB_n$};
	\node (C) at (0,-0.5) {$U$};
	\node (D) at (1,-0.5) {$\mathrm{Az}_n.$};
	\node at (0.5,0) {$\lrcorner^\mathrm h$};
	\path[->,font=\scriptsize,>=angle 90]
	(C') edge node [above] {$F$} (D')
	(C') edge (C)
	(D') edge node [right] {$\mathrm{End}$} (D)
	(C) edge node [above] {$\mathbf c_A$} (D);
	\end{tikzpicture}
\end{center}

By the fibre sequence (\ref{eq:1-HSN}), we get that $\mathscr X(A)\to U$ is a $\mathbf G_m$-gerbe on $U$, that is a $\mathrm B\mathbf G_m$-bundle. We call it the \emph{gerbe of trivializations}.

\begin{cor} The Azumaya algebra $A$ induces the \emph{cohomological Brauer class} $[\mathscr X(A)]\in H^2(U,\mathbf G_m)$.

\end{cor}

\subsubsection{Azumaya algebras and twisted sheaves} We have seen previously that Azumaya algebras give $\mathbf G_m$-gerbes. The most succint way to express this is to notice that $\mathbf G_m$ is commutative (or $\mathbf E_2$) and thus we can deloop further to get a fibre sequence of 2-stacks on $X$:
\begin{equation}\tag{2}\label{eq:Azumaya and twisted}
\VB\xrightarrow{\mathrm{End}}\mathrm{Az}\xrightarrow{\mathscr X}\mathrm B^2\,\mathbf G_m.
\end{equation}
That is, the map of stacks $\VB\xrightarrow{\mathrm{End}}\mathrm{Az}$ is a $\mathbf G_m$-gerbe
\begin{cor} Let $\mathscr X$ be a $\mathbf G_m$-gerbe on $U$. Then Azumaya algebras on $U$ with Brauer class $[\mathscr X]$ corres\-pond to \emph{$\mathscr X$-twisted vector bundles on $U$}, that is $\B\,\mathbf G_m$-equivariant maps $\mathscr X\xrightarrow F{\VB}$.

\end{cor}

\subsection{Derived Skolem--Noether Theorem} Let $X\to S$ be a morphism of schemes. By the classical Skolem--Noether Theorem, we see that the stack $\B\PGL_n(X/S)$ classifies families of twisted vector bundles of rank $n$. Lieblich has compatified this stack by letting twisted vector bundles degenerate into twisted totally supported perfect coherent sheaves \cite{lieblich2009compactified}*{Theorem 6.2.4}. To get a more algebraic description of these objects, Lieblich has proven the following result:

\begin{thm}[\cite{lieblich2009compactified}*{Theorem 5.1.5}] Let $E$ be a perfect complex on a scheme $X$. Let $\mathscr O_E=\mathrm{Coker}(\mathscr O_X\to\mathbf R\mathrm{End} E)$. Then we have a short exact sequence of sheaves of groups on $X$
$$
1\to\mathscr O_E^\times\to\mathop{\mathrm{Aut}}_{\mathrm h\,\mathrm{Perf}}E\xrightarrow{\mathrm{Ad}}\mathop{\mathrm{Aut}}_{\mathop{\mathrm{Alg}}\,{\mathrm h\,\mathrm{Perf}}}\mathbf R\mathrm{End}E\to1.
$$

\end{thm}

Let $A$ be a weak algebra object in the monoidal 1-category $\mathbf D(X)^{\otimes^{\mathbf L}}$. Then $A$ is a \emph{pre-generalized Azumaya algebra} if there exists:
\begin{enumerate}
 \item an \'etale covering $U\to X$,
 \item a totally supported perfect coherent sheaf $F$ on $U$ and
 \item a quasi-isomorphism of weak algebras $A|^{\mathbf L}U\simeq\mathbf R\End(F)$.
\end{enumerate}

We denote by $\mathscr {PR}_X$ the category fibred in groupoids of pre-generalized Azumaya algebras on $X$. The \emph{stack of generalized Azumaya algebras $\mathrm{GAz}_X$} is the 1-stackification of $\mathscr {PR}_X$. In other words, a \emph{generalized Azumaya algebra} is 1-descent data of pre-generalized Azumaya algebras.

\begin{cor}[\cite{lieblich2009compactified}*{Corollary 5.2.1.9 and \S5.2.4}] Let $U\to X$ be a morphism of schemes. Then the following assertions hold:
 \begin{enumerate}
  \item Let $A$ be a generalized Azumaya algebra on $U$. Then it induces a gerbe of trivializations $\mathscr X(A)\in\B^2\mathbf G_m(U)$ and in particular a cohomological Brauer class $[\mathscr X(A)]\in H^2(U,\mathbf G_m)$.
  \item Let $\mathscr X$ be a $\mathbf G_m$-gerbe on $U$. Then generalized Azumaya algebras with Brauer class $[\mathscr X]$ correspond to $\mathscr X$-twisted totally supported perfect coherent sheaves on $U$.
 \end{enumerate}
\end{cor}

Note how these statements are about the derived 1-categories. Therefore, in this setting, one cannot detect higher descent conditions, which is the reason why 1-stackification is needed in the construction of the stack of generalized Azumaya algebras \cite{lieblich2009compactified}*{\S5.2} unless higher homotopy groups vanish, for example when $X$ is a surface \cite{lieblich2009compactified}*{Proposition 6.4.1}. Using our Homotopical Skolem--Noether Theorem, we can show that the higher homotopy groups always vanish, so stackification is not needed in general:

\begin{customprop}{\ref{prop:pre-generalized Azumaya is stack}} Let $X$ be a quasi-compact and quasi-separated scheme. Then the category fibred in groupoids of pre-generalized Azumaya algebras $\mathscr{PR}_X$ is a 1-stack.

\end{customprop}

\subsection{Derived Azumaya algebras}

\subsubsection{$\Br\stackrel{?}{=}\Br'$}

Let $X$ be a scheme. Two Azumaya algebras $A$ and $B$ on $X$ are \emph{Morita equivalent} if there exists an equivalence of categories $\mathrm{Mod}_A\simeq\mathrm{Mod}_B$ between their categories of modules. The Morita equivalence class of an Azumaya algebra is its \emph{Brauer class}. The set $\mathrm{Br}$ of Morita equivalence classes of Azumaya algebras admits a group structure via tensor product, thus we get the \emph{Brauer group}.
 
One can see that taking cohomological Brauer classes induces an injection
$$
\mathrm{Br}(X)\xrightarrow{A\mapsto[\mathscr X(A)]} H^2(X,\mathbf G_m)_{\mathrm{tors}}=:\mathrm{Br}'(X).
$$
One can ask whether this is an isomorphism. A general positive result is the following:
\begin{thm}[\cite{deJong2003result}*{Theorem 1.1}] Let $X$ be a scheme equipped with an ample invertible sheaf. Then we have $\mathrm{Br}(X)=\mathrm{Br}'(X)$.

\end{thm}
and a negative counterexample is the following:
\begin{cor}[\cite{edidin2001brauer}*{Corollary 3.11}] Let $X=X_1\cup X_2$ be two copies of $\Spec\mathbf C[x^2,xy,y^2]$ glued along the nonsingular locus. Let $\mathscr X'$ be the $\boldsymbol\mu_2$-gerbe obtained by gluing $(\B\boldsymbol\mu_2)_{X_i},\,i=1,2$ along the nontrivial involution $Q\mapsto Q\times_{\boldsymbol\mu_2}P$ where $P=\mathbf A^2\setminus\{(0,0)\}$ with the antipodal action. Let $\mathscr X=\mathscr X'\times_{\B\boldsymbol\mu_2}\mathrm B\mathbf G_m$. Then $[\mathscr X]$ is not represented by an Azumaya algebra.

\end{cor}
This geometrical formulation has been explained to us by Siddharth Mathur. In the situation of the counterexample, there does not exist a twisted vector bundle $\mathscr X\to\VB$. On the other hand, there does exist a twisted totally supported perfect complex $\mathscr X\to\Perf$. Extending the notion of Azumaya algebras accordingly leads us to \emph{derived Azumaya algebras}.

\subsubsection{Derived Azumaya algebras} Let $\QC^{\otimes^{\mathbf L}}$ denote the symmetric monoidal $\infty$-stack of (unbounded) complexes of quasi-coherent modules. Let $X$ be a quasi-compact and quasi-separated scheme. Then an algebra $A\in\Alg\QC(X)$ is a derived Azumaya algebra, if there exists
 \begin{enumerate}
 \item an \'etale covering $U\to X$,
 \item a totally supported perfect complex $E$ on $U$ and
 \item a quasi-isomorphism of algebras $A|^{\mathbf L}U\simeq\REnd E$.
 \end{enumerate}
 Let $A$ be a derived Azumaya algebra. Then we denote by $\Mod_A$ the dg-category of $A$-dg-modules. We say that two derived Azumaya algebras $A$ and $B$ are \emph{Morita equivalent} if there exists an equivalence of dg-categories $\Mod_A\simeq\Mod_B$. The \emph{derived Brauer group} $\dBr$ is the group of Morita equivalence classes of derived Azumaya algebras.
 
 \subsubsection{$\dBr=\dBr'$} Let $\Dg$ denote the $\infty$-stack of dg-categories on $X$. Let $\Dg^{\Az}\subset\Dg$ denote the full substack on locally trivial dg-categories. Let $\Pic\QC\subseteq\QC$ denote the full substack on invertible complexes. By the homotopical Eilenberg--Watts theorem \cite{lurie2014higher}*{Theorem 4.8.4.1}, the functor $\Pic\QC\xrightarrow{L\mapsto(\otimes L)}\Aut_{\Dg}\QC$ is an equivalence. Since a complex $E\in\QC(X)$ is invertible if and only if it is a shift of a line bundle, we get $\Dg^{\Az}\simeq\B^2\,\mathbf G_m\times\B\,\mathbf Z$. Therefore we get an injective map
$\dBr\xrightarrow{[A]\mapsto[\Mod_A]}\dBr':=H^2(X,\mathbf G_m)\times H^1(X,\mathbf Z)$. One can again ask if this map is surjective.

\begin{thm}[\cite{toen2012derived}*{Corollary 4.8}] Let $\mathscr M$ be an fppf-locally trivial dg-category on a quasi-compact and quasi-separated (derived) scheme $X$. Then there exists a derived Azumaya algebra $A$ such that $\mathscr M\simeq\Mod_A$.

\end{thm}
 
 \subsubsection{The case of Spectral Algebraic Geometry} We can develop the theory of derived Azumaya algebras in Spectral Algebraic Geometry too. In this setting, $\mathrm{QC}$ denotes the symmetric monoidal $\infty$-stack of quasi-coherent module spectra.

Here too we have a $\mathrm{dBr}=\mathrm{dBr}'$-type result:

\begin{thm}[\cite{antieau2014brauer}*{Corollary 6.20}] Let $X$ be a quasi-compact and quasi-separated connective spectral scheme. Then every cohomological Brauer
class on $X$ lifts to a derived Azumaya algebra.

\end{thm}

While we were preparing this paper, Benjamin Antieau has told us about the paper \cite{gepner2016brauer}, in which the Homotopical Skolem--Noether Theorem is proven in the case of affine spectral schemes \cite{gepner2016brauer}*{Proposition 5.15}. In what follows, we shall show that the result holds in a more general sense.

\subsection{Homotopical Skolem--Noether Theorem -- general case}

\subsubsection{Towards the Homotopical Skolem--Noether Theorem} Recall that the key ingredient in identifying derived Brauer classes as cohomology classes was the equivalence $\Dg^{\Az}\simeq\B\Pic\QC$ following from the Homotopical Eilenberg--Watts Theorem. In other words, it shows that the outer square in the following diagram is homotopy Cartesian:
\begin{center}
 \begin{tikzpicture}[xscale=3,yscale=2]
	\node (C') at (0,0.5) {$\Pic\QC$};
	\node (D') at (1,0.5) {$\TPerf^\simeq$};
	\node (E') at (2,0.5) {$X$};
	\node (C) at (0,-0.5) {$X$};
	\node (D) at (1,-0.5) {$\Deraz$};
	\node (E) at (2,-0.5) {$\Dg^{\Az}.$};
	\node at (0.5,0) {$\lrcorner^\mathrm h$};
	\node at (1.5,0) {$\lrcorner^\mathrm h$};
	\path[->,font=\scriptsize,>=angle 90]
	(C') edge node [above] {$\otimes^{\mathbf L}E$} (D')
	(C') edge (C)
	(D') edge (E')
	(D') edge node [right] {$\REnd$} (D)
	(E') edge (E)
	(C) edge (D)
	(D) edge node [above] {$\Mod$} (E);
	\end{tikzpicture}
\end{center}

The Homotopical Skolem--Noether Theorem consists of showing that the inner squares are homotopy Cartesian too. Here, $\TPerf^\simeq\subseteq\QC$ is the interior of the full substack on totally supported perfect complexes, $E\in\TPerf(X)$ is a totally supported perfect complex on $X$ and $\Deraz\subseteq\mathrm{Alg}\mathrm{QC}$ is the full substack on derived Azumaya algebras.

The Homotopical Skolem--Noether Theorem gives a long exact sequence on homotopy sheaves:
 \begin{multline*}
 \dotsb\to\pi_{n+1}(\mathrm{Deraz},\mathbf R\mathrm{End}E)\xrightarrow{\mathrm{Mod}}\pi_n(\mathrm{Pic}\mathrm{QC},\mathscr O)\xrightarrow{\otimes^{\mathbf L}E}
 \pi_n(\mathrm{TPerf},E)\\
 \xrightarrow{\mathbf R\mathrm{End}E}\pi_n(\mathrm{Deraz},\mathbf R\mathrm{End}E)\xrightarrow{\mathrm{Mod}}
 \pi_{n-1}(\mathrm{Pic}\mathrm{QC},\mathscr O)\to\dotsb
 \end{multline*}
 In the case of complexes of quasi-coherent modules over a quasi-compact and quasi-separated scheme, this long exact sequence splits to give isomorphisms
 \begin{equation}
 \pi_n(\mathrm{TPerf},E)\xrightarrow[\cong]{\mathbf R\mathrm{End}}\pi_n(\mathrm{Deraz},\mathbf R\mathrm{End}E)\tag{3}
 \end{equation}
 for $n>1$ and a short exact sequence
 \begin{equation}
 1\to\mathbf G_m\to\mathrm{Aut}_{\mathrm{TPerf}}(E)\xrightarrow{\mathrm{Ad}}\mathrm{Aut}_{\mathrm{Deraz}}(\mathbf R\mathrm{End}E)\to1.\tag{4}
 \end{equation}
 Note that (4) is the Derived Skolem--Noether Theorem. This together with the collection of isomorphisms (3) form the key step in showing that $\mathscr{PR}_X$ is a 1-stack.
 
 \subsubsection{Azumaya algebra objects in a presentable monoidal quasi-category} It turns out that the main result can be established on the level of generality of monoidal quasi-categories. In Section 2 we summarize the notions of algebras and modules and in Section 3 Morita Theory in Higher Algebra \cite{lurie2014higher}.
 
 Let $K$ be a quasi-category equipped with a Grothendieck topology and ${}^{\mathrm{op}}\mathscr C^\otimes\to(\mathbf{Assoc}^\otimes)^{\mathrm{op}}\times K$ a family of presentable monoidal quasi-categories with descent over $K^\op$. This gets us coCartesian fibrations over $K^\op$:
\begin{enumerate}
    \item $\LTens\mathscr C$ classifying families of presentable left-tensored quasi-categories,
    \item $\LTens_*\mathscr C$ classifying families of presentable left-tensored quasi-categories with a section and
    \item $\Alg\mathscr C$ classifying families of algebra objects in $\mathscr C$.
\end{enumerate}
These are equipped with
\begin{enumerate}
    \item the forgetful functor $\LTens_*\mathscr C\to\LTens\mathscr C$ that is a left fibration and
    \item the pointed module functor $\Alg\mathscr C\xrightarrow{\Mod_*:A\mapsto(\Mod_A,A_A)}\LTens_*\mathscr C$ that is fully faithful with a right adjoint.
\end{enumerate}
Take $X\in K$. By the Homotopical Morita Theorem (Lemma \ref{lem:Az alg mod}) a pair $(\mathscr M,M)\in\LTens_*\mathscr C(X)$ is in the essential image of $\Mod_*$ if and only if $M\in\mathscr M$ is a \emph{dualizable generator}, that is the map $\Mod_{\End M}\xrightarrow{\otimes_{\End E}M}\mathscr M$ is an equivalence. We denote by $\mathscr C_\dgen\subseteq\mathscr C$ the full subprestack on dualizable generators and $\LTens_\dgen\mathscr C\subseteq\LTens_*\mathscr C$ the full subprestack on pairs $(\mathscr M,M)$ where $M\in\mathscr M$ is a dualizable generator.
 
Let $A\in\Alg\mathscr C(X)$ be an algebra object. Then \emph{$A$ is an Azumaya algebra object} if there exists
\begin{enumerate}
    \item a $\tau$-covering $U\to X$,
    \item a dualizable generator $M\in\mathscr C_\dgen(U)$ and
    \item an equivalence $A|U:=\mathscr C(U)\otimes_{C(X)}A\simeq\End M$.
\end{enumerate}
By the Homotopical Morita Theorem, we have an equivalence $A|U\simeq\End M$ of algebras for some $M\in\mathscr C_\dgen(U)$ if and only if we have an equivalence $\Mod_A|U\simeq\mathscr C(U)$ of presentable quasi-categories left-tensored over $\mathscr C$. We denote by $\Az\mathscr C\subseteq\Alg\mathscr C$ the full subprestack on Azumaya algebras and $\LTens_\dgen^{\Az}\mathscr C\subseteq\LTens_\dgen\mathscr C$ the full subprestack on locally trivial pointed presentable left-tensored quasi-categories over $\mathscr C$. Note that the equivalence $\Alg\mathscr C\xrightarrow{\Mod_*}\LTens_\dgen\mathscr C$ restricts to an equivalence $\Az\mathscr C\to\LTens_\dgen^{\Az}\mathscr C$.
 
\subsubsection{Descent for presentable left-tensored quasi-categories with descent} Take $T\in K$ and let $\mathscr M\in\LTens\mathscr C(T)$ be a presentable left-tensored quasi-category over $\mathscr C(T)$. Then we say that \emph{$\mathscr M$ has $\tau$-descent} if for all \v Cech nerves $\Delta_+^\op\xrightarrow{\Bar U_\bullet}K$ of $\tau$-coverings over $T$, the canonical map
$$
\mathscr M|X\to\lim_{n\ge0}(\mathscr M|U_n)
$$
is an equivalence. We let $\LTens^\desc\mathscr C\subseteq\LTens\mathscr C$ denote the full subprestack on presentable left-tensored quasi-categories over $\mathscr C$ with $\tau$-descent. We let $\LTens^{\Az}\subseteq\LTens^\desc$ denote the full subprestack on locally trivial presentable left-tensored quasi-categories over $\mathscr C$.

The following is the key criterion for the prestacks of interest to have descent: we say that the family $\mathscr C^\otimes$ of presentable monoidal quasi-categories \emph{has $\tau$-descent} if the following conditions hold:
\begin{enumerate}
    \item The underlying Cartesian fibration ${}^\op\mathscr C\xrightarrow pK$ has $\tau$-descent.
    \item \emph{Base changes in $\mathscr C^\otimes$ commute with $\tau$-descent data} that is for the \v Cech nerve $\Delta^{\op}\xrightarrow{U_\bullet}K$ of a $\tau$-covering, a $q$-limit diagram of coCartesian sections $\Bar{\mathscr M_\bullet}\in\Gamma_\coC(U_\bullet^\op,\LTens\mathscr C)$ and a map $V\to\Bar U_{-1}$ in $K$, the base change $\mathscr C(V)\otimes_{\mathscr C(\Bar U_{-1})}\Bar{\mathscr M_\bullet}$ is also a $q$-limit diagram.
\end{enumerate}

\begin{customthm}{\ref{prop:LTen^desc is a stack}} Let $K$ be a quasi-category equipped with a Grothendieck topology $\tau$. Let $\mathscr C^\otimes\to K^\op\times\Assoc^\otimes$ be a coCartesian family of presentable monoidal quasi-categories with $\tau$-descent. Then the following assertions hold.
\begin{enumerate}
 \item The family ${}^\op\LTens^\desc\mathscr C\to K$ of presentable quasi-categories left-tensored over $\mathscr C$ with $\tau$-descent is a $\tau$-stack.
 \item For any object $T\in K$ and associative algebra $A\in\Alg\mathscr C(T)$, the quasi-category $\Mod_A\mathscr C(T)$ left-tensored over $\mathscr C(T)$ of right $A$-modules in $\mathscr C(T)$ has $\tau$-descent.
\end{enumerate} 

\end{customthm}

\begin{customcor}{\ref{cor:descent for dgen}} The prestacks $\Az\mathscr C$ and $\LTens^{\Az}\mathscr C$ also have $\tau$-descent.

\end{customcor}

To make it easier to show that $\mathscr C^\otimes$ has $\tau$-descent, we have proven the following criterion:

\begin{customcor}{\ref{cor:dualizable commutes with descent} }Let $K$ be a quasi-category equipped with a Grothendieck topology $\tau$ and $\mathscr C^\otimes\to K^\op\times\Assoc^\otimes$ a coCartesian fibration of presentable monoidal quasi-categories. Suppose that the following assertions hold:
\begin{enumerate}
 \item The underlying Cartesian fibration ${}^\op\mathscr C\xrightarrow pK$ has $\tau$-descent.
 \item For all objects $U\in K$, the presentable monoidal quasi-category $\mathscr C(U)\in\Alg(\PrL)$ has dualizable underlying presentable quasi-category.
\end{enumerate}
Then base changes in $q$ commute with $\tau$-descent data.

\end{customcor}

\subsubsection{The Homotopical Skolem--Noether Theorem}

\begin{customthm}{\ref{thm:HSN}} Let $K$ be a quasi-category with final object $S$, let $\tau$ be a Grothendieck topology on it. Let $\Cart_{/K}^\tau$ denote the quasi-category of Cartesian fibrations over $K$ with $\tau$-descent. Let $\mathscr X\subseteq\Cart_{/K}^\tau$ denote the full subcategory on right fibrations over $K$ with $\tau$-descent, which is an $\infty$-topos. Let $\mathscr C^\otimes\to\Assoc^\otimes\times K^\op$ be a family of presentable monoidal quasi-categories with $\tau$-descent. Then the following assertions hold:
\begin{enumerate}
\item We have a fibre sequence in $(\Cart_{/K}^\tau)_*$:
$$
({}^\op\mathscr C_\dgen^\simeq,\mathscr O)\xrightarrow{\End}
({}^\op\Az\mathscr C,\mathscr O)\xrightarrow{\Mod}
({}^\op\LTens^{\Az}\mathscr C,\Mod_{\mathscr O})
$$
\item Let $E\in\mathscr C_\dgen(S)$. Then we have a fibre sequence in $(\Cart_{/K}^\tau)_*$:
$$
({}^\op\Pic\mathscr C,\mathscr O)\xrightarrow{\otimes E}
({}^\op\mathscr C_\dgen^\simeq,E)\xrightarrow{\End}
({}^\op\Az\mathscr C,\End E).
$$
\item We have a long exact sequence of homotopy sheaves in $\mathrm h\,\mathscr X$:
\begin{multline*}
\dotsb\to\pi_2({}^\op\Az\mathscr C,\End E)
\to\pi_1({}^\op\Pic\mathscr C,\mathscr O)
\to\pi_1({}^\op\mathscr C_\dgen,E)
\to\pi_1({}^\op\Az\mathscr C,\End E)\to\\
\to\pi_0({}^\op\Pic\mathscr C)
\to\pi_0({}^\op\mathscr C_\dgen)
\to\pi_0({}^\op\Az\mathscr C)
\to\pi_0({}^\op\LTens^{\Az}\mathscr C)
=*.
\end{multline*}
\end{enumerate}
 
\end{customthm}

\subsection{Homotopical Skolem--Noether Theorem -- applications}

\subsubsection{Algebraic Geometry} Let $X$ be a quasi-compact and quasi-separated scheme. We consider $K=\mathrm{Fppf}_X$ and ${}^{\mathrm{op}}\mathscr C^\otimes=\mathrm{QC}^\otimes$. We have already seen $\mathrm{Pic}\,\mathrm{QC}\simeq\mathrm B\,\mathbf G_m\times\mathbf Z$. A complex $E$ is dualizable if and only if it is a perfect complex. It is a generator if and only if it is totally supported. If $E$ is totally supported, then the multiplication map $\mathbf G_m\to\pi_1(\mathrm{QC},E)$ is injective. Therefore, the long exact sequence splits and we get the following result:

\begin{customcor}{\ref{cor:Skolem--Noether for schemes}}[Homotopical Skolem--Noether Theorem for schemes] Let $S$ be a quasi-compact and quasi-separated scheme. Let $\Cart_S^\fppf$ denote the quasi-category of Cartesian fibrations on $\St_S$ which satisfy fppf descent.

(1) Let ${}^\op\TPerf_S:={}^\op(\QC_S)_\dgen$ denote the Cartesian fibration of totally supported perfect complexes on $S$, ${}^\op\Deraz_S:={}^\op\Az\QC_S$ the Cartesian fibration of derived Azumaya algebras on $S$ and ${}^\op\Dg^{\Az}_S:={}^\op\LTens^{\Az}\QC_S$ the Cartesian fibration of locally trivial presentable quasi-categories left-tensored over $\QC_S^\otimes$. Then the sequence in $(\Cart_S^\fppf)_*$:
$$
({}^\op\TPerf_S^\simeq,\mathscr O)\xrightarrow{\End}({}^\op\Deraz_S,\mathscr O)\xrightarrow{\Mod}({}^\op\Dg^{\Az}_S,\mathscr D)
$$
is a homotopy fibre sequence.

(2) Let $E\in\TPerf(S)$ be a totally supported perfect complex on $S$. Then the sequence in $(\Cart_S^\fppf)_*$:
$$
(\B\mathbf G_m\times\mathbf Z,\mathscr O)\xrightarrow{\otimes E}({}^\op\TPerf_S^\simeq,E)\xrightarrow{\End}({}^\op\Deraz_S,\End E)
$$
is a homotopy fibre sequence.

(3) We have isomorphisms of sheaves of groups
$$
\pi_i\Omega({}^\op\TPerf_S,E)\cong\pi_i\Omega({}^\op {\Deraz_S},\REnd E)
$$
for $i>0$, a short exact sequence of sheaves of groups
$$
1\to\mathbf G_m\xrightarrow{a\mapsto a\cdot}\Aut_{\Perf}E\xrightarrow{\Ad}\Aut_{\Deraz}(\REnd E)\to1,
$$
and an exact sequence of pointed sheaves of sets
$$
*\to\mathbf Z=\pi_0(\B\mathbf G_m\times\mathbf Z)\xrightarrow{\otimes E}\pi_0\TPerf_S\xrightarrow{\REnd}\pi_0\Deraz_S\xrightarrow{\Mod}\pi_0\Dg^{\Az}_S=*\to*.
$$

\end{customcor}

\subsubsection{Homotopical Algebraic Geometry} We can apply the Homotopical Skolem--Noether Theorem to Derived and Spectral Algebraic Geometry too:

\begin{customcor}{\ref{thm:HSN for DAG}}[Homotopical Skolem--Noether theorem for Derived and Spectral Algebraic Geometry] Let $S$ be a derived or spectral affine scheme. Let $\Cart_S^\fpqc$ denote the quasi-category of Cartesian fibrations on $\St_S$ which satisfy fpqc descent.

(1) Let ${}^\op(\Perf^\simeq_\gen)_S={}^\op(\QC_S)_\dgen^\simeq$ denote the right fibration of perfect generator complexes on $S$, ${}^\op\Deraz_S:={}^\op\Az\QC_S$ the Cartesian fibration of derived Azumaya algebras on $S$ and ${}^\op\Dg^{\Az}_S:={}^\op\LTens^{\Az}\QC_S$ the Cartesian fibration of locally trivial presentable quasi-categories left-tensored over $\QC_S^\otimes$. Then the sequence in $(\Cart_S^\fpqc)_*$:
$$
({}^\op(\Perf^\simeq_\gen)_S,\mathscr O)\xrightarrow{\End}({}^\op\Deraz_S,\mathscr O)\xrightarrow{\Mod}({}^\op\Dg^{\Az}_S,\mathscr D)
$$
is a homotopy fibre sequence.

(2) Let $E\in\Perf_\gen(S)$ be a perfect generator complex on $S$. Then the sequence in $(\Cart_S^\fpqc)_*$:
$$
({}^\op\Pic\QC_S,\mathscr O)\xrightarrow{\otimes E}({}^\op(\Perf^\simeq_\gen)_S,E)\xrightarrow{\End}({}^\op\Deraz_S,\End E)
$$
is a homotopy fibre sequence.

\end{customcor}

\subsubsection{$\Ind\Coh$ and crystals} In the characteristic 0 case, we can also apply our main result to the families of symmetric monoidal quasi-categories $\Ind\Coh^{\otimes^!}$ and $(\Crys^r)^{\otimes^!}$:

\begin{customcor}{\ref{thm:HSN for IndCoh}}[Homotopical Skolem--Noether Theorem for $\Ind\Coh$] Let $\TCoh\subseteq\Coh$ denote the full substack on totally supported coherent complexes and we let $\Az\Ind\Coh\subseteq\Alg\Ind\Coh^{\otimes^!}$ denote the full substack on algebra objects locally equivalent to endomorphism algebras of totally supported coherent complexes. Then the following assertions hold:
\begin{enumerate}
 \item The following is a fibre sequence in $\Stk_{\lft}$:
 $$
 {}^\op\TCoh^\simeq\xrightarrow{\End}\Az\Ind\Coh^\simeq\xrightarrow{\Mod}{}^\op\LTens^{\Az}\Ind\Coh^\simeq\simeq(\B^2\mathbf G_m\times\B\mathbf Z).
 $$
 \item Let $\mathscr Y$ be a stack locally of finite type and $E\in\TCoh(\mathscr Y)$ a totally supported complex with bounded coherent cohomology sheaves on $\mathscr Y$. Then the following is a fibre sequence in $(\Stk_{\lft})_{/\mathscr Y}$:
 $$
 ((\B\mathbf G_m)\times\mathbf Z)_{\mathscr Y}\simeq{}^\op\Pic\Ind\Coh_{\mathscr Y}\xrightarrow{\otimes E}{}^\op\TCoh_\mathscr Y^\simeq\xrightarrow{\End}{}^\op\Az\Ind\Coh_\mathscr Y^\simeq.
 $$
 \end{enumerate}
 
\end{customcor}

\begin{customcor}{\ref{thm:HSN for Crys}}[Homotopical Skolem--Noether Theorem for $\Crys$] Let ${}^\op\Crys_{\Coh}^{r,\gen}\subseteq{}^\op\Crys^r)$ denote the full substack on coherent generator complexes and let ${}^\op\Az\Crys^r\subseteq{}^\op\Alg(\Crys^r)^{\otimes^!}$ denote the full substack on Azumaya algebra objects. Then the following assertions hold:
\begin{enumerate}
 \item The following is a fibre sequence in $\Stk_{\lft}$:
 $$
 {}^\op(\Crys_{\Coh}^{r,\gen})^\simeq\xrightarrow{\End}(\Az\Crys^r)^\simeq\xrightarrow{\Mod}{}^\op\LTens^{\Az}(\Az\Crys^r)^\simeq\simeq(\B^2\mathbf G_m\times\B\mathbf Z)_{\dR}.
 $$
 \item Let $\mathscr Y$ be a stack locally of finite type and $E\in\Crys_{\Coh}^{r,\dgen}(\mathscr Y)$ a generator complex with bounded coherent cohomology sheaves on $\mathscr Y$. Then the following is a fibre sequence in $(\Stk_{\lft})_{/\mathscr Y_{\dR}}$:
 $$
 ((\B\mathbf G_m)\times\mathbf Z)_{\mathscr Y_{\dR}}\simeq{}^\op\Pic\Crys^r_{\mathscr Y}\xrightarrow{\otimes E}{}^\op(\Crys_{\Coh}^{r,\gen})_\mathscr Y^\simeq\xrightarrow{\End}{}^\op(\Az\Crys^r)_\mathscr Y^\simeq.
 $$
\end{enumerate}
 
\end{customcor}

Finally, the Homotopical Skolem--Noether Theorem for Derived Algebraic Geometry gives a correspondence between twisted crystals and derived Azumaya algebras:
\begin{customcor}{\ref{cor:twisted crystals as deraz}} Let $\mathscr Y$ be a prestack and $T$ a twisting on $\mathscr Y$. Then the quasi-category ${}^\dgen\Crys^{T,l}(\mathscr Y)$ of $T$-twisted left crystals on $\mathscr Y$ that are dualizable generators is equivalent to the quasi-category $\Deraz^T(\mathscr Y_{\dR})$ of derived Azumaya algebras on $\mathscr Y_{\dR}$ with Brauer class $T$.

\end{customcor}

\subsection*{Acknowledgements} Ajneet Dhillon is supported by an NSERC discovery grant. P\'al Zs\'amboki is partially supported by the project NKFIH K 138828.

\section{Algebras and modules in Higher Algebra}\label{subsection:D otimes} In this section, we summarize notions of algebras of modules needed for Morita theory following \cite{lurie2014higher}. We will make free use of the theory of quasi-categories developed in \cite{lurie2009higher}. For a quick summary on stacks fibred in Kan complexes and quasi-categories, see our previous work \cite{dhillon2020twisted}*{\S2.1}.

\begin{notn}
 Let $K$ be a simplicial set and $\mathscr X\to K$ a Cartesian fibration. Then we denote by $\mathscr X^\simeq\subseteq\mathscr X$ its \emph{interior}: for any vertex $T\in K$, the restriction $\mathscr X^\simeq(T)\subseteq\mathscr X(T)$ of the inclusion of the fibre is the inclusion of the largest sub-Kan complex \cite{dhillon2020twisted}*{Definition 2.7}.
\end{notn}

\subsection{Monoidal quasi-categories and algebras}

Higher algebraic structure are encoded as higher categorical diagrams on $\infty$-operads. We shall now expliain what this means.

\begin{notn}\label{n:finStar} Let $\Fin_*$ denote the nerve of the category with

\begin{itemize}

\item objects the pointed finite sets $\langle n\rangle=\{\ast,1,\dotsc,n\}$ for $n\ge0$. We denote $\langle n\rangle^\circ=\{1,\dotsc,n\}$.

\item morphism set $\Hom_{\Fin_*}(\langle m\rangle,\langle n\rangle)=\{\langle m\rangle\xrightarrow\alpha\langle n\rangle:\alpha(\ast)=\ast\}$. A morphism $\langle m\rangle\to\langle n\rangle$ can be thought of as a partially defined morphism $\langle m\rangle^\circ\to\langle n\rangle^\circ$.

\end{itemize}
A map $\langle m\rangle\xrightarrow f\langle n\rangle$ is \emph{inert}, if for each $i\in\langle n\rangle^\circ$, we have $|f^{-1}(\{i\})|=1$. Let $\mathscr C^\otimes\xrightarrow p\Fin_*$ be a morphism of simplicial sets. Then an edge $e$ in $\mathscr C^\otimes$ is \emph{inert}, if it is a $p$-coCartesian edge over an inert edge in $\Fin_*$.

For each $n>0$ and $i\in\langle n\rangle^\circ$, we fix the inert map $\langle n\rangle\xrightarrow{\rho^i}\langle 1\rangle$ with
$$
\rho^i(j)=\begin{cases}
1 & i=j \\
\ast & \text{else.}
\end{cases}
$$

\end{notn}

\begin{defn} Let $\mathscr C^\otimes\xrightarrow p \Fin_*$ be a morphism of simplicial sets. We denote $\mathscr C^\otimes_{\langle1\rangle}$ by $\mathscr C$. Then $p$ (or by abuse of notation: $\mathscr C^\otimes$, or even $\mathscr C$) is a \emph{symmetric monoidal quasi-category}, if it is a coCartesian fibration of $\infty$-operads \cite{lurie2014higher}*{Definition 2.1.2.13}, that is

\begin{enumerate}

\item it is a coCartesian fibration, and

\item for each $n\ge0$, the map $\mathscr C^\otimes_{\langle n\rangle}\xrightarrow{(\rho^i_!)_{i=1}^n}\mathscr C^{\times n}$ is a categorical equivalence.

\end{enumerate}

Because of property (2), we denote by $C_1\oplus\dotsb\oplus C_n\in\mathscr C^\otimes_{\langle n\rangle}$ a preimage along $(\rho^i_!)_{i=1}^n$ of $(C_1,\dotsc,C_n)\in\mathscr C^{\times n}$.

\end{defn}

The idea here is that coCartesian edges give the usual operations. A coCartesian edge over $\langle2\rangle\xrightarrow{1,2\mapsto1}\langle1\rangle$ gives a product operation $C\oplus D\mapsto C\otimes D$. It is a coCartesian edge, so it is unique up to homotopy. By the same uniqueness, for example, we get the homotopy commutative diagram
\begin{center}

\begin{tikzpicture}[scale=1.5]
\node (A) at (210:1) {$\mathscr C_{\langle2\rangle}^\otimes$};
\node (B) at (90:1) {$\mathscr C_{\langle2\rangle}^\otimes$};
\node (C) at (330:1) {$\mathscr C$};
\path[->,font=\scriptsize,>=angle 90]
(A) edge node [above left] {$(1\mapsto2,\,2\mapsto1)_!$} (B)
(A) edge node [above] {$(1,2\mapsto1)_!$} (C)
(B) edge node [above right] {$(1,2\mapsto1)_!$} (C);
\end{tikzpicture}

\end{center}
giving homotopies $C\otimes D\simeq D\otimes C$ natural in $C,D\in\mathscr C$.

\begin{defn} Let $\mathscr C^\otimes\xrightarrow p\Fin_*$ be a symmetric monoidal quasi-category. Then a \emph{commutative algebra object in $\mathscr C$} is a morphism of $\infty$-operads $\Fin_*\xrightarrow A\mathscr C^\otimes$ \cite{lurie2014higher}*{Definition 2.1.2.7}, that is, it is a section of $p$ which takes inert maps to inert maps. The \emph{quasi-category of commutative algebras in $\mathscr C$} is the full subcategory $\CAlg(\mathscr C)\subseteq\Fun_{\Fin_*}(\Fin_*,\mathscr C^\otimes)$ on commutative algebras. We also denote $\Fin_*$ by $\Comm^\otimes$.

\end{defn}

\begin{defn}
	\label{d:assoc}
	The associative $\infty$-operad denoted by $\Assoc^\otimes$ is the $\infty$-operad $\Assoc^\otimes\rightarrow \Fin_*$ constructed
	by taking the nerve of a functor of ordinary categories $\bAssoc^\otimes \rightarrow \Fin_*$.
	 We abuse notation and denote by $\Fin_*$ the ordinary category
	that produces the simplicial set $\Fin_*$. The data of $\bAssoc^\otimes$ is :
	\begin{enumerate}
		\item objects $\langle n\rangle$ for $n\ge0$,
		\item a morphism $\langle m\rangle\to\langle n\rangle$ is given by a map $\langle m\rangle\xrightarrow\alpha\langle n\rangle$ in $\Fin_*$, and for each $i\in\langle n\rangle^\circ$, a linear ordering on the finite set $\alpha^{-1}(i)$, and
		\item composition is given by lexicographical ordering.
	\end{enumerate}
The proof that this produces an $\infty$-operad can be found in \cite{lurie2014higher}*{Example 2.1.1.21 and Remark 4.1.1.4}.
\end{defn}

\begin{defn} A \emph{monoidal quasi-category} is a coCartesian fibration of $\infty$-operads $\mathscr C^\otimes\to\Assoc^\otimes$ \cite{lurie2014higher}*{Definition 4.1.10}.
Oberve that a coCartesian edge $\mathscr C^\otimes_{\langle2\rangle}\xrightarrow m\mathscr C$ over the map $\{1<2\}\to\{1\}$ is a product map
$$
\{C,D\}\mapsto C\otimes D,
$$
and a coCartesian edge over the map $\langle0\rangle\to\langle1\rangle$ gives a unit object $\mathbf 1\in\mathscr C$ for tensor product, where that this is a unit object on the left is shown by the homotopy commutative diagram
\begin{center}

\begin{tikzpicture}[scale=1.5]
\node (A) at (90:1) {$\mathscr C_{\langle2\rangle}^\otimes$};
\node (B) at (210:1) {$\mathscr C$};
\node (C) at (330:1) {$\mathscr C$};
\path[->,font=\scriptsize,>=angle 90]
(B) edge node [above left] {$(\langle1\rangle\xrightarrow{1\mapsto2}\langle2\rangle)_!$} (A)
(A) edge node [above right] {$m$} (C)
(B) edge node [above] {$\id$} (C);
\end{tikzpicture}

\end{center}
which follows from that coCartesian edges are unique up to homotopy. That is, we have a homotopy $\mathbf 1\otimes C\simeq C$ natural in $C\in\mathscr C$.

\end{defn}

\begin{defn} \label{d:picard}
	Let $\mathscr C$ be a monoidal quasi-category. Then $C\in\mathscr C$ is \emph{invertible}, if the endofunctor $\mathscr C\xrightarrow{\otimes C}\mathscr C$ is an equivalence. The \emph{Picard quasi-category} $\Pic(\mathscr C)\subseteq\mathscr C$ is the full subcategory on invertible objects.
 
\end{defn}

\begin{defn} Let $\mathscr C$ be a monoidal quasi-category. Then an \emph{algebra object in $\mathscr C$} is a morphism of $\infty$-operads $\Assoc^\otimes\xrightarrow A\mathscr C^\otimes$. The \emph{quasi-category of algebras in $\mathscr C$} is the full subcategory $\Alg(\mathscr C)\subseteq\Fun_{\Fin_*}(\Assoc^\otimes,\mathscr C^\otimes)$ on algebras.
 
\end{defn}

\begin{defn}\label{d:lm}
	Following the procedure in \ref{d:assoc} we construct an $\infty$-operad $\LM^\otimes$ with a forgetful map
	$LM^\otimes\rightarrow \Assoc^\otimes$. The simplicial set $\LM^\infty$ is obtained be taking the nerve of the category with
	\begin{enumerate}
		\item objects $(\langle n\rangle,S)$ where $S\subseteq\langle n\rangle^\circ$, and
		\item a map $(\langle n'\rangle,S')\to(\langle n\rangle,S)$ is given by a map $\langle n'\rangle\xrightarrow\alpha\langle n\rangle$ in $\Assoc^\otimes$ such that
		\begin{enumerate}
			\item we have $\alpha(S'\cup\{*\})\subseteq S\cup\{*\}$, and
			\item for $s\in S$, we have $\alpha^{-1}(\{s\})\cap S'=\{s'\}$, where $s'=\max\alpha^{-1}(\{s\})$.
		\end{enumerate}
	\end{enumerate}
\end{defn}

We denote $\mathfrak a=(\langle1\rangle,\emptyset),\,\mathfrak m=(\langle1\rangle,\{1\})\in\LM^\otimes$.

\subsection{Left tensored quasi-categories and left modules}

\begin{defn} Let $\mathscr C'$ be a monoidal quasi-category and $\mathscr M$ a quasi-category. A \emph{left-tensored struture of $\mathscr M$ over $\mathscr C'$} is a coCartesian fibration of $\infty$-operads $\mathscr C^\otimes\to\LM^\otimes$ \cite{lurie2014higher}*{Definition 4.2.1.19} such that we have an equivalence of monoidal quasi-categories $\mathscr C'\simeq\mathscr C_{\mathfrak a}$, and an equivalence of quasi-categories $\mathscr M\simeq\mathscr C_{\mathfrak m}$.
\end{defn} 

 We have a monomorphism
$$
\Assoc^\otimes\xrightarrow{\langle n\rangle\mapsto(\langle n\rangle,\emptyset)}\LM^\otimes
$$
restricting along which we get the monoidal quasi-category structure $\mathscr C_{\mathfrak a}^\otimes:=\mathscr C\times_{\LM^\otimes}\Assoc^\otimes$.

We also have to forgetful map
$$
\LM^\otimes\xrightarrow{(\langle n\rangle,S)\mapsto\langle n\rangle}\Assoc^\otimes
$$
pulling back a monoidal quasi-category $\mathscr D$ along which we get the \emph{left-tensored structure of $\mathscr D$ over itself}.

Let $C_1,\dotsc,C_n\in\mathscr C_{\mathfrak a}$ and $M,N\in\mathscr M$. Then we let
$$
\Map_{\mathscr C}(\{C_1,\dotsc,C_n\}\otimes M,N)\subseteq\Map_{\mathscr C}(\{C_1,\dotsc,C_n,M\},N)
$$
denote the full subgroupoid over the map $(\langle n+1\rangle,\{n+1\})\xrightarrow{1<\dotsb<n}(\langle 1\rangle,\{1\})$ in $\LM^\otimes$.

\begin{defn} Let $\mathscr C\to\LM^\otimes$ be a coCartesian fibration of $\infty$-operads which equips $\mathscr M\simeq\mathscr C_{\mathfrak m}$ with a left tensored structure over $\mathscr C_{\mathfrak a}$. Let $M,N\in\mathscr M$. Then a \emph{morphism object $\Mor_{\mathscr C}(M,N)\in\mathscr C_{\mathfrak a}$} is a representing object for the presheaf on $\mathscr C_{\mathfrak a}$:
$$
C\mapsto\Map_{\mathscr C}(\{C\}\otimes M,N).
$$
We say that \emph{$\mathscr M$ is enriched over $\mathscr C_{\mathfrak a}$}, if it has a morphism object for all $M,N\in\mathscr M$.

Let's show how to get an enriched composition map in a quasi-category $\mathscr M$ enriched over $\mathscr C_{\mathfrak a}$. Let $L,M,N\in\mathscr M$. Then we have universal maps
\begin{align*}
\alpha_{LM}\in&\Map_{\mathscr C}({\Mor_{\mathscr C}(L,M)}\otimes L,M),\\
\alpha_{MN}\in&\Map_{\mathscr C}({\Mor_{\mathscr C}(M,N)}\otimes M,N),\\
\alpha_{LN}\in&\Map_{\mathscr C}({\Mor_{\mathscr C}(L,N)}\otimes L,N).
\end{align*}
By the universal property of $\alpha_{LN}$, there exists a map $\Mor_{\mathscr C}(M,N)\otimes\Mor_{\mathscr C}(L,M)\xrightarrow c\Mor_{\mathscr C}(L,N)$ making the diagram

\begin{center}
 
\begin{tikzpicture}[xscale=6, yscale=2]
\node (A) at (-1,0) {$\Mor_{\mathscr C}(M,N)\otimes\Mor_{\mathscr C}(L,M)\otimes L$};
\node (B) at (0,1) {$\Mor_{\mathscr C}(L,N)\otimes L$};
\node (B') at (0,0) {$\Mor_{\mathscr C}(M,N)\otimes M$};
\node (C) at (1,0) {$N.$};
\path[->,font=\scriptsize,>=angle 90]
(A) edge [dashed, bend left] node [above left] {$c\otimes\id$} (B)
(A) edge node [above] {$\id\otimes\alpha_{LM}$} (B')
(B) edge [bend left] node [above right] {$\alpha_{LN}$} (C)
(B') edge node [above] {$\alpha_{MN}$} (C);
\end{tikzpicture}

\end{center}
commutative.
 
\end{defn}

\begin{defn}\label{d:linear} Let $\mathscr M^\otimes\xrightarrow p\LM^\otimes$ and $\mathscr N\otimes\xrightarrow q\LM^\otimes$ be coCartesian fibrations of $\infty$-operads, and
$$
\mathscr M^\otimes\times_{\LM^\otimes}\Assoc^\otimes\xrightarrow\alpha\mathscr C^\otimes\xrightarrow\beta\mathscr N^\otimes\times_{\LM^\otimes}\Assoc^\otimes
$$
equivalences of $\infty$-operads. That is, $\mathscr M$ and $\mathscr N$ are left-tensored over the monoidal quasi-category $\mathscr C$. Then an $\LM^\otimes$-functor $\mathscr M^\otimes\xrightarrow F\mathscr N^\otimes$ is \emph{($\mathscr C$-)linear}, if
\begin{enumerate}
 \item it takes $p$-coCartesian edges to $q$-coCartesian edges, and
 \item we have $F|\mathscr M^\otimes\times_{\LM^\otimes}\Assoc^\otimes=\beta\circ\alpha$.
\end{enumerate}
We let $\LinFun(\mathscr M,\mathscr N)=\LinFun_{\mathscr C}(\mathscr M,\mathscr N)\subseteq\Fun_{\LM^\otimes}(\mathscr M^\otimes,\mathscr N^\otimes)$ denote the full subcategory on linear functors.

\end{defn}

\begin{defn} \label{d:modules}
	Let $\mathscr C^\otimes\xrightarrow p\LM^\otimes$ exhibit $\mathscr C_{\mathfrak m}=\mathscr M$ as a quasi-category left-tensored over the monoidal quasi-category $\mathscr C_{\mathfrak a}^\otimes$. Then a \emph{left module in $\mathscr C$} is a morphism of $\infty$-operads $\LM^\otimes\xrightarrow M\mathscr C^\otimes$. Let $\Assoc^\otimes\xrightarrow A\mathscr C^\otimes$ be an algebra in $\mathscr C_{\mathfrak a}$. Then a \emph{left $A$-module in $\mathscr M$} is a left module $\LM^\otimes\xrightarrow M\mathscr C^\otimes$ such that $M|\Assoc^\otimes=A$. The \emph{quasi-category of left modules in $\mathscr C$} is the full subcategory $\LMod(\mathscr C)\subseteq\Fun_{\Fin_*}(\LM^\otimes,\mathscr C^\otimes)$ on left modules. The \emph{quasi-category of left $A$-modules in $\mathscr M$} is the full subcategory $\LMod_A(\mathscr M)\subseteq\LMod(\mathscr C)$ on left $A$-modules. By abuse of notation, when $\mathscr C^\otimes$ is clear from context, we will also denote this by ${}_A\Mod$ or $\LMod_A$.

Similarly, via the $\infty$-operad $\RM^\otimes$ we can define right modules \cite{lurie2014higher}*{Variant 4.2.1.36}.

\end{defn}

\subsection{Bi-tensored quasi-categories and bimodules}

\begin{defn}\label{defn:BM} Let $\mathscr M$ be a quasi-category, and $\mathscr C_-$, $\mathscr C_+$ monoidal quasi-categories. Then a \emph{bitensored structure of $\mathscr M$ over $\mathscr C_-$ on the left and $\mathscr C_+$ on the right} is a coCartesian fibration of $\infty$-operads $\mathscr C^\otimes\to\BM^\otimes$ \cite{lurie2014higher}*{Definition 4.3.1.17} such that we have an equivalence of quasi-categories $\mathscr C_{\mathfrak m}\simeq\mathscr M$, and equivalences of monoidal quasi-categories $\mathscr C_{\mathfrak a_-}\simeq\mathscr C_-$, $\mathscr C_{\mathfrak a_+}\simeq\mathscr C_+$. Here, the $\infty$-operad $\BM^\otimes$ has
\begin{enumerate}
 \item objects $(\langle n\rangle,c_-,c_+)$ where $c_-,c_+$ are maps $\langle n\rangle^\circ\to[1]$, and
 \item a morphism $(\langle n'\rangle,c_-',c_+')\xrightarrow\alpha(\langle n\rangle,c_-,c_+)$ is a morphism $\langle n'\rangle\xrightarrow\alpha\langle n\rangle$ in $\Assoc^\otimes$ such that for $j\in\langle n\rangle^\circ$ and $\alpha^{-1}\{j\}=\{i_1\succ\dotsb\succ i_k\}$ we have
 \begin{enumerate}
  \item $c_-'(i_1)=c_-(j)$,
  \item $c_+'(i_\ell)=c_-'(i_{\ell+1})$ for $\ell=1,\dotsc,k-1$, and
  \item $c_+'(i_k)=c_+(j)$.
 \end{enumerate}

\end{enumerate}
We let $\mathfrak a_-=(\langle1\rangle,0,0)$, $\mathfrak m=(\langle1\rangle,0,1)$, and $\mathfrak a_+=(\langle1\rangle,1,1)$. Unless specified otherwise, we let $\mathscr C_-=\mathscr C_{\mathfrak a_-}$, and $\mathscr C_+=\mathscr C_{\mathfrak a_+}$.

\end{defn}

\begin{rem} Let $(\langle n\rangle,c_-,c_+)$ and $i\in[1,n]$. Then one can say that $c_-(i)$ says which algebra we're acting with on the left, and $c_+(i)$ says which algebra we're acting with on the right.
 
\end{rem}

\begin{defn}\label{d:bimodules}

We let $\BMod(\mathscr M)=\Alg_{\BM}(\mathscr M)$. Let $A\in\Alg(\mathscr C_-)$ and $B\in\Alg(\mathscr C_+)$. Then an \emph{$(A,B)$-bimodule object (in $\mathscr M$)} is $M\in\BMod(\mathscr M)$ such that $M|\Alg(\mathscr C_-)=A$ and $M|\Alg(\mathscr C_+)=B$. We let ${}_A\Mod_B\subseteq\BMod(\mathscr M)$ denote the full subcategory of $(A,B)$-bimodule objects.

\end{defn}

\begin{cons}\label{cons:Pr} Let $\mathscr C^\otimes\xrightarrow q\BM^\otimes$ be a coCartesian fibration of $\infty$-operads. Then the quasi-category $\LMod(\mathscr C_{\mathfrak m})$ of left module objects can be equipped with the structure of a quasi-category right-tensored over $\mathscr C_+$ \cite{lurie2014higher}*{\S4.3.2}. Heuristically, for $M\in\LMod_A$ and $C\in\mathscr C_+$, the left $A$-module structure on $M\otimes C$ is given by
$$
A\otimes(M\otimes C)\simeq(A\otimes M)\otimes C\xrightarrow{\alpha\otimes C}M\otimes C.
$$
More precisely, we have a map $\LM^\otimes\times\RM^\otimes\xrightarrow{\BPr}\BM^\otimes$ defined as follows.
\begin{itemize}
 \item For objects $(\langle m\rangle,S)\in\LM^\otimes$ and $(\langle n\rangle,T)\in\RM^\otimes$, we have
 $$
 \BPr((\langle m\rangle,S),(\langle n\rangle,T))=(X_*,c_-,c_+),
 $$
 where
 $$
 X=(\langle m\rangle\times T)\cup (S\times\langle n\rangle)\subseteq\langle m\rangle^\circ\times\langle n\rangle^\circ\cong\langle mn\rangle^\circ,
 $$
 and for $(i,j)\in X$, we have
 $$
 c_-(i,j)=\begin{cases}
         0 & j\in T,\\
         1 & j\notin T
        \end{cases},
 \text{ and }
 c_+(i,j)=\begin{cases}
           0 & i\notin S,\\
           1 & i\in S.
          \end{cases}
 $$
 \item Consider morphisms $(\langle m\rangle,S)\xrightarrow\alpha(\langle m'\rangle,S')$ in $\LM^\otimes$ and $(\langle n\rangle,T)\xrightarrow\beta(\langle n'\rangle,T')$ in $\RM^\otimes$. Let $(X_*,c_-,c_+)=\BPr((\langle m\rangle,S),(\langle n\rangle,T))$ and $(X'_*,c'_-,c'_+)=\BPr((\langle m'\rangle,S'),(\langle n'\rangle,T'))$.
 \begin{enumerate}
  \item The image of $\BPr(\alpha,\beta)$ in $\Fin_*$ is the map $X_*\xrightarrow\gamma X'_*$ such that for $(i,j)\in X$, we have
  $$
  \gamma(i,j)=\begin{cases}
                             (\alpha(i),\beta(j)) & \alpha(i)\in\langle m'\rangle^\circ\text{ and }\beta(j)\in\langle n'\rangle^\circ,\\
                             * & \text{else.}
                            \end{cases}
  $$
  \item Let $(i',j')\in X'$. We need to give
  $$
  \gamma^{-1}(i',j')=(\alpha^{-1}(i')\times\beta^{-1}(j'))\cap X
  $$
  a linear ordering satisfying the conditions in Definition \ref{defn:BM}.
  \begin{itemize}
   \item Suppose that $i'\notin S'$. Then we have $\alpha^{-1}(i')\cap S=\emptyset$. As we have $X=(\langle m\rangle\times T)\cup(S\times\langle n\rangle)$, we get $j'\in T'$. Therefore, there exists a unique $j\in T$ such that $\beta(j)=j'$, and we get $\gamma^{-1}(i',j')=\alpha^{-1}(i')\times\{j\}$. We can give this the linear ordering induced by that on $\alpha^{-1}(i')$.
   \item Similarly, if $j'\notin T'$, then we have $\gamma^{-1}(i',j')=\{i\}\times\beta^{-1}(j')$, which we can give the linear ordering induced by that on $\beta^{-1}(j')$.
   \item Suppose that $i'\in S'$ and $j'\in T'$. Then there exists a unique $i\in S$ resp.~$j\in T$ such that $\alpha(i)=i'$ resp.~$\beta(j)=j'$. Thus, by definition of $X$ we get 
   $$
   \gamma^{-1}(i',j')=\alpha^{-1}(i')\times\{j\}\bigsqcup_{\{(i,j)\}}\{i\}\times\beta^{-1}(j').
   $$
   We can give this the unique linear ordering such that
   \begin{enumerate}
    \item on $\alpha^{-1}(i')\times\{j\}$ it is induced by that on $\alpha^{-1}(i')$,
    \item on $\{i\}\times\beta^{-1}(j')$ it is induced by that on $\beta^{-1}(j')$, and
    \item for $i''\in\langle m\rangle^\circ$ and $j''\in\langle n\rangle^\circ$, we have $(i'',j)\preceq(i,j'')$.
   \end{enumerate}

  \end{itemize}

 \end{enumerate}
 
\end{itemize}

With this, we can define a quasi-category  $\overline{\LMod}(\mathscr C_{\mathfrak m})^\otimes$ and  map $\overline{\LMod}(\mathscr C_{\mathfrak m})^\otimes\xrightarrow p\RM^\otimes$ as follows. For a map of simplicial sets $K\xrightarrow f\RM^\otimes$, we can regard $\LM^\otimes\times K$ as a simplicial set over $\BM^\otimes$ as the composite $\LM^\otimes\times K\xrightarrow{\id\times f}\LM^\otimes\times\RM^\otimes\xrightarrow{\BPr}\BM^\otimes$. Therefore, we can let
$$
\Hom_{\RM^\otimes}(K,\overline{\LMod}(\mathscr C_{\mathfrak m})^\otimes)=\Hom_{\BM^\otimes}(\LM^\otimes\times K,\mathscr C^\otimes).
$$
Let $\LMod(\mathscr C_{\mathfrak m})^\otimes\subseteq\overline{\LMod}(\mathscr C_{\mathfrak m})^\otimes$ denote the full subcategory on maps $\LM^\otimes\times\{X\}\to\mathscr C^\otimes$ taking inert edges of $\LM^\otimes$ to inert edges of $\mathscr C^\otimes$.

Note that the postcomposite of the canonical inclusion $\LM^\otimes\times\{\mathfrak m\}\to\LM^\otimes\times\RM^\otimes$ by $\BPr$ is the canonical inclusion $\LM^\otimes\times\{\mathfrak m\}\cong\LM^\otimes\to\BM^\otimes$. This gives an isomorphism $\LMod(\mathscr C_{\mathfrak m})^\otimes\times_{\RM^\otimes}\{\mathfrak m\}\cong\LMod(\mathscr C_{\mathfrak m})$.

Let $K\to\RM^\otimes$ be a map of simplicial sets. Then a map of simplicial sets $K\to\LMod(\mathscr C_{\mathfrak m})$ over $\RM^\otimes$ is a map of simplicial sets $\LM^\otimes\times K\xrightarrow f\mathscr C^\otimes$ over $\BM^\otimes$. Precomposing this by the canonical inclusion $\{\mathfrak m\}\times\RM^\otimes\to\LM^\otimes\times\RM^\otimes$, we get a map of simplicial sets $K\to\mathscr C^\otimes\times_{\BM^\otimes}\RM^\otimes$ over $\RM^\otimes$. In sum, this gives a map of simplicial sets $\LMod(\mathscr C_{\mathfrak m})^\otimes\to\mathscr C^\otimes\times_{\BM^\otimes}\RM^\otimes$ over $\RM^\otimes$. The preimage
$$
\LMod(\mathscr C_{\mathfrak m})^\otimes_{\mathfrak a}\to\mathscr C^\otimes_+
$$
of this map along the canonical inclusion $\mathscr C_+^\otimes=\mathscr C^\otimes\times_{\BM^\otimes}\Assoc^\otimes\to\mathscr C^\otimes\times_{\BM^\otimes}\RM^\otimes$ is a trivial Kan fibration \cite{lurie2014higher}*{Proposition 4.3.2.6}.

The map $\LMod(\mathscr C_{\mathfrak m})^\otimes\xrightarrow p\RM^\otimes$ is a coCartesian fibration of $\infty$-operads \cite{lurie2014higher}*{Proposition 4.3.2.5.~1)}. Therefore, it gives the quasi-category $\LMod(\mathscr C_{\mathfrak m})$ a left-tensored structure over the monoidal quasi-category $\mathscr C_+^\otimes$.

Moreover, precomposition by $\BPr$ gives an equivalence of categories \cite{lurie2014higher}*{Theorem 4.3.2.7}
$$
\BMod(\mathscr C_{\mathfrak m})\to\RMod(\LMod(\mathscr C_{\mathfrak m})).
$$

\end{cons}

\begin{thm}\cite{lurie2014higher}*{Theorem 4.8.4.1}
 Let $\mathscr L$ be a collection of simplicial sets containing $N(\Delta)^\op$. Let $\mathscr C$ a monoidal quasi-category compatible with $\mathscr L$-indexed colimits, $\mathscr M$ a quasi-category left-tensored over $\mathscr C$ compatible with $\mathscr L$-indexed colimits, and $\Assoc^\otimes\xrightarrow A\mathscr C^\otimes$ an associative algebra object. Then the composite
 $$
 \LinFun^{\mathscr L}_{\mathscr C}(\RMod_A\mathscr C,\mathscr M)\xrightarrow{F\mapsto(F\circ)}
 \Fun(\LMod_A\RMod_A\mathscr C,\LMod_A\mathscr M)\xrightarrow{\ev_A}\LMod_A\mathscr M
 $$
 is an equivalence, with quasi-inverse mapping $M\in\LMod_A\mathscr M$ to $\RMod_A\mathscr C\xrightarrow{E\mapsto E\otimes_AM}\mathscr M$.
\end{thm}

Observe that the right hand of the equivalence does not depend on ${\mathscr L}$, hence neither does the left.

\begin{cor}[Homotopical Eilenberg--Watts theorem]
 Let $\Assoc^\otimes\xrightarrow B\mathscr C^\otimes$ be another associative algebra object. Then the map
 $$
 \LinFun^{\mathscr L}_{\mathscr C}(\RMod_A\mathscr C,\RMod_B\mathscr C)\xrightarrow{F\mapsto F(A)}{}_A\BMod_B
 $$
 is an equivalence, with quasi-inverse mapping $M\in {}_A\BMod_B$ to $\RMod_A\mathscr C\xrightarrow{E\mapsto E\otimes_AM}\RMod_B\mathscr C$.
\end{cor}

\subsection{Relative tensor product}

\begin{defn} We let $\Tens^\otimes$ denote the generalized $\infty$-operad with
\begin{enumerate}
 \item objects tuples $(\langle n\rangle,[k],c_-,c_+)$ where $\langle n\rangle\in\Assoc^\otimes$, $[k]\in\Delta^\op$, and $c_-,c_+$ are maps of sets $[1,n]\to[k]$ such that
 $$
 c_-(i)\le c_+(i)\le c_-(i)+1\text{ for all }i\in[1,n],
 $$
 and
 \item a morphism $(\langle n\rangle,[k],c_-,c_+)\to(\langle n'\rangle,[k'],c_-',c_+')$ is a pair of a morphism $\langle n\rangle\xrightarrow\alpha\langle n'\rangle$ in $\Assoc^\otimes$ and a morphism $[k']\xrightarrow\lambda[k]$ in $\Delta^\op$ such that for $j\in[1,n']$ with $\alpha^{-1}\{j\}=\{i_1,\prec\dotsb\prec i_m\}$, we have
 \begin{enumerate}
  \item $c_-(i_1)=\lambda(c_-'(j))$,
  \item $c_+(i_\ell)=c_-(i_{\ell+1})$ for $\ell\in[1,m-1]$, and
  \item $c_+(i_m)=\lambda(c_+'(j))$.
 \end{enumerate}

\end{enumerate}
The forgetful functor $\Tens^\otimes\to\Fin_*^\otimes\times\Delta^\op$ is a family of $\infty$-operads \cite{lurie2014higher}*{Definition 2.3.2.10}. For $k\ge0$, the fibre $\Tens^\otimes_{[k]}$ is the $\infty$-operadic colimit of the diagram

\begin{center}
 
\begin{tikzpicture}[xscale=2, yscale=2]
\node (0) at (0,1) {$\Tens^\otimes_{\{0\}}$};
\node (01) at (1,0) {$\Tens^\otimes_{\{0,1\}}$};
\node (1) at (2,1) {$\Tens^\otimes_{\{1\}}$};
\node (dots) at (3,0) {$\dotsb$};
\node (k-1) at (4,1) {$\Tens^\otimes_{\{k-1\}}$};
\node (k-1k) at (5,0) {$\Tens^\otimes_{\{k-1,k\}}$};
\node (k) at (6,1) {$\Tens^\otimes_{\{k\}}$};
\path[->,font=\scriptsize,>=angle 90]
(0) edge (01)
(1) edge (01)
(1) edge (dots)
(k-1) edge (dots)
(k-1) edge (k-1k)
(k) edge (k-1k);
\end{tikzpicture}

\end{center}

\cite{lurie2014higher}*{Proposition 4.4.1.11}. In particular, for a monoidal quasi-category $\mathscr C^\otimes\xrightarrow q\Assoc^\otimes$, we have canonical equivalences
$$
\Alg(\mathscr C)\to\Alg_{\Tens_{[0]}}(\mathscr C),\quad\BMod(\mathscr C)\to\Alg_{\Tens_{[1]}}(\mathscr C),\quad\BMod(\mathscr C)\times_{\Alg(\mathscr C)}\BMod(\mathscr C)\to\Alg_{\Tens_{[2]}}(\mathscr C)
$$
where in the fibre product on the right, the left projection map is the algebra on the right and the right projection map is the algebra on the left.

We let $\Tens^\otimes_\succ$ be the strict pull-back of $\Tens^\otimes\to\Delta^\op$ along the map $\Delta^1\xrightarrow{\{0,2\}\hookrightarrow[2]}\Delta^\op$. Let $A,B,C\in\Alg(\mathscr C)$ and $M\in{}_A\Mod_B,N\in{}_B\Mod_C$. This data determines $F_0\in\Alg_{\Tens_{[2]}}(\mathscr C)$. Let $K\in{}_A\Mod_C$. We say that \emph{$F\in\Alg_{\Tens_\succ}(\mathscr C)$ exhibits $K$ as the relative tensor product $M\otimes_BN$}, if
\begin{enumerate}
 \item we have $F|\Tens_{[2]}=F_0$,
 \item we have $F|\Tens_{[1]}=K$, and
 \item the diagram $\Tens^\otimes_\succ\xrightarrow F\mathscr C^\otimes$ is a $q$-operadic colimit \cite{lurie2014higher}*{Definition 3.1.1.2}.
\end{enumerate}

Let $A$ be an algebra in $\mathscr C$. Then the quasi-category ${}_A\BMod_A$ can be equipped with a monoidal structure given by relative tensor product \cite{lurie2014higher}*{Proposition 4.4.3.12}.
 
\end{defn}

\subsection{Endomorphism algebras and dualizable generators}

\begin{rem} An element $(\langle n\rangle,[k],c_-,c_+)\in\Tens^\otimes$ indexes an $n$-term expression of action of algebras $A_0,\dotsc,A_k$. For $i\in\langle n\rangle^\circ$, we can act on the $i$-th element by $A_{c_-(i)}$ on the left, and $A_{c_+(i)}$ on the right.
 
\end{rem}

\begin{defn} Let $\mathscr C$ be a monoidal quasi-category, $\mathscr M$ a quasi-category left-tensored over $\mathscr C$, and $M\in\mathscr M$ an object. Then an object $C\in\mathscr C$ equipped with a map $C\otimes M\xrightarrow\alpha M$ is an \emph{endomorphism object of $M$}, denoted by $\End_{\mathscr M}(M)$ or $\End(M)$, if it represents the presheaf on $\mathscr C$:
$$
C'\mapsto\Map_{\mathscr M}(C'\otimes M,M).
$$
Let $A$ be an algebra object in $\mathscr C$. Then we say that a left $A$-module structure on $M$ \emph{exhibits $A$ as an endomorphism algebra of $M$}, if $M\in\LMod\mathscr M$ with this left module structure represents the right fibration \cite{lurie2014higher}*{Corollary 4.7.1.42}
$$
\LMod\mathscr M\times_{\mathscr M}\{M\}\to\Alg\mathscr C
$$
given by restriction.

It can be shown that if $C\otimes M\xrightarrow\alpha M$ exhibits $C\in\mathscr C$ as an endomorphism object of $M$, then $\alpha$ lifts to a module structure $M\in\LMod_A\mathscr M$, which exhibits $A\in\Alg\mathscr C$ as an endomorphism algebra of $M$ \cite{lurie2014higher}*{\S4.7.1}, in a way that is unique up to homotopy.
 
\end{defn}

\begin{defn}\label{defn:dual} Let $\mathscr C$ be a monoidal quasi-category with neutral object $\mathscr O$, and $C\in\mathscr C$ an object. Then an object $C^\vee$ equipped with a map $C^\vee\otimes C\xrightarrow{\ev}\mathscr O$ is a \emph{right dual of $C$}, if this data induce an adjunction
$$
\adjoints{\mathscr C}{\mathscr C}{\otimes C}{\otimes C^\vee}.
$$
In this case, the map $C\otimes C^\vee\otimes C\xrightarrow{C\otimes\ev}C$ shows that $A:=C\otimes C^\vee$ is an endomorphism object of $C$. This equips $A$ with an algebra structure, and $C$ with a left $A$-module structure. We say that $C$ is a \emph{dualizable generator}, if the functor
$$
{}_A\Mod\xrightarrow{\otimes_AC}\mathscr C
$$
is essentially surjective. We denote by $\mathscr C_{\dgen}\subseteq\mathscr C$ the full subcategory of dualizable generators.
 
\end{defn}

\begin{prop} Let $\mathscr C$ be a monoidal quasi-category. Suppose that $\mathscr C$ admits geometrical realizations, and the tensor product $\mathscr C\times\mathscr C\xrightarrow\otimes\mathscr C$ commutes with geometrical realizations. Let $C\in\mathscr C$ be a right dualizable object, and $A=\End C$. Then the following assertions hold.
\begin{enumerate}
 \item The right $A$-module $C^\vee$ is a right dual to the left $A$-module $C$.
 \item The functor ${}_A\Mod\xrightarrow{\otimes_AC}\mathscr C$ is fully faithful.
\end{enumerate}
 
\end{prop}

\begin{proof} We can assume $A=C\otimes C^\vee$. Then the map $C^\vee\otimes C\xrightarrow\ev\mathscr O$ factors through the canonical map $C^\vee\otimes C\xrightarrow t C^\vee\otimes_AC$ as $C^\vee\otimes_AC\xrightarrow{\ev_A}\mathscr O$. We claim that the morphisms
\begin{align*}
C^\vee\otimes_AC\xrightarrow{\ev_A}\mathscr O&\text{ in }\mathscr C,\text{ and }\\
A\xrightarrow{\id}C\otimes C^\vee&\text{ in }{}_A\Mod_A
\end{align*}
exhibit $C^\vee\in\Mod_A$ as a right dual to $C\in{}_A\Mod$. That is, we need to show that the composites
\begin{align*}
C\xrightarrow{\simeq}A\otimes_AC\xrightarrow{\id\otimes_AC}&C\otimes C^\vee\otimes_AC\xrightarrow{C\otimes\ev_A}C\otimes_AA\xrightarrow\simeq C,\text{ and }\\
C^\vee\xrightarrow\simeq C^\vee\otimes_AA\xrightarrow{C^\vee\otimes_A\id}&C^\vee\otimes_AC\otimes C^\vee\xrightarrow{\ev_A\otimes C^\vee}\mathscr O\otimes C^\vee\xrightarrow\simeq C^\vee
\end{align*}
are homotopic to $\id_C$ and $\id_{C^\vee}$, respectively \cite{lurie2014higher}*{Proposition 4.6.2.1}. Since the forgetful functors on module categories are conservative \cite{lurie2014higher}*{Corollary 4.3.3.3}, it is enough to prove these assertions in $\mathscr C$. We'll show the first; the second is similar.

Let $\mathscr O\xrightarrow1C^\vee\otimes C$ denote the coevaluation map of the duality $(C,C^\vee)$ in $\mathscr C$. We have the homotopy commutative diagram
\begin{center}

\begin{tikzpicture}[xscale=3, yscale=1.5]
\node (A) at (0,0) {$C$};
\node (B) at (1,1) {$\mathscr O\otimes C$};
\node (B') at (1,-1) {$A\otimes_AC$};
\node (C) at (2,1) {$C\otimes C^\vee\otimes C$};
\node (C') at (2,-1) {$C\otimes C^\vee\otimes_AC$};
\node (D) at (3,1) {$C\otimes\mathscr O$};
\node (D') at (3,-1) {$C\otimes_AA.$};
\node (E) at (4,0) {$C$};
\path[->,font=\scriptsize,>=angle 90]
(A) edge node [above] {$\simeq$} (B)
(A) edge node [above] {$\simeq$} (B')
(B) edge node [above] {$1\otimes C$} (C)
(B') edge node [above] {$\id\otimes_AC$} (C')
(C) edge node [right] {$C\otimes t$} (C')
(C) edge node [above] {$C\otimes\ev$} (D)
(C') edge node [above] {$C\otimes\ev_A$} (D')
(D) edge node [above] {$\simeq$} (E)
(D') edge node [above] {$\simeq$} (E);
\end{tikzpicture}

\end{center}
The top composite is homotopical to $\id_C$ by assumption. Therefore so is the bottom composite, as required.

Therefore, we have an adjunction $\smalladjoints{\Mod_A}{\mathscr C}{\otimes_AC}{\otimes C^\vee}$ with unit map $\id\xrightarrow{\id}\otimes_AC\otimes C^\vee$ and counit map $\otimes C^\vee\otimes_AC\xrightarrow{\otimes\ev_A}\id$. Since the unit map is an equivalence, the left adjoint $\otimes_AC$ is fully faithful, as claimed.

\end{proof}

\begin{cor}\label{cor:dgen equiv} Let $\mathscr C$ be a monoidal quasi-category. Suppose that $\mathscr C$ has geometric realizations, and the tensor product map $\mathscr C\times\mathscr C\to\mathscr C$ preserves geometric realizations. Let $C\in\mathscr C_\dgen$ be a dualizable generator. Then the functors $\Mod_A\xrightarrow{\otimes_AC}\mathscr C$ and $\mathscr C\xrightarrow{\otimes C^\vee}\Mod_A$ are mutually inverse equivalences.
 
\end{cor}

\section{The Morita functor in Higher Algebra}

To set up the Morita functor $A\mapsto\Mod_A$ sending an algebra $A$ in a monoidal quasi-category $\mathscr C$ to the quasi-category $\Mod_A$ of right $A$-modules that is left-tensored over $\mathscr C$, we use classifying objects for monoidal quasi-categories and quasi-categories left-tensored by them. The endomorphism algebra functor will be the right adjoint of the pointed version $A\mapsto(\Mod_A,A)$ of the above functor. To make sure that this right adjoint exists, we will need to restrict to presentable underlying quasi-categories and make sure that the algebra and module structure maps respect colimits. In our main application, $\mathscr C$ will be the monoidal quasi-category of cochain complexes over an $S$-scheme $X$, and we will be interested in the endomorphism algebras $A=\REnd(E)$ for perfect complexes $E\in\mathscr C$, and the dg-categories $\Mod_A$ of right $A$-modules. In this section, we shall follow \cite{lurie2014higher}*{\S4.8}

\subsection{Cartesian monoidal structures and monoid objects}\label{sss:Cartesian monoidal structures}

To get started, we will equip the quasi-category of quasi-categories $\Cat_\infty$ with the symmetric monoidal structure given by taking products. This is a Cartesian monoidal structure, and thus unique up to equivalence.

Let $\mathscr C^\otimes\xrightarrow p N(\Fin_*)$ be an $\infty$-operad. A \emph{lax Cartesian structure on $\mathscr C$} is a functor $\mathscr C^\otimes\xrightarrow\pi\mathscr D$ into a quasi-category such that for all $n\ge0$ and $X\in\mathscr C_n^\otimes$, the collection of canonical maps $\{\pi(C)\xrightarrow{\pi((\rho_i)_!)}\pi(C_i)\}_{i=1}^n$ is a product diagram in $\mathscr D$.

The lax Cartesian structure $\pi$ is a \emph{weak Cartesian structure}, if
\begin{enumerate}
 \item the $\infty$-operad $p$ is a symmetric monoidal quasi-category, and
 \item for any $p$-coCartesian edge $f$ over an active map of the form $\langle n\rangle\to\langle 1\rangle$ in $\Fin_*$, its image $\pi(f)$ is an equivalence in $\mathscr D$.
\end{enumerate}
The weak Cartesian structure $\pi$ is a \emph{Cartesian structure}, if its restriction $\mathscr C\to\mathscr D$ is an equivalence of quasi-categories.

Let $\mathscr C$ be a quasi-category with finite products. Then there exists a Cartesian structure $\mathscr C^\times\to\mathscr C$ \cite{lurie2014higher}*{Proposition 2.4.1.5}. Moreover, this is the unique Cartesian symmetric monoidal structure on $\mathscr C$ up to equivalence \cite{lurie2014higher}*{Corollary 2.4.1.8}. A symmetric monoidal structure $\mathscr C^\otimes$ on $\mathscr C$ is \emph{Cartesian}, if
\begin{enumerate}
 \item the unit object $1\in\mathscr C$ is final, and
 \item for any objects $C,D\in\mathscr C$, the diagram of canonical maps
 $$
 C\xleftarrow\simeq C\otimes1\leftarrow C\otimes D\rightarrow1\otimes D\xrightarrow\simeq D
 $$
 is a product diagram in $\mathscr C$.
\end{enumerate}

Let $\mathscr D$ be a quasi-category, and $\mathscr O^\otimes$ an $\infty$-operad. Then an \emph{$\mathscr O^\otimes$-monoid in $\mathscr D$} is a lax Cartesian structure of the form $\mathscr O^\otimes\to\mathscr D$. We let $\Mon_{\mathscr O}\mathscr D\subseteq\Fun(\mathscr O^\otimes,\mathscr D)$ denote the full subcategory of monoid objects.

Monoid objects give an alternative description of algebra objects in Cartesian symmetric monoi\-dal quasi-categories: let $\mathscr C^\otimes\xrightarrow\pi\mathscr D$ be a Cartesian structure. Then the postcomposition map $\Alg_{\mathscr O}\mathscr C\xrightarrow{\pi\circ}\Mon_{\mathscr O}\mathscr D$ is an equivalence \cite{lurie2014higher}*{Proposition 2.4.2.5}.

\subsection{Families of monoidal quasi-categories}

In this subsection we will construct classifying objects for families of algebras.
Let $\mathscr D=\Cat_\infty$ and $\pi$ be the canonical Cartesian structure $\Cat_\infty^\times\to\Cat_\infty$. Let $\mathscr C^\otimes\xrightarrow p\mathscr O^\otimes$ be a coCartesian fibration. Then it is classified by a functor $\mathscr O^\otimes\xrightarrow{\mathbf c_p}\Cat_\infty$. The coCartesian fibration $p$ is an $\mathscr O$-monoidal quasi-category if and only if the classifying map $\mathbf c_p$ is a monoid object \cite{lurie2014higher}*{Example 2.4.2.4}. Moreover, the map $\Alg_{\mathscr O}\Cat_\infty\xrightarrow{\pi\circ}\Mon_{\mathscr O}\Cat_\infty$ is an equivalence. Thus we see that $\mathscr O$-monoidal quasi-categories are also classified by $\mathscr O$-algebra objects in $\Cat_\infty^\times$.

Let $K$ be a quasi-category, $\mathscr O^\otimes$ an $\infty$-operad, and $\mathscr C^\otimes\xrightarrow p\mathscr O^\otimes\times K$ a coCartesian fibration. Then it is classified by a map $K\xrightarrow{\mathbf c_p}\Fun(\mathscr O^\otimes,\Cat_\infty)$. By construction, the map $\mathbf c_p$ maps into $\Mon_{\mathscr O}\Cat_\infty\subseteq\Fun(\mathscr O^\otimes,\Cat_\infty)$ if and only if $p$ is a \emph{coCartesian $K$-family of $\mathscr O$-monoidal categories}, that is
\begin{enumerate}
 \item in addition to $p$ being a coCartesian fibration,
 \item for all $k\in K$, the fibre $\mathscr C^\otimes_k\xrightarrow{p_k}\mathscr O^\otimes$ is an $\mathscr O$-monoidal quasi-category.
\end{enumerate}
We let $\Cat_\infty^{\Mon}=\Mon_{\Assoc}\Cat_\infty$. It classifies coCartesian families of monoidal quasi-categories.

\subsection{Families of associative algebra objects}\label{sss:families of associative algebra objects}

Let $K$ be a quasi-category, and $\mathscr C^\otimes\xrightarrow p\Assoc^\otimes\times K$ a coCartesian family of monoidal quasi-categories classified. Then a section $A$ of $p$ is a \emph{family of associative algebra objects of $p$}, if for every $k\in K$, the fibre $A_k$ is an associative algebra object of the monoidal quasi-category $p_k$, i.e a morphism of $\infty$-operads.

\begin{notn}\label{notn: partial sections}
 Let $X\xrightarrow qB\times C$ be a map of simplicial sets. Then the \emph{simplicial set of partial sections of $q$ over $C$} is the {simplicial sets} $\Gamma_C(q)$ over $C$ defined by letting for a map of simplicial sets $L\to C$:
 $$
 \Hom_C(L,\Gamma_C(q))=\Hom_{B\times C}(B\times L,X).
 $$
 Sometimes we will let $\Gamma_C(X)=\Gamma_C(q)$.
\end{notn}

\begin{rem} 1) Note that we have an adjunction
$$
\adjoints{(\Set_{\Delta})_{/C}}{(\Set_\Delta)_{/(B\times C)}}{B\times}{\Gamma_C}.
$$
2) In case $C=*$, we get the absolute section object $\Gamma_*X=\Gamma X$.
 
\end{rem}

Let $\Alg\mathscr C\subseteq\Gamma_K\mathscr C^\otimes$ be the full subcategory on associative algebra objects.Then by construction $\Alg\mathscr C$ classifies associative algebra objects in $p$. The map $\Alg\mathscr C\to K$ is a coCartesian fibration \cite{lurie2014higher}*{Lemma 4.8.3.13.~1)}.

{Now we want to classify pairs $(\mathscr C^\otimes,A)$ where $\mathscr C^\otimes$ is a monoidal quasi-category, and $A$ is an associative algebra object in $\mathscr C^\otimes$. Note that the identity map of $\Mon_{\Assoc}\Cat_\infty$ classifies the universal coCartesian family of monoidal quasi-categories $\widetilde{\Mon}_{\Assoc}\Cat_\infty\xrightarrow{p_0}\Assoc^\otimes\times\Cat_\infty$. Therefore, we have a strict fibre product diagram of simplicial sets}
\begin{center}

\begin{tikzpicture}[xscale=5,yscale=2]
\node (B') at (0,1) {$\mathscr C^\otimes$};
\node (A') at (1,1) {$\widetilde{\Mon}_{\Assoc}\Cat_\infty$};
\node (B) at (0,0) {$\Assoc^\otimes\times K$};
\node (A) at (1,0) {$\Assoc^\otimes\times\Mon_{\Assoc}\Cat_\infty.$};
\node at (.5,.5) {$\lrcorner$};
\path[->,font=\scriptsize,>=angle 90]
(B') edge (A')
(B') edge node [right] {$p$} (B)
(A') edge node [right] {$p_0$} (A)
(B) edge node [above] {$\id\times\mathbf c_p$} (A);
\end{tikzpicture}

\end{center}
It follows that, sections $A$ of $p$ correspond to maps 
$$\Assoc^\otimes\times K\to\widetilde{\Mon}_{\Assoc}\Cat_\infty\quad \text{over}\quad \Assoc^\otimes\times\Mon_{\Assoc}\Cat_\infty.$$ 
A map $\Assoc^\otimes\times K\xrightarrow{A'}\widetilde{\Mon}_{\Assoc}\Cat_\infty$ corresponds to a map $K\xrightarrow{A''}\Fun(\Assoc^\otimes,\widetilde{\Mon}_{\Assoc}\Cat_\infty)$. The map $A'$ is over $\Assoc^\otimes\times\Mon_{\Assoc}\Cat_\infty$ if and only if the postcomposite of $A''$ with 
$$
\Fun(\Assoc^\otimes,\widetilde{\Mon}_{\Assoc}\Cat_\infty)\xrightarrow{p_0\circ}\Fun(\Assoc^\otimes,\Assoc^\otimes\times\Mon_{\Assoc}\Cat_\infty)
$$ factors through the product of partially constant maps
$$
\Mon_{\Assoc}\Cat_\infty\xrightarrow{c'(\{\mathscr D^\otimes\})=(\id_{\Assoc^\otimes},\const_{\{\mathscr D^\otimes\}})}\Fun(\Assoc^\otimes,\Assoc^\otimes\times\Mon_{\Assoc}\Cat_\infty).
$$
Therefore, pairs $(\mathscr C^\otimes,A)$ of coCartesian families of monoidal quasi-categories and a section are classified by the strict fibre product of simplicial sets
\begin{center}

\begin{tikzpicture}[xscale=7,yscale=2]
\node (B') at (0,1) {$\widetilde{\Cat}_\infty^{\Alg}$};
\node (A') at (1,1) {$\Fun(\Assoc^\otimes,\widetilde{\Mon}_{\Assoc}\Cat_\infty)$};
\node (B) at (0,0) {$\Mon_{\Assoc}\Cat_\infty$};
\node (A) at (1,0) {$\Fun(\Assoc^\otimes,\Assoc^\otimes\times\Mon_{\Assoc}\Cat_\infty).$};
\node at (.5,.5) {$\lrcorner$};
\path[->,font=\scriptsize,>=angle 90]
(B') edge (A')
(B') edge (B)
(A') edge node [right] {$p_0\circ$} (A)
(B) edge node [above] {$c'$} (A);
\end{tikzpicture}

\end{center}
Most importantly, pairs $(\mathscr C^\otimes,A)$ of coCartesian families of monoidal quasi-categories and families of associative algebra objects are classified by the full subcategory $\Cat_\infty^{\Alg}\subseteq\widetilde{\Cat}_\infty^{\Alg}$ on pairs $(\mathscr C^\otimes,A)$ of monoidal quasi-categories and associative algebra objects.

\subsection{Compatibility with colimits}\label{ss:compatibility with colimits}

Let $K,L$ be simplicial sets, and $\mathscr C^\otimes\xrightarrow p\Assoc^\otimes\times K$ a coCartesian family of monoidal quasi-categories. Then we say that \emph{$p$ is compatible with $L$-indexed colimits}, if the following conditions are satisfied.
\begin{enumerate}
 \item Let $k\in K$ be a vertex. Then the fibre monoidal quasi-category $\mathscr C^\otimes\xrightarrow{p_k}\Assoc^\otimes$ commutes with $L$-indexed colimits. That is,
 \begin{enumerate}
  \item the underlying quasi-category $\mathscr C_k$ has $L$-indexed colimits, and
  \item the tensor product functor $\mathscr C_k\times\mathscr C_k\to\mathscr C_k$ commutes with $L$-indexed colimits componentwise.
 \end{enumerate}
 \item Let $k\xrightarrow ek'$ be an edge in $K$. Then the induced functor on the underlying quasi-categories $\mathscr C_k\to\mathscr C_{k'}$ commutes with $L$-indexed colimits.

\end{enumerate}
Let $\mathscr L$ be a collection of simplicial sets. Then we say that \emph{$p$ commutes with $\mathscr L$-indexed colimits}, if for all $L\in\mathscr L$, $p$ commutes with $L$-indexed colimits.

Let $K\xrightarrow{\mathbf c_p}\Mon_{\Assoc}\Cat_\infty$ classify $p$. Then we have the following.
\begin{enumerate}
 \item For a vertex $k\in K$, the fibre $\mathscr C^\otimes_k\xrightarrow{p_k}\Assoc^\otimes$ is equivalent to the pullback of $\widetilde{\Mon}_{\Assoc}\Cat_\infty\xrightarrow{p_0}\Assoc^\otimes\times\Mon_{\Assoc}\Cat_\infty$ along the composite $$\Assoc^\otimes\times\{k\}\hookrightarrow\Assoc^\otimes\times K\xrightarrow{\id\times\mathbf c_p}\Assoc^\otimes\times\Mon_{\Assoc}\Cat_\infty.$$
 \item For an edge $k\xrightarrow ek'$ in $K$, the induced map $\mathscr C_k\to\mathscr C_{k'}$ on the underlying quasi-categories, as an edge $\Delta^1\xrightarrow{e_*}\Cat_\infty$, classifies the pullback of $\mathscr C^\otimes\xrightarrow p\Assoc^\otimes\times K$ along the inclusion $\Delta^{\{\mathfrak a\}}\times\Delta^e\to\Assoc^\otimes\times K$. Therefore, the functor $e_*$ is naturally equivalent to the composite
 $$
 \Delta^e\hookrightarrow K\xrightarrow{\mathbf c_p}\Mon_{\Assoc}\Cat_\infty\xrightarrow{\circ(\Delta^{\{\langle1\rangle\}}\hookrightarrow\Assoc^\otimes)}\Cat_\infty.
 $$
\end{enumerate}
This shows the following for a collection of simplicial sets $\mathscr L$.
\begin{enumerate}
 \item Let $\Cat_\infty^{\Mon}(\mathscr L)=\Mon_{\Assoc}^{\mathscr L}\Cat_\infty\subseteq\Mon_{\Assoc}\Cat_\infty$ be the largest subcategory with
 \begin{enumerate}
  \item vertices classifying monoidal quasi-categories compatible with $\mathscr L$-indexed colimits, and
  \item edges classifying monoidal functors $\mathscr C^\otimes\to\mathscr D^\otimes$ such that the restriction $\mathscr C\to\mathscr D$ to underlying quasi-categories commutes with $\mathscr L$-indexed colimits.
 \end{enumerate}
 Then $\Mon_{\Assoc}^{\mathscr L}\Cat_\infty$ classifies families of monoidal quasi-categories compatible with $\mathscr L$-indexed colimits.
 \item Let
 $$
 \widetilde{\Mon}_{\Assoc}^{\mathscr L}\Cat_\infty=\Mon_{\Assoc}^{\mathscr L}\Cat_\infty\varprod_{\Mon_{\Assoc}\Cat_\infty}\widetilde{\Mon}_{\Assoc}\Cat_\infty.
 $$
 Then the projection map $\widetilde{\Mon}_{\Assoc}^{\mathscr L}\Cat_\infty\to\Mon_{\Assoc}^{\mathscr L}\Cat_\infty$ is the universal family of monoi\-dal quasi-categories compatible with $\mathscr L$-indexed colimits, that is it is classified by the identity map of $\Mon_{\Assoc}^{\mathscr L}\Cat_\infty$.
 \item Let
 $$
 \Cat_\infty^{\Alg}(\mathscr L)=\Cat_\infty^{\Alg}\varprod_{\Cat_\infty^{\Mon}}\Cat_\infty^{\Mon}(\mathscr L).
 $$
 Then it classifies pairs $(\mathscr C^\otimes,A)$ of families of monoidal quasi-categories compatible with $\mathscr L$-indexed colimits, and families of associative algebras on them.
 
\end{enumerate}

\subsection{Families of left-tensored quasi-categories}\label{ss:families of left-tensored quasi-categories}

As in the case of monoidal quasi-categories, we can define and classify families of left tensored quasi-categories. Let $K$ be a simplicial set. Then a \emph{coCartesian family of left-tensored quasi-categories over $K$} is a
\begin{enumerate}
 \item coCartesian fibration $\mathscr M^\otimes\xrightarrow q\LM^\otimes\times K$ such that
 \item for all $k\in K$, the fibre $\mathscr M^\otimes_k\xrightarrow{q_k}\LM^\otimes$ is a left-tensored quasi-category.
\end{enumerate}
Note that the restriction
$$
\mathscr M^\otimes_{\mathfrak a}:=\mathscr M^\otimes\varprod_{\LM^\otimes}\Assoc^\otimes
$$
is a coCartesian family of monoidal quasi-categories. We say that $\mathscr M^\otimes$ is a \emph{coCartesian family of quasi-categories over $K$ left-tensored over $\mathscr M^\otimes_{\mathfrak a}$}.

Let $\mathscr L$ be a collection of simplicial sets. Then we say that \emph{$q$ commutes with $\mathscr L$-indexed colimits}, if
\begin{enumerate}
 \item for each $k\in K$, the fibre left-tensored quasi-category $\mathscr M^\otimes_k\xrightarrow{q_k}\LM^\otimes$ commutes with $\mathscr L$-indexed colimits, that is
 \begin{enumerate}
  \item the underlying quasi-categories $\mathscr M_{\mathfrak m}$ and $\mathscr M_{\mathfrak a}$ admit $\mathscr L$-indexed colimits, and
  \item the tensor product $\mathscr M_{\mathfrak a}\times\mathscr M_{\mathfrak a}\to\mathscr M_{\mathfrak a}$ and left action $\mathscr M_{\mathfrak a}\times\mathscr M_{\mathfrak m}\to\mathscr M_{\mathfrak m}$ functors commute with $\mathscr L$-indexed colimits, and
 \end{enumerate}
 \item for each edge $k\xrightarrow ek'$ in $K$, the induced maps on underlying quasi-categories $\mathscr M_{k,\mathfrak m}\to\mathscr M_{k',\mathfrak m}$ and $\mathscr M_{k,\mathfrak a}\to\mathscr M_{k',\mathfrak a}$ commute with $\mathscr L$-indexed colimits.
\end{enumerate}
Just as in the case of monoidal quasi-categories, left-tensored quasi-categories are classified by
$$
\Cat_{\infty}^{\Mod}:=\Mon_{\LM}\Cat_\infty,
$$
and left-tensored quasi-categories compatible with $\mathscr L$-indexed colimits are classified by the largest subcategory $\Cat_\infty^{\mathscr L}=\Mon_{\LM}^{\mathscr L}\Cat_\infty\subseteq\Mon_{\LM}\Cat_\infty$ such that
\begin{enumerate}
 \item its vertices classify left-tensored quasi-categories compatible with $\mathscr L$-indexed colimits, and
 \item its edges classify equivariant functors $\mathscr M^\otimes\xrightarrow{f}\mathscr N^\otimes$ such that the restrictions to the underlying quasi-categories $\mathscr M^\otimes_{\mathfrak m}\xrightarrow{f_{\mathfrak m}}\mathscr M^\otimes_{\mathfrak m}$ and $\mathscr M^\otimes_{\mathfrak a}\xrightarrow{f_{\mathfrak a}}\mathscr N^\otimes_{\mathfrak a}$.
\end{enumerate}

\subsection{Families of right module objects}\label{sss:families of right module objects}

Let $K$ be a simplicial set, and $\mathscr M^\otimes\xrightarrow q\RM^\otimes\times K$ a coCartesian family of right-tensored quasi-categories. Then a section $\RM^\otimes\times K\xrightarrow M\mathscr M^\otimes$ is a \emph{family of right module objects}, if for each $k\in K$, the restriction $\RM^\otimes\xrightarrow{M_k}\mathscr M^\otimes_k$ is a right module object of $q_k$.

To classify right module objets in $\mathscr M^\otimes$ we can apply the same construction as in the case of associative algebra objects (\ref{sss:families of associative algebra objects}): we let $\RMod\mathscr M\subseteq\Gamma_K(q)$ be the full subcategory on right module objects. Then the induced map $\RMod(\mathscr M)\to K$ is a coCartesian fibration \cite{lurie2014higher}*{Lemma 4.8.3.13.~3)}.

Let $\mathscr C^\otimes\xrightarrow{q'}\Assoc^\otimes\times K$ denote the restriction $q|\Assoc^\otimes\times K$. Then we get a restriction map
$$
\RMod\mathscr M\xrightarrow r\Alg\mathscr C.
$$
In case $q$ commutes with $N(\Delta^\op)$-indexed colimits, the map $r$ is a coCartesian fibration \cite{lurie2014higher}*{Lemma 4.8.3.15}.

Let $\Assoc^\otimes\times K\xrightarrow A\mathscr C^\otimes$ be an associative algebra object. Then a \emph{family of right $A$-modules} is a family of right modules $M$ such that $M|(\Assoc^\otimes\times K)=A$. Therefore, families of right $A$-modules are classified by
$$
\RMod_A\mathscr M_{\mathfrak m}=\Mod_A\mathscr M_{\mathfrak m}=\RMod\mathscr M\varprod_{\Alg\mathscr C}K
$$
where $K\to\Alg\mathscr C$ is the map classifying $A$. Note that by construction the projection map $\RMod_A\mathscr M_{\mathfrak m}\to K$ is a coCartesian fibration.

Let $\mathscr C^\otimes\xrightarrow q\Assoc^\otimes\times K$ be a coCartesian family of monoidal quasi-categories compatible with $N(\Delta^\op)$-indexed colimits, and $\Assoc^\otimes\times K\xrightarrow A\mathscr C^\otimes$ be a family of associative algebra objects. Now we will relativize the construction of \ref{cons:Pr} to give the coCartesian family $\RMod_A\mathscr C$ of right $A$-module objects in $q$ a left-tensored structure over $q$. Let $\BPr_0$ denote the composite $\LM^\otimes\times\RM^\otimes\xrightarrow{\BPr}\BM^\otimes\xrightarrow U\Assoc^\otimes$ where $U$ is the forgetful functor. Then we can take the commutative diagram with strict Cartesian squares
\begin{center}

\begin{tikzpicture}[xscale=4,yscale=2]
\node (BCa) at (0,1) {$\Bar{\mathscr C}^\otimes_{\mathfrak a}$};
\node (AL) at (0,0) {$\Assoc^\otimes\times\LM^\otimes\times K$};
\node (BC) at (1,1) {$\Bar{\mathscr C}^\otimes$};
\node (RL) at (1,0) {$\RM^\otimes\times\LM^\otimes\times K$};
\node (C) at (2,1) {$\mathscr C^\otimes$};
\node (A) at (2,0) {$\Assoc^\otimes\times K.$};
\path[->,font=\scriptsize,>=angle 90]
(BCa) edge (BC)
(BCa) edge node [right] {$\Bar q_a$} (AL)
(BC) edge (C)
(BC) edge node [right] {$\Bar q$} (RL)
(C) edge node [right] {$q$} (A)
(AL) edge [bend right] node [right] {$\Bar A$} (BCa)
(AL) edge (RL)
(RL) edge node [above] {$\BPr_0$} (A)
(A) edge [bend right] node [right] {$A$} (C);
\end{tikzpicture}

\end{center}
With this, we can let
$$
\RMod_A(\mathscr C)^\otimes=\RMod_{\Bar A}\Bar{\mathscr C}.
$$
Since we assumed that $q$ commutes with $N(\Delta^\op)$-indexed colimits, so does $\Bar q$, and therefore $\RMod_A(\mathscr C)^\otimes\xrightarrow p\LM^\otimes\times K$ is a coCartesian fibration. Note that for $k\in K$, the fibre $\RMod_A(\mathscr C)^\otimes_k=\RMod_{A_k}(\mathscr C_k)^\otimes$ where the latter is the left-tensored quasi-category of right $A_k$-modules as defined in (\ref{cons:Pr}). Therefore, the map $p$ is a coCartesian family of right-tensored quasi-categories. Let $\Tilde{\mathscr C}^\otimes\xrightarrow p_a\Assoc^\otimes\times K$ denote the pullback $p|(\Assoc^\otimes\times K)$. Then the inclusion $\{\mathfrak m\}\to\RM^\otimes$ induces a morphism $\Tilde{\mathscr C}^\otimes\to\mathscr C^\otimes$ of coCartesian fibrations over $\Assoc^\otimes\times K$. Since its fibre over each $k\in K$ is a trivial Kan fibration (\ref{cons:Pr}), it is a categorical equivalence.

\subsection{The Morita functor}

Note that this construction is natural in $K$. That is, let $K'\xrightarrow fK$ be a morphism of simplicial sets. Then we have
$$
\RMod_{A|f}(\mathscr C|f)^\otimes=\RMod_A(\mathscr C)^\otimes|(\id_{\LM^\otimes}\times f).
$$
In particular, we can give these left-tensored quasi-categories of right module objects as pullbacks along the classifying maps of algebras of the universal left-tensored quasi-category of right module objects. More precisely, let $\mathscr L$ be a collection of simplicial sets containing $N(\Delta^\op)$, and recall (\ref{sss:families of associative algebra objects}) that
$$
\Cat_\infty^{\Alg}(\mathscr L)\subseteq\Mon_{\Assoc^\otimes}^{\mathscr L}\Cat_\infty\varprod_{\Fun(\Assoc^\otimes,\Assoc^\otimes\times\Mon_{\Assoc}^{\mathscr L}\Cat_\infty)}\Fun(\Assoc^\otimes,\widetilde{\Mon}_{\Assoc}^{\mathscr L}\Cat_\infty).
$$
Therefore, it admits maps $\Cat_\infty^{\Alg}(\mathscr L)\xrightarrow{\pi_1}\Mon_{\Assoc}^\otimes(\mathscr L)$ and $\Cat_\infty^{\Alg}(\mathscr L)\times\Assoc^\otimes\xrightarrow{\pi_2}\widetilde{\Mon}_{\Assoc}^{\mathscr L}\Cat_\infty$ induced by the projection maps. Using these, we can construct the universal pair $(\widetilde{\Cat}_\infty^{\Alg}(\mathscr L),A_\univ^{\mathscr L})$ of a coCartesian family of monoidal quasi-categories compatible with $\mathscr L$-indexed colimits, and an associative algebra object in it, as given in the commutative diagram with strict Cartesian square
\begin{center}

\begin{tikzpicture}[xscale=5,yscale=2]
\node (TC) at (0,1) {$\widetilde{\Cat}_\infty^{\Alg}(\mathscr L)$};
\node (TM) at (1,1) {$\widetilde{\Mon}_{\Assoc}^{\mathscr L}\Cat_\infty$};
\node (C) at (0,0) {$\Cat_\infty^{\Alg}(\mathscr L)\times\Assoc^\otimes$};
\node (M) at (1,0) {$\Mon_{\Assoc}^{\mathscr L}\Cat_\infty\times\Assoc^\otimes.$};
\path[->,font=\scriptsize,>=angle 90]
(TC) edge (TM)
(TC) edge (C)
(TM) edge (M)
(C) edge [bend left] node [left] {$A_\univ^{\mathscr L}$} (TC)
(C) edge node [above left] {$\pi_2$} (TM)
(C) edge node [above] {$\pi_1\times\id$} (M);
\end{tikzpicture}

\end{center}
Now we can form the universal coCartesian family
$$
\RMod_{A_\univ^{\mathscr L}}(\widetilde{\Cat}_\infty^{\Alg}(\mathscr L))^\otimes\to\LM^\otimes\times\Cat_\infty^{\Alg}(\mathscr L)
$$
of left-tensored quasi-categories of right module objects compatible with $\mathscr L$-indexed colimits. It is classified by the \emph{Morita functor}
$$
\Cat_\infty^{\Alg}(\mathscr L)\xrightarrow{\Theta}\Mon_{\LM}^{\mathscr L}\Cat_\infty.
$$
The reason we call this functor the Morita functor is the following: Let $K$ be a simplicial set, $\mathscr C^\otimes\xrightarrow q\Assoc^\otimes\times K$ a coCartesian family of monoidal quasi-categories compatible with $\mathscr L$, and $\Assoc^\otimes\times K\xrightarrow A\mathscr C^\otimes$ an associative algebra object. Then the pair $(q,A)$ is classified by a map $K\xrightarrow{\mathbf c_{q,A}}\Cat_\infty^{\Alg}(\mathscr L)$. Since as we said above, the formation of the coCartesian family of left-tensored quasi-categories of right module objects is natural, we get a diagram of homotopy Cartesian squares
\begin{center}

\begin{tikzpicture}[xscale=5,yscale=2]
\node (RC) at (0,1) {$\RMod_A(\mathscr C)^\otimes$};
\node (K) at (0,0) {$\LM^\otimes\times K$};
\node (RU) at (1,1) {$\RMod_{A_\univ^{\mathscr L}}(\widetilde{\Cat}_\infty^{\Alg}(\mathscr L))^\otimes$};
\node (U) at (1,0) {$\LM^\otimes\times\Cat_\infty^{\Alg}(\mathscr L)$};
\node (TM) at (2,1) {$\widetilde{\Mon}_{\LM}^{\mathscr L}\Cat_\infty$};
\node (M) at (2,0) {$\LM^\otimes\times\Mon_{\LM}^{\mathscr L}\Cat_\infty.$};
\path[->,font=\scriptsize,>=angle 90]
(RC) edge (RU)
(RC) edge (K)
(RU) edge (TM)
(RU) edge (U)
(TM) edge (M)
(K) edge node [above] {$\id\times\mathbf c_{q,A}$} (U)
(U) edge node [above] {$\id\times\Theta$} (M);
\end{tikzpicture}

\end{center}
That is, the composite $K\xrightarrow{\mathbf c_{q,A}}\Cat_\infty^{\Alg}(\mathscr L)\xrightarrow\Theta\Mon_{\LM}^{\mathscr L}\Cat_\infty$ classifies $\RMod_A(\mathscr C)^\otimes$. In other words, we have
$$
\Theta(\mathscr C^\otimes,A)=\RMod_A(\mathscr C)^\otimes
$$
As we have seen (\ref{sss:families of right module objects}), we have a canonical equivalence $\RMod_A(\mathscr C)^\otimes|\Assoc^\otimes\times K\to\mathscr C^\otimes$ of coCartesian families of monoidal quasi-categories. Therefore, we get a homotopy commutative diagram of quasi-categories
\begin{center}

\begin{tikzpicture}[scale=2]
\node (A) at (150:1) {$\Cat_\infty^{\Alg}(\mathscr L)$};
\node (B) at (30:1) {$\Cat_\infty^{\Mod}(\mathscr L)$};
\node (C) at (270:1) {$\Cat_\infty^{\Mon}(\mathscr L).$};
\path[->,font=\scriptsize,>=angle 90]
(A) edge node [above] {$\Theta$} (B)
(A) edge node [below left] {$\phi$} (C)
(B) edge node [below right] {$\psi$} (C);
\end{tikzpicture}

\end{center}
It can be shown that \cite{lurie2014higher}*{Proposition 4.8.5.1}
\begin{enumerate}
 \item the forgetful maps $\phi$ and $\psi$ are coCartesian fibrations, and
 \item the Morita functor $\Theta$ carries $\phi$-coCartesian edges to $\psi$-coCartesian edges.
\end{enumerate}

\subsection{The endomorphism algebra functor}

Since we would like to study Azumaya algebra objects, we need an endomorphism algebra functor $E\mapsto\End E$. As we will want this in the form of a morphism of stacks, we need a construction natural in the choice of the left-tensored quasi-category (in our main application, these will be the dg-categories of cochain complexes over schemes). Therefore, in the first place, we need a classifying object for triples $(\mathscr C^\otimes,\mathscr M,M)$, where $\mathscr C^\otimes$ is a monoidal quasi-category, $\mathscr M^\otimes$ is a quasi-category left-tensored over $\mathscr C$, and $M\in\mathscr M$.

Let $\mathscr L$ be a collection of simplicial sets containing $N(\Delta^\op)$. Let $\mathscr S(\mathscr L)\subseteq\mathscr S$ denote the smallest subcategory containing $\Delta^0$, which has $\mathscr L$-indexed colimits. As in $\mathscr S$ finite products commute with small colimits, the subcategory $\mathscr S(\mathscr L)$ has finite products. Therefore, we can equip it with the Cartesian symmetric monoidal structure $\mathscr S(\mathscr L)^\times$ (\ref{sss:Cartesian monoidal structures}). In particular, we have the trivial algebra $\mathbf 1\in\Alg\mathscr S(\mathscr L)$ on $\Delta^0$. One can show that $(\mathscr S(\mathscr L)^\times,\mathbf 1)\in\Cat_\infty^{\Alg}(\mathscr L)$ is an initial object \cite{lurie2014higher}*{Lemma 4.8.5.3}.

We let $\mathfrak M=\Theta(\mathscr S(\mathscr L)^\times,\mathbf1)\in\Cat_\infty^{\Mod}(\mathscr L)$. Let $\mathscr M^\otimes$ be a left-tensored quasi-category compatible with $\mathscr L$-indexed colimits. Then the restriction map
$$
\Map_{\Cat_\infty^{\Mod}(\mathscr L)}(\mathfrak M,\mathscr M^\otimes)\xrightarrow{F\mapsto F(\Delta^0)}\mathscr M^\simeq
$$
is a trivial Kan fibration \cite{lurie2014higher}*{Remark 4.8.5.4}. Thus, objects in the undercategory $\Cat_\infty^{\Mod}(\mathscr L)_{\mathfrak M/}$ correspond to triples $(\mathscr C^\otimes,\mathscr M^\otimes,M)$ of a monoidal quasi-category $\mathscr C^\otimes$ compatible with $\mathscr L$-indexed colimits, a quasi-category $\mathscr M^\otimes$ left-tensored over $\mathscr C$ compatible with $\mathscr L$-indexed colimits, and an object $M\in\mathscr M$, as required.

Since $(\mathscr S(\mathscr L)^\times,\mathbf1)\in\Cat_\infty^{\Alg}(\mathscr L)$ is an initial object, the forgetful functor $\Cat_\infty^{\Alg}(\mathscr L)_{(\mathscr S(\mathscr L)^\times,\mathbf1)/}\xrightarrow U\Cat_\infty^{\Alg}(\mathscr L)$ is a trivial Kan fibration. Therefore, we can let $\Cat_\infty^{\Alg}(\mathscr L)\xrightarrow{\Theta_*}\Cat_\infty^{\Mod}(\mathscr L)_{\mathfrak M/}$ denote the composite
$$
\Cat_\infty^{\Alg}(\mathscr L)\xleftarrow[\simeq]{U}\Cat_\infty^{\Alg}(\mathscr L)_{(\mathscr S(\mathscr L)^\times,\mathbf1)/}
\xrightarrow{\Theta_{(\mathscr S(\mathscr L)^\times,\mathbf1)/}}\Cat_\infty^{\Mod}(\mathscr L)_{\mathfrak M/}.
$$
Informally, it carries $(\mathscr C^\otimes,A)\in\Cat_\infty^{\Alg}(\mathscr L)$ to $(\mathscr C^\otimes,\RMod_A\mathscr C,A_A)\in\Cat_\infty^{\Mod}(\mathscr L)_{\mathfrak M/}$. One can show that the functor $\Theta_*$ is fully faithful \cite{lurie2014higher}*{Theorem 4.8.5.5}.

Then endomorphism functor will be a right adjoint to $\Theta_*$. To get it, we will need to assume that the underlying quasi-categories we're dealing with are presentable. Let $\widehat{\Cat}_\infty$ denote the quasi-category of big quasi-categories. Then we can use the same constructions as above to get classifying objects for higher algebraic structures with big underlying quasi-categories. We let $\mathscr L$ denote the collection of all small simplicial sets. We let
\begin{enumerate}
 \item $\Pr^{\Mon}\subseteq\widehat{\Cat}_\infty^{\Mon}(\mathscr L)$ denote the full subcategory on monoidal quasi-categories with presentable underlying quasi-category,
 \item $\Pr^{\Alg}=\widehat{\Cat}_\infty^{\Alg}(\mathscr L)\varprod_{\widehat{\Cat}_\infty^{\Mon}(\mathscr L)}\Pr^{\Mon}$, and
 \item $\Pr^{\Mod}\subseteq\widehat{\Cat}_\infty^{\Mod}(\mathscr L)$ the full subcategory on presentable quasi-categories left-tensored over a presentable monoidal quasi-category.
\end{enumerate}
Then the Morita functor $\widehat{\Cat}_\infty^{\Alg}\xrightarrow{\Hat\Theta}\widehat{\Cat}_\infty^{\Mod}$ restricts to a functor $\Pr^{\Alg}\xrightarrow{\Hat\Theta}\Pr^{\Mod}$ \cite{lurie2014higher}*{Corollary 4.2.3.7}. We also get the pointed version $\Pr^{\Alg}\xrightarrow{\Hat\Theta_*}\Pr^{\Mod}_{\mathfrak M/}$ the same way. One can show that this functor $\Hat\Theta_*$ is fully faithful, and it admits a right adjoint \cite{lurie2014higher}*{Theorem 4.8.5.11}. The right adjoint maps $(\mathscr C^\otimes,\mathscr M^\otimes,M)\in\Pr^{\Mod}$ to $(\mathscr C^\otimes,\End_{\mathscr M}M)\in\Pr^{\Alg}$.

\subsection{Tensor products of quasi-categories}\label{ss:tensor products of quasi-categories}

Let $\mathscr O^\otimes$ be an $\infty$-operad, and $\mathscr K$ a collection of simplicial sets. To finish, let us discuss another way to classify coCartesian families of $\mathscr O$-monoidal quasi-categories compatible with $\mathscr K$-indexed colimits, which will be useful when we study descent in the next section.

In \S\ref{sss:Cartesian monoidal structures}, we have seen that the quasi-category of quasi-categories $\Cat_\infty$ can be equipped with the Cartesian symmetric monoidal structure, using which we can take the quasi-category of monoid objects $\Mon_{\mathscr O}\Cat_\infty$, and that
\begin{enumerate}
 \item by straightening-unstraightening classifies coCartesian families of $\mathscr O$-monoidal quasi-cate\-gories, and
 \item is equipped with an equivalence $\Mon_{\mathscr O}\Cat_\infty\to\Alg_{\mathscr O}\Cat_\infty$.
\end{enumerate}

Then we have seen in \S\ref{ss:compatibility with colimits}-\ref{ss:families of left-tensored quasi-categories} that the 2-full subcategory $\Mon_{\mathscr O}^{\mathscr K}\Cat_\infty$ on objects classifying $\mathscr O$-monoidal quasi-categories compatible with $\mathscr K$-indexed colimits and edges classifying morphisms of $\mathscr O$-monoidal quasi-categories compatible with $\mathscr K$-indexed colimits by construction classifies coCartesian families of $\mathscr O$-monoidal quasi-categories compatible with $\mathscr K$-indexed colimits.

The alternative approach for the latter which we will now explain is to equip the full subcategory $\Cat_\infty(\mathscr K)\subseteq\Cat_\infty$ on quasi-categories with $\mathscr K$-indexed colimits with a symmetric monoidal structure so that the equivalence $\Alg_{\mathscr O}\Cat_\infty\to\Mon_{\mathscr O}\Cat_\infty$ restricts to an equivalence $\Alg_{\mathscr O}\Cat_\infty(\mathscr K)\to\Mon_{\mathscr O}^{\mathscr K}\Cat_\infty$.

First of all, consider the following explicit construction $\Cat_\infty^\otimes$ of the Cartesian symmetric monoidal structure on $\Cat_\infty$:
\begin{enumerate}
 \item An object over $\langle n\rangle\in\Fin_*$ is an $n$-tuple $[X_1,\dotsc,X_n]$ of quasi-categories.
 \item A morphism $[X_1,\dotsc,X_n]\to[Y_1,\dotsc,Y_m]$ over a morphism $\langle n\rangle\xrightarrow\alpha\langle m\rangle$ is a collection of morphisms of quasi-categories $\prod_{i:\alpha(i)=j}X_i\xrightarrow{\eta_j}Y_j$ for all $1\le j\le m$.
\end{enumerate}
Let $P$ be the collection of all small simplicial sets partially ordered by inclusion. Then we let $\mathscr M\subseteq\Cat_\infty^\otimes\times N(P)$ be the 2-full subcategory on
\begin{enumerate}
 \item pairs $([X_1,\dotsc,X_n],\mathscr K)$ such that $X_i$ has $\mathscr K$-indexed colimits for all $1\le i\le n$, and
 \item morphisms $([X_1,\dotsc,X_n],\mathscr K)\xrightarrow{\{\eta_j\}}([Y_1,\dotsc,Y_m],\mathscr K')$ such that $\prod_{i:\alpha(i)=j}X_i\xrightarrow{\eta_j}Y_j$ is compatible with $\mathscr K$-indexed colimits for all $1\le j\le m$.
\end{enumerate}
Then one can show that the map $\mathscr M\to N(\Fin_*)\times N(P)$ is a coCartesian fibration of symmetric monoidal quasi-categories \cite{lurie2014higher}*{Proposition 4.8.1.3}. For a collection of simplicial sets $\mathscr K$, we let $\Cat_\infty(\mathscr K)^\otimes=\mathscr M\varprod_{N(P)}\{\mathscr K\}$. Then one can show that the inclusion map $\Cat_\infty(\mathscr K)^\otimes\to\Cat_\infty^\otimes$ is a lax monoidal functor \cite{lurie2014higher}*{Proposition 4.8.1.4}. This implies that the equivalence $\Alg_{\mathscr O}\Cat_\infty\to\Mon_{\mathscr O}\Cat_\infty$ restricts to an equivalence $\Alg_{\mathscr O}\Cat_\infty(\mathscr K)\to\Mon_{\mathscr O}^{\mathscr K}\Cat_\infty$.

In particular, let $\mathscr K$ denote the collection of all small simplicial sets. Then we have the symmetric monoidal quasi-category of big quasi-categories compatible with all small colimits $\widehat{\Cat}_\infty(\mathscr K)^\otimes$. One can show that the full subcategory $\PrL\subseteq\widehat{\Cat}_\infty(\mathscr K)$ is closed under tensor product \cite{lurie2014higher}*{Proposition 4.8.1.15}. Therefore, the inclusion $\PrL\to\widehat{\Cat}_\infty(\mathscr K)$ can be lifted to a symmetric monoidal functor $(\PrL)^\otimes\to\widehat{\Cat}_{\infty}(\mathscr K)^\otimes$. This in turn gives the equivalence $\Alg_{\mathscr O}\PrL\to\Mon_{\mathscr O}\PrL$ we wanted.

\section{Homotopical Skolem--Noether theorem} Let $K$ be a quasi-category equipped with a Grothendieck topology $\tau$. Let $\mathscr C^\otimes\xrightarrow q\Assoc^\otimes\times K^\op$ be a coCartesian family of presentable monoidal quasi-categories. In Subsection 1 we introduce Azumaya algebra objects in $\mathscr C^\otimes$ and give a criterion on $\mathscr C^\otimes$ so that the prestack $\Az\mathscr C$ of Azumaya algebra objects in $\mathscr C^\otimes$ an the prestack $\LTens^{\Az}\mathscr C$ of locally trivial presentable left-tensored quasi-categories over $\mathscr C^\otimes$ have $\tau$-descent. In Subsection 2 we prove our main result.

\subsection{Descent for presentable left-tensored quasi-categories with descent}

\begin{notn}

Let $K$ denote a quasi-category equipped with a Grothendieck topology $\tau$ \cite{lurie2009higher}*{\S6.2.2}. 

Let $\mathscr C^\otimes\xrightarrow q\Assoc^\otimes\times K^\op$ be a coCartesian family of presentable monoidal quasi-categories. Then it is classified by a map $K^\op\xrightarrow{\mathbf c_{\mathscr C}}\Pr^{\Mon}$. In our main example, the site $K$ is the fppf site over a scheme, and $\mathscr C^\otimes$ is the family of monoidal quasi-categories of cochain complexes of quasi-coherent sheaves.

Let $\Alg\mathscr C=\Pr^{\Alg}\varprod_{\Pr^{\Mon}}K^\op$. In our main example, it classifies derived associative algebras over cochain complexes of quasi-coherent sheaves.

\end{notn}

\begin{defn}
    Take $X\in K$ and let $A\in\Alg\mathscr C(X)$ be an algebra object. Then \emph{$A$ is an Azumaya algebra object} if there exists
    \begin{enumerate}
        \item a $\tau$-covering $U\to X$,
        \item a dualizable generator $M\in\mathscr C_\dgen(U)$ and
        \item an equivalence $A|U:=\mathscr C(U)\otimes_{C(X)}A\simeq\End M$.
    \end{enumerate}
    We denote by $\Az\mathscr C\subseteq\Alg\mathscr C$ the full subprestack on Azumaya algebras.
\end{defn}

\begin{notn}

Let $\LTens\mathscr C=\Pr^{\Mod}\varprod_{\Pr^{\Mon}}K^\op$.  In our main example, it classifies presentable quasi-categories left-tensored over a monoidal quasi-category of cochain complexes of quasi-coherent sheaves. That is, it classifies lax dg-categories.

Let $\LTens_*\mathscr C=\Pr^{\Mod}_{\mathfrak M/}\varprod_{\Pr^{\Mon}}K^\op$. Let $\LTens_\dgen\mathscr C\subseteq\LTens_*\mathscr C$ denote the full subprestack on triples $(U,\mathscr M,M)$ of an object $U\in K$, a quasi-category $\mathscr M$ left-tensored over $\mathscr C(U)$, and an object $M\in\mathscr M$ such that $M$ is a dualizable generator in $\mathscr M$. Let $\LTens_\dgen^{\Az}\mathscr C\subseteq\LTens_\dgen\mathscr C$ denote the full subprestack on locally trivial pointed presentable left-tensored quasi-categories over $\mathscr C$.

\end{notn}

\begin{rem}
    By the Homotopical Morita Theorem, we have an equivalence $A|U\simeq\End M$ of algebras for some $M\in\mathscr C_\dgen(U)$ if and only if we have an equivalence $\Mod_A|U\simeq\mathscr C(U)$ of presentable quasi-categories left-tensored over $\mathscr C$. Therefore, the equivalence $\Alg\mathscr C\xrightarrow{\Mod_*}\LTens_\dgen\mathscr C$ restricts to an equivalence $\Az\mathscr C\to\LTens_\dgen^{\Az}\mathscr C$.
\end{rem}

Let ${}^\op\LTens\mathscr C\xrightarrow{q^\op}K$ denote the structure map. Let $U\xrightarrow fX$ be a $\tau$-covering in $K$. Let $\Delta_+^\op\xrightarrow{\Bar U_\bullet}K$ denote its \v Cech nerve \cite{lurie2009higher}*{below Proposition 6.1.2.11}. Then \emph{${}^\op\LTens\mathscr C$ has descent along $f$}, if the restriction map
$$
\Gamma_{\Cart}(\Bar U_\bullet,{}^\op\LTens\mathscr C)\xrightarrow{r_{\Cart}}\Gamma_{\Cart}(U_\bullet,{}^\op\LTens\mathscr C)
$$
is an equivalence of quasi-categories. This will not hold in general. But we can select a full subcategory ${}^\op\LTens^\desc\mathscr C\subseteq{}^\op\LTens\mathscr C$ of presentable quasi-categories left-tensored over $\mathscr C$ \emph{with $\tau$-descent}, and we can show that
\begin{enumerate}
 \item the restriction ${}^\op\LTens^\desc\mathscr C\to K$ is a biCartesian fibration which satisfies $\tau$-descent, and
 \item it has all the objects needed for Morita theory: for all $U\in K$ and $A\in\Alg\mathscr C_U$, the right module quasi-category $\Mod_A$ has $\tau$-descent.
\end{enumerate}
First of all, consider the restriction map on the section quasi-categories
$$
\Gamma(\Bar U_\bullet,{}^\op\LTens\mathscr C)\xrightarrow r\Gamma(U_\bullet,{}^\op\LTens\mathscr C).
$$
Suppose that every section $\Delta^\op\xrightarrow k{}^\op\LTens\mathscr C$ has a $q^\op$-colimit $\Delta_+^\op\xrightarrow{\Bar k}{}^\op\LTens\mathscr C$. Then $r$ has a fully faithful left adjoint, which takes $k$ to $\Bar k$ \cite{lurie2009higher}*{Corollary 4.3.2.16 and Proposition 4.3.2.17}.

\begin{prop}\label{prop:LTen has relative limits} Let $K$ be a quasi-category, $\mathscr C^\otimes\to K^\op\times\Assoc^\otimes$ a coCartesian family of presentable monoidal quasi-categories, and $L$ a simplicial set. Then the family $\LTens\mathscr C\xrightarrow qK^\op$ of presentable quasi-categories left-tensored over $\mathscr C$ has all $L$-indexed $q$-limits.

\end{prop}

\begin{cor}\label{cor:RAd of r} Over any diagram $L^\vartriangleleft\to K^\op$, the restriction map
$$
\Gamma(L^\vartriangleleft,\LTens\mathscr C)\xrightarrow r\Gamma(L,\LTens\mathscr C)
$$
admits a fully faithful right adjoint.

\end{cor}

\begin{proof}[Proof of Proposition \ref{prop:LTen has relative limits}]

It is enough to show \cite{lurie2009higher}*{Corollary 4.3.1.11} that
\begin{enumerate}
 \item for all $U\in K$, the fibre $\LTens\mathscr C(U)$ has $L$-indexed limits, and
 \item for all maps $U\to V$ in $K^\op$, the restriction map $\LTens\mathscr C(V)\xrightarrow{\mathscr M\mapsto{}_{\mathscr C(U)}\mathscr M}{}^\op\LTens\mathscr C(U)$ preserves $L$-indexed limits.
\end{enumerate}
We have $\LTens\mathscr C(U)\simeq\LMod_{\mathscr C(U)}\PrL$ \cite{lurie2009higher}*{Remark 4.8.3.6}. As $\PrL$ admits all limits \cite{lurie2009higher}*{Proposition 5.5.3.13}, the quasi-category $\LMod_{\mathscr C(U)}\PrL$ also has limits \cite{lurie2014higher}*{Corollary 4.2.3.3}, which shows (1).

Since $\Pr^{\Mod}\xrightarrow\psi\Pr^{\Mon}$ is biCartesian by Lemma \ref{lem:psi is biCartesian}, the restriction map $\LTens\mathscr C(V)\to\LTens\mathscr C(U)$ admits a left adjoint \cite{lurie2009higher}*{Corollary 5.2.2.5}, and thus it preserves limits \cite{lurie2009higher}*{Proposition 5.2.3.5}, which shows (2).

\end{proof}

\begin{lem}\label{lem:psi is biCartesian} The universal family $\Pr^{\Mod}\xrightarrow\psi\Pr^{\Mon}$ of presentable left-tensored quasi-categories is biCartesian.

\end{lem}

\begin{proof} Let $\mathscr L$ denote the collection of all small simplicial sets. Then the forgetful map $\widehat{\Cat}_\infty^{\Mod}(\mathscr L)\xrightarrow{\Hat\psi}\widehat{\Cat}_\infty^{\Alg}(\mathscr L)$ is both Cartesian \cite{lurie2014higher}*{Corollary 4.2.3.2} and coCartesian \cite{lurie2014higher}*{Proposition 4.8.5.1}. Take a morphism $\mathscr D\xrightarrow f\mathscr D'$ in $\Pr^{\Alg}$, $\mathscr M\in\LMod_{\mathscr D}\PrL$, and $\mathscr M'\in\LMod_{\mathscr D'}\PrL$. The $\Hat\psi$-Cartesian edge $\mathscr M'\to{}_{\mathscr D}\mathscr M'$ is an equivalence on the underlying quasi-categories \cite{lurie2009higher}*{Corollary 4.2.3.2}, therefore ${}_{\mathscr D}\mathscr M'$ has a presentable underlying quasi-category, and thus is an object of $\LMod_{\mathscr D}\PrL$. By Lemma \ref{lem:base change of presentable left-tensored quasi-category is presentable}, the base change $\mathscr D'\otimes_{\mathscr D}\mathscr M$ is also presentable.

\end{proof}

\begin{lem}\label{lem:base change of presentable left-tensored quasi-category is presentable} Let $\mathscr D^\otimes\to(\mathscr D')^\otimes$ be a morphism of presentable monoidal quasi-categories. Let $\mathscr M^\otimes$ be a presentable quasi-category left-tensored over $\mathscr D^\otimes$. Then the base change $\mathscr D'\otimes_{\mathscr D}\mathscr M$ is also presentable.

\end{lem}

\begin{proof} On the underlying quasi-categories, the $\Hat\psi$-coCartesian edge $\mathscr M\to\mathscr D'\otimes_{\mathscr D}\mathscr M$ can be given by the bar construction, which is a colimit diagram \cite{lurie2014higher}*{Theorem 4.4.2.8}. Since $\PrL$ has colimits and the inclusion $\PrL\to\widehat{\Cat}_\infty$ preserves colimits \cite{lurie2009higher}*{Theorem 5.5.3.18}, the base change $\mathscr D'\otimes_{\mathscr D}\mathscr M$ has presentable underlying quasi-category too.

\end{proof}

\begin{cor}\label{cor:r_coc has RAd} Over any diagram $L^\vartriangleleft\to K^\op$, the restriction map on the coCartesian section quasi-categories
$$
\Gamma_{\coCart}(L^\vartriangleleft,\LTens\mathscr C)\xrightarrow{r_{\coCart}}\Gamma_{\coCart}(L,\LTens\mathscr C).
$$
admits a fully faithful right adjoint.

\end{cor}

\begin{proof}

Let us denote $\Bar\Gamma_\coC=\Gamma_{\coCart}(L^\vartriangleleft,\LTens\mathscr C)$, $\Gamma_\coC=\Gamma_{\coCart}(L,\LTens\mathscr C)$, $\Bar\Gamma=\Gamma(L^\vartriangleleft,\LTens\mathscr C)$ and $\Gamma=\Gamma(L,\LTens\mathscr C)$. Then the inclusions $\Gamma_\coC\xrightarrow i\Gamma$ resp.~$\Bar\Gamma_\coC\xrightarrow{\Bar i}\Bar\Gamma$ admit right adjoints $\Gamma\xrightarrow R\Gamma_\coC$ resp.~$\Bar\Gamma\xrightarrow{\Bar R}\Bar\Gamma_\coC$ \cite{lurie2009higher}*{Lemma 5.5.3.16}. Let us denote the right adjoint to the restriction map by $\Gamma\xrightarrow{\lim}\Bar\Gamma$. We claim that the composite $\Gamma_\coC\subseteq\Gamma\xrightarrow\lim\Bar\Gamma\xrightarrow{\Bar R}$ is shown to be the right adjoint to $r_\coC$ by the composite $\mathscr M\xrightarrow u\Bar{\mathscr M}\xrightarrow{\Bar Ru}\Bar R\lim r\mathscr M$, where the two $u$ denote unit maps for the adjunctions $(\Bar i,\Bar R)$ resp.~$(r,\lim)$ applied to $\mathscr M\in\Bar\Gamma_\coC$. The claim can be checked on the commutative diagram
\begin{center}

\begin{tikzpicture}[xscale=5,yscale=2]
\node (0) at (0,0) {$\Map_{\Gamma_\coC}(r\mathscr M,\mathscr N)$};
\node (1) at (1,0) {$\Map_\Gamma(r\mathscr M,\mathscr N)$};
\node (2) at (2,0) {$\Map_{\Bar\Gamma}(\mathscr M,\lim\mathscr N)$};
\node (3) at (3,0) {$\Map_{\Bar\Gamma_\coC}(\mathscr M,\Bar R\lim\mathscr N)$};
\node (12) at (1.5,1) {$\Map_{\Bar\Gamma}(\lim r\mathscr M,\lim\mathscr N)$};
\node (23) at (2.5,1) {$\Map_{\Bar\Gamma_\coC}(\Bar R\mathscr M,\Bar R\lim\mathscr N)$};
\node (123) at (2,2) {$\Map_{\Bar\Gamma_\coC}(\Bar R\lim r\mathscr M,\Bar R\lim\mathscr N),$};
\path[->,font=\scriptsize,>=angle 90]
(0) edge node [above] {$i$} node [below] {$\simeq$} (1)
(1) edge node [above left] {$\lim$} (12)
(1) edge node [below] {$\simeq$} (2)
(12) edge node [above left] {$\Bar R$} (123)
(12) edge node [above right] {$\circ u$} (2)
(123) edge node [above right] {$\circ\Bar Ru$} (23)
(2) edge node [above left] {$\Bar R$} (23)
(2) edge node [below] {$\simeq$} (3)
(23) edge node [above right] {$\circ u$} (3);
\end{tikzpicture}

\end{center}
where $\mathscr N\in\Gamma_\coC$.

\end{proof}

\begin{notn} Recall that the restriction to $U_{-1}=X$-map $\Gamma_{\coCart}(\Bar U_\bullet^\op,\LTens\mathscr C)\to\LTens\mathscr C(X)$ is an equivalence of quasi-categories \cite{dhillon2018stack}*{Lemma 2.5}. Take $\mathscr M_\bullet\in\Gamma_{\coCart}(U_\bullet^\op,\LTens\mathscr C)$. We denote its image by the right adjoint by $\lim\mathscr M_\bullet\in\LTens\mathscr C(X)$.

If we start out from $\mathscr M\in\LTens\mathscr C(X)$, then we denote its image by the right adjoint by $\mathscr C(U_\bullet^\op)\otimes_{\mathscr C(X)}\mathscr M$. Note that we get a unit map
$$
\mathscr M\to\lim(\mathscr C(U_\bullet^\op)\otimes_{\mathscr C(X)}\mathscr M).
$$

\end{notn}

\begin{defn} Let $K$ be a quasi-category equipped with a Grothendieck topology $\tau$ and $\mathscr C^\otimes\to K^\op\times\Assoc^\otimes$ a coCartesian fibration of presentable monoidal quasi-categories. Let $\LTens\mathscr C\xrightarrow qK^\op$ be the corresponding family of presentable quasi-categories left-tensored over $\mathscr C^\otimes$. We say that \emph{base changes in $q$ commute with $\tau$-descent data} if for the \v Cech nerve $\Delta^{\op}\xrightarrow{U_\bullet}K$ of a $\tau$-covering, a $q$-limit diagram $\Bar{\mathscr M_\bullet}\in\Gamma_\coC(U_\bullet^\op,\LTens\mathscr C)$ and a map $V\to\Bar U_{-1}$ in $K$, the base change $\mathscr C(V)\otimes_{\mathscr C(\Bar U_{-1})}\Bar{\mathscr M_\bullet}$ is also a $q$-limit diagram.

We say that the family $\mathscr C^\otimes$ \emph{has $\tau$-descent} if the following conditions hold:
\begin{enumerate}
 \item The underlying Cartesian fibration ${}^\op\mathscr C\xrightarrow pK$ has $\tau$-descent.
 \item Base changes in $q$ commute with $\tau$-descent data.
\end{enumerate}

Let  $T\in K$ be an object and $\mathscr M\in\LTens\mathscr C(T)$ a presentable quasi-category left-tensored over $\mathscr C(T)$. We say that \emph{$\mathscr M$ has $\tau$-descent}, if for all $\tau$-coverings $U\xrightarrow fX$ over $T$ in $K$, letting $\Delta^\op_+\xrightarrow{\Bar U_\bullet}K$ denote the \v Cech nerve of $f$, the canonical map
$$
\mathscr C(X)\otimes_{\mathscr C(T)}\mathscr M\to\lim(\mathscr C(U_\bullet^\op)\otimes_{\mathscr C(X)}\mathscr M)
$$
is an equivalence. We denote by $\LTens^\desc\mathscr C\subseteq\LTens\mathscr C$ the full subcategory on pairs $(T,\mathscr M)$ of objects $T\in K$ and presentable quasi-categories $\mathscr M$ left-tensored over $\mathscr C(T)$ with $\tau$-descent.

\end{defn}

\begin{prop}\label{prop:dualizable commutes with descent} Let $K$ be a quasi-category equipped with a Grothendieck topology $\tau$ and $\mathscr C^\otimes\to K^\op\times\Assoc^\otimes$ a coCartesian fibration of presentable monoidal quasi-categories. Suppose that the underlying Cartesian fibration ${}^\op\mathscr C\xrightarrow pK$ has $\tau$-descent. Then the following assertions hold:
\begin{enumerate}
 \item Suppose that for all morphisms $V\to U$ in $K$, the bimodule $\mathscr C(V)\in{}_{\mathscr C(V)}\Mod_{\mathscr C(U)}$ is left dualizable. Then base changes in $q$ commute with $\tau$-descent data.
 \item Suppose that base changes in $q$ commute with $\tau$-descent data. Let $T\in K$ be an object and $\mathscr M\in\LTens\mathscr C(T)$ a presentable quasi-category left-tensored over $\mathscr C(T)$. Suppose that $\mathscr M\in\LTens\mathscr C(T)$ is left dualizable. Then $\mathscr M$ has $\tau$-descent.

\end{enumerate}

\end{prop}

\begin{proof} (1) Since ${}^\op\mathscr C\to K$ has $\tau$-descent and $\Bar U_\bullet$ is the \v Cech nerve of a $\tau$-covering, the diagram $\Delta_+\xrightarrow{\mathscr C^\otimes(\Bar U_\bullet)}\Pr^{\Mon}$ is a limit diagram. Therefore, the base change $\mathscr C(V)\otimes_{\mathscr C(\Bar U_{-1})}\Bar{\mathscr M_\bullet}$ is a $q$-limit diagram if and only if it is a limit diagram. Since $\mathscr C(V)\in{}_{\mathscr C(V)}\Mod_{\mathscr C(U)}$ is left dualizable, the functor $\LTens\mathscr C(U)\xrightarrow{\mathscr C(V)\otimes_{\mathscr C(U)}}\LTens\mathscr C(V)$ has a left adjoint and thus commutes with small limits \cite{lurie2014higher}*{Proposition 4.6.2.1}.

(2) Let $\Delta_+^\op\xrightarrow{\Bar U_\bullet}K$ be the \v Cech nerve of a $\tau$-covering over $T$. We need to show that the augmented simplicial diagram $\mathscr C(\Bar U_\bullet)\otimes_{\mathscr C(T)}\mathscr M$ is a $q$-limit. Since $p$ is a $\tau$-stack, the augmented simplicial diagram $\mathscr C(\Bar U_\bullet)$ of presentable monoidal quasi-categories is a limit diagram. Therefore, it is enough to show that $\mathscr C(\Bar U_\bullet)\otimes_{\mathscr C(T)}\mathscr M$ is a limit diagram. This follows from that $\mathscr M\in\LTens\mathscr C(T)$ is left dualizable.

\end{proof}

\begin{cor}\label{cor:dualizable commutes with descent} Let $K$ be a quasi-category equipped with a Grothendieck topology $\tau$ and $\mathscr C^\otimes\to K^\op\times\Assoc^\otimes$ a coCartesian fibration of presentable monoidal quasi-categories. Suppose that the following assertions hold:
\begin{enumerate}
 \item The underlying Cartesian fibration ${}^\op\mathscr C\xrightarrow pK$ has $\tau$-descent.
 \item For all objects $U\in K$, the presentable monoidal quasi-category $\mathscr C(U)\in\Alg(\PrL)$ has dualizable underlying presentable quasi-category.
\end{enumerate}
Then base changes in $q$ commute with $\tau$-descent data.

\end{cor}

\begin{proof} This follows from the more general result Lemma \ref{lem:proper is dualizable}

\end{proof}

\begin{lem}\label{lem:proper is dualizable} Let $A\in\Alg\mathscr C$ be an algebra object in a monoidal quasi-category. Suppose that the underlying object $A\in\mathscr C$ has left duality data
$$
\Check A\otimes A\xrightarrow e\boldsymbol 1,\quad \boldsymbol1\xrightarrow cA\otimes\Check A.
$$
Then the object $A\in{}_A\Mod\mathscr C$ has left duality data
$$
(\Check A\otimes A)\otimes_AA\xrightarrow\simeq\Check A\otimes A\xrightarrow e\boldsymbol1,\quad A\xrightarrow{c\otimes\id_A}A\otimes\Check A\otimes A.
$$

\end{lem}

\begin{proof} This can be checked directly.

\end{proof}

\begin{thm}\label{prop:LTen^desc is a stack} Let $K$ be a quasi-category equipped with a Grothendieck topology $\tau$. Let $\mathscr C^\otimes\to K^\op\times\Assoc^\otimes$ be a coCartesian family of presentable monoidal quasi-categories with $\tau$-descent. Then the following assertions hold.
\begin{enumerate}
 \item The family ${}^\op\LTens^\desc\mathscr C\to K$ of presentable quasi-categories left-tensored over $\mathscr C$ with $\tau$-descent is a $\tau$-stack.
 \item For any object $T\in K$ and associative algebra $A\in\Alg\mathscr C(T)$, the quasi-category $\Mod_A\mathscr C(T)$ left-tensored over $\mathscr C(T)$ of right $A$-modules in $\mathscr C(T)$ has $\tau$-descent.
\end{enumerate} 

\end{thm}

\begin{proof} (1) Let $\Delta_+^\op\xrightarrow{\Bar U_\bullet}K$ be the \v Cech nerve of a $\tau$-covering. We need to show that the restriction map
$$
\Gamma_\coC(\Bar U_\bullet^\op,\LTens^\desc\mathscr C)\xrightarrow{r^\desc}\Gamma_\coC(U_\bullet^\op,\LTens^\desc\mathscr C)
$$
is an equivalence. By Corollary \ref{cor:r_coc has RAd}, the restriction map on coCartesian sections
$$
\Gamma_\coC(\Bar U_\bullet^\op,\LTens\mathscr C)\xrightarrow{r_\coC}\Gamma_\coC(U_\bullet^\op,\LTens\mathscr C)
$$
has a fully faithful right adjoint $\lim_\coC$. By Lemma \ref{lem:limit of LTens desc has desc}, this map restricts to a fully faithful right adjoint of $r^\desc$. In other words, for any descent diagram $\mathscr M_\bullet\in\Gamma_\coC(U_\bullet^\op,\LTens^\desc\mathscr C)$, the counit map $r^\desc\lim_\coC\mathscr M_\bullet\to\mathscr M_\bullet$ is an equivalence. Let $\Bar{\mathscr N}_\bullet\in\Gamma_\coC(\Bar U_\bullet^\op,\LTens^\desc\mathscr C)$ be an effective descent diagram. Since $\Bar{\mathscr N}_{-1}$ is a presentable quasi-category left-tensored over $\mathscr C(\Bar U_{-1})$ with descent, by definition the unit map $\Bar{\mathscr N}_\bullet\to\lim_\coC r^\desc\Bar{\mathscr N}_\bullet$ is an equivalence too. This shows that $r^\desc$ is an equivalence of quasi-categories, as claimed.

(2) The presentable quasi-category $\Mod_A\mathscr C(T)$ left-tensored over $\mathscr C(T)$ has a left dual \cite{lurie2014higher}*{Remark 4.8.4.8}. Therefore it has $\tau$-descent by Proposition \ref{prop:dualizable commutes with descent} (2).

\end{proof}

\begin{lem}\label{lem:limit of LTens desc has desc} Let $\mathscr M_\bullet\in\Gamma_\coC(U_\bullet^\op,\LTens^\desc\mathscr C)$ be a descent diagram of presentable quasi-categories left-tensored over $\mathscr C$ with descent. Then its $q$-limit $\Bar{\mathscr M}\in\LTens\mathscr C(\Bar U_{-1})$ also has descent.

\end{lem}

\begin{proof} Let's consider the entire $q$-limit diagram $\Bar{\mathscr M}_\bullet\in\Gamma_\coC(\Bar U_\bullet^\op,\LTens\mathscr C)$. Let $X=\Bar U_{-1}$. Let $\Delta_+^\op\xrightarrow{\Bar V_\bullet}K$ be the \v Cech nerve of a $\tau$-covering with $\Bar V_{-1}=X$. We need to show that the unit map
$$
\Bar{\mathscr M}_{-1}\to\lim_{n\ge0}\mathscr C(V_n)\otimes_{\mathscr C(X)}\Bar{\mathscr M}_{-1}
$$
is an equivalence. For $m,n\ge-1$, let $\Bar W_{mn}=U_m\times_XV_n\in K$, $\Bar{\mathscr M}_{mn}=\mathscr C(\Bar W_{mn})\otimes_{\mathscr C(X)}\Bar{\mathscr M}_{-1}$, let
$$
\Bar{\mathscr M_{-1}}\xrightarrow{u^h_{-1}}\lim_{m\ge0}\Bar{\mathscr M}_{m,-1}
\text{ and }
\Bar{\mathscr M_{-1}}\xrightarrow{u^v_{-1}}\lim_{n\ge0}\Bar{\mathscr M}_{-1,n}
$$
denote unit maps, and let 
$$
u^h_n=\mathscr C(\Bar V_n)\otimes_{\mathscr C(X)}u^h_{-1}
\text{ and }
u^v_m=\mathscr C(\Bar U_m)\otimes_{\mathscr C(X)}u^v_{-1}.
$$
Then we have the commutative diagram
\begin{center}

\begin{tikzpicture}[xscale=4,yscale=2]
\node (A) at (0,1) {$\Bar{\mathscr M_{-1}}$};
\node (A') at (0,0) {$\lim_{n\ge0}\Bar{\mathscr M}_{-1,n}$};
\node (B) at (1,1) {$\lim_{m\ge0}\Bar{\mathscr M}_{m,-1}$};
\node (B') at (1,0) {$\lim_{m,n\ge0}\Bar{\mathscr M}_{mn}$.};
\path[->,font=\scriptsize,>=angle 90]
(A) edge node [above] {$u^h_{-1}$} (B)
(A) edge node [right] {$u^v_{-1}$} (A')
(A') edge node [above] {$\lim_{n\ge0}u^h_n$} (B')
(B) edge node [right] {$\lim_{m\ge0}u^v_m$} (B');
\end{tikzpicture}

\end{center}

\end{proof}
Since $\Bar{\mathscr M}_\bullet$ is a limit diagram, $u^h_{-1}$ is an equivalence. Since $\Bar{\mathscr M}_{m,-1}=\mathscr M_m$ has descent for $m\ge0$, $\lim_{m\ge0}u^v_m$ is an equivalence. Since $u^h_{-1}$ is an equivalence, for $n\ge0$, as base changes in $q$ commute with $\tau$-descent data, the map $u^h_n$ is also an equivalence. Therefore, $\lim_{n\ge0}u^h_n$ is an equivalence. All this implies that $u^v_{-1}$ is an equivalence, which is what we needed to prove.

\begin{cor}\label{cor:descent for pointed} The co-family ${}^\op\LTens^\desc_*(\mathscr C)$ of pointed left-tensored quasi-categories with descent satisfies $\tau$-descent too.
 
\end{cor}

\begin{proof} During this proof, we will use the alternative join \cite{lurie2009higher}*{\S4.2.1} to work with the classifying object $(\Pr^{\Mod})^{\mathfrak M/}$ and thus the subobject $\LTens^*_\desc\mathscr C$ on pointed left-tensored quasi-categories with descent of its pullback $\LTens^*\mathscr C$ along $K^\op\xrightarrow{\mathbf c_{\mathscr C}}\Pr^{\Mon}$. Let $\Delta^+\xrightarrow{U_\bullet^+}K$ be the \v Cech nerve of a $\tau$-covering. We claim that the restriction map
$$
\Gamma_{\Cart}(U_\bullet^+,{}^\op\LTens^*_\desc\mathscr C)\xrightarrow{r_*}\Gamma_{\Cart}(U_\bullet,{}^\op\LTens^*_\desc\mathscr C)
$$
is a trivial fibration. Let $K''\subset K'$ be an inclusion of simplicial sets. The lifting problem
\begin{center}

\begin{tikzpicture}[xscale=4,yscale=2]
\node (A) at (0,1) {$K''$};
\node at (0,0.5) {$\cap$};
\node (A') at (0,0) {$K'$};
\node (B) at (1,1) {$\Gamma_{\Cart}(U_\bullet^+,{}^\op\LTens^*_\desc\mathscr C)$};
\node (B') at (1,0) {$\Gamma_{\Cart}(U_\bullet,{}^\op\LTens^*_\desc\mathscr C)$};
\path[->,font=\scriptsize,>=angle 90]
(A) edge (B)
(A') edge (B')
(A') edge [dashed] (B)
(B) edge node [right] {$r_*$} (B');
\end{tikzpicture}

\end{center}
is equivalent to the lifting problem
\begin{center}

\begin{tikzpicture}[xscale=5,yscale=2]
\node (A) at (0,1) {$(K''\times\Delta^1)\cup(K'\times\Delta^{\{0\}})$};
\node at (0,0.5) {$\cap$};
\node (A') at (0,0) {$K'\times\Delta^1$};
\node (B) at (1,1) {$\Gamma_{\Cart}(U_\bullet^+,{}^\op\LTens^\desc\mathscr C)$};
\node (B') at (1,0) {$\Gamma_{\Cart}(U_\bullet,{}^\op\LTens^\desc\mathscr C),$};
\path[->,font=\scriptsize,>=angle 90]
(A) edge (B)
(A') edge (B')
(A') edge [dashed] (B)
(B) edge node [right] {$r$} (B');
\end{tikzpicture}

\end{center}
which has a solution as ${}^\op\LMod\mathscr C$ has $\tau$-descent, and thus $r$ is a trivial fibration. This proves the claim.
 
\end{proof}

\begin{cor}\label{cor:descent for dgen} The co-family ${}^\op\LTens^\desc_\dgen\mathscr C$ of presentable left-tensored quasi-categories with descent pointed by dualizable generators satisfies $\tau$-descent too.

\end{cor}

\begin{proof} By Corollary \ref{cor:descent for pointed}, the restriction map
$$
\Gamma_\coC(\Bar U_\bullet^\op,\LTens^\desc_*\mathscr C)
\to\Gamma_\coC(U_\bullet^\op,\LTens^\desc_*\mathscr C)
$$
is a trivial Kan fibration. Therefore, it is enough to show that if for an effective descent datum $(\Bar{\mathscr M}_\bullet,\Bar M_\bullet)\in\Gamma_\coC(\Bar U_\bullet^\op,\LTens_*^\desc\mathscr C)$, for all $n\ge0$ the object $M_n\in\mathscr M_n$ is a dualizable generator, then the limit $\Bar M_{-1}\in\Bar{\mathscr M_{-1}}$ is a dualizable generator too. We have a counit map
$$
(\mathscr C,\Mod_{\End M}\mathscr C,\End M)\xrightarrow{v_{(\mathscr C,\mathscr M,M)}=\otimes_{\End M}M}(\mathscr C,\mathscr M,M)
$$
natural in $(\mathscr C,\mathscr M,M)\in\Pr^{\Mod}_{/\mathfrak M}$ \cite{lurie2014higher}*{Theorem 4.8.5.11} giving an edge $v_{(\Bar{\mathscr C}_\bullet,\Bar{\mathscr M}_\bullet,\Bar M_\bullet)}$ in the coCartesian section quasi-category $\Gamma_\coC(\Bar U_\bullet^\op,\LTens^\desc_*\mathscr C)$. We claim that the edge $v_{(\Bar{\mathscr C}_\bullet,\Bar{\mathscr M}_\bullet,\Bar M_\bullet)}$ is actually in $\Gamma_\coC(\Bar U_\bullet^\op,\LTens^\desc_*\mathscr C)$. This follows from Theorem \ref{prop:LTen^desc is a stack} (2) and Lemma \ref{lem:base change of End}. This shows that the map
$$
(\mathscr C(\Bar U_{-1}),\Mod_{\End\Bar M_{-1}}\mathscr C(\Bar U_{-1})
\xrightarrow{v_{(\mathscr C(\Bar U_{-1}),\Bar{\mathscr M}_{-1},\Bar M_{-1})}=\otimes_{\End\Bar M_{-1}}\Bar M_{-1}}
(\mathscr C(\Bar U_{-1}),\Bar{\mathscr M}_{-1},\Bar M_{-1})
$$ is the limit of the maps
$$
(\mathscr C(U_n),\Mod_{\End M_n}\mathscr C(U_n),\End M_n)\xrightarrow{v_{(\mathscr C(U_n),\mathscr M_n,M_n)}=\otimes_{\End M_n}\mathscr C(U_n)}(\mathscr C(U_n),\mathscr M_n,M_n).
$$
Therefore, since the maps $v_{(\mathscr C(U_n),\mathscr M_n,M_n)}$ are equivalences for all $n\ge0$, so is their limit $v_{(\mathscr C(\Bar U_{-1}),\Bar{\mathscr M}_{-1},\Bar M_{-1})}$, as we needed to show.

\end{proof}

\begin{lem}\label{lem:base change of End} Let $\mathscr C^\otimes\xrightarrow f(\mathscr C')^\otimes$ be a morphism of presentable monoidal quasi-categories. Let $\mathscr M^\otimes$ be a presentable quasi-category left-tensored over $\mathscr C^\otimes$. Let $M\in\mathscr M$ be an object. Then the base change
$$
(\mathscr C'\otimes_{\mathscr C}\Mod_{\End M}\mathscr C,\boldsymbol1_{\mathscr C'}\otimes\End M)
\xrightarrow{\mathscr C'\otimes_{\mathscr C}(-\otimes_{\End M}M)}
(\mathscr C'\otimes_{\mathscr C}\mathscr M,\boldsymbol1_{\mathscr C'}\otimes M)
$$
exhibits $\boldsymbol1_{\mathscr C'}\otimes\End M\in\Alg\mathscr C'$ as an endomorphism algebra of $\boldsymbol 1_{\mathscr C'}\otimes M\in\mathscr C'\otimes_{\mathscr C}\mathscr M$.

\end{lem}

\begin{proof} Take a coCartesian edge
$$
\Mod_{\End M}\mathscr C\to\mathscr C'\otimes_{\mathscr C}\Mod_{\End M}\mathscr C
$$
in $\Pr^{\Mod}$ over $f$. Postcomposing it with the natural equivalence $\mathscr C'\otimes_{\mathscr C}\Mod_{\End M}\mathscr C\to\Mod_{\End M}(\mathscr C'_{\mathscr C})$ \cite{lurie2014higher}*{Theorem 4.8.4.6}, we get a coCartesian edge in $\Pr^{\Mod}$ over $f$ of the form
$$
\Mod_{\End M}\mathscr C\to\Mod_{\End M}(\mathscr C'_{\mathscr C}).
$$
Since the Morita functor $\Pr^{\Alg}\xrightarrow\Theta\Pr^{\Mod}$ takes the coCartesian edge $\End M\to\boldsymbol1_{\mathscr C'}$ in $\Pr^{\Alg}$ over $f$ to the coCartesian edge
$$
\Mod_{\End M}\mathscr C\to\Mod_{\boldsymbol1_{\mathscr C'}\otimes\End M}(\mathscr C')
$$
in $\Pr^{\Mod}$ over $f$ \cite{lurie2014higher}*{Proposition 4.8.5.1}, we get an equivalence
$$
\mathscr C'\otimes_{\mathscr C}\Mod_{\End M}\mathscr C\simeq\Mod_{\boldsymbol1_{\mathscr C'}\otimes\End M}\mathscr C'.
$$
Let
$$
\End M\otimes M\to M
$$
be a map in $\mathscr M$ exhibiting $\End M\in\mathscr C$ as an endomorphism object of $M$. Then the coCartesian edge $\mathscr M\to\mathscr C'\otimes_{\mathscr C}\mathscr M$ in $\Pr^{\Mod}$ over $f$ takes it to a map
$$
(\boldsymbol1_{\mathscr C'}\otimes\End M)\otimes(\boldsymbol1_{\mathscr C'}\otimes M)\to(\boldsymbol1_{\mathscr C'}\otimes M)
$$
in $\mathscr C'\otimes_{\mathscr C}\mathscr M$ exhibiting $\boldsymbol1_{\mathscr C'}\otimes\End M\in\mathscr C'$ as an endomorphism object of $\boldsymbol1_{\mathscr C'}\otimes M\in\mathscr C'\otimes_{\mathscr C}\mathscr M$, giving an equivalence
$$
\Mod_{\boldsymbol1_{\mathscr C'}\otimes\End M}\mathscr C'\simeq
\Mod_{\End(\boldsymbol1_{\mathscr C'}\otimes M)}\mathscr C'
$$
and thus concluding the proof.

\end{proof}

\subsection{Homotopical Skolem--Noether Theorem}

\begin{thm}[Homotopical Skolem--Noether Theorem]\label{thm:HSN} Let $K$ be a quasi-category with final object $S$, let $\tau$ be a Grothendieck topology on it. Let $\Cart_{/K}^\tau$ denote the quasi-category of Cartesian fibrations over $K$ with $\tau$-descent. Let $\mathscr X\subseteq\Cart_{/K}^\tau$ denote the full subcategory on right fibrations over $K$ with $\tau$-descent, which is an $\infty$-topos. Let $\mathscr C^\otimes\to\Assoc^\otimes\times K^\op$ be a family of presentable monoidal quasi-categories with $\tau$-descent. Then the following assertions hold:

(1) We have a fibre sequence in $(\Cart_{/K}^\tau)_*$:
$$
({}^\op\mathscr C_\dgen^\simeq,\mathscr O)\xrightarrow{\End}
({}^\op\Az\mathscr C,\mathscr O)\xrightarrow{\Mod}
({}^\op\LTens^{\Az}\mathscr C,\Mod_{\mathscr O})
$$
(2) Let $E\in\mathscr C_\dgen(S)$. Then we have a fibre sequence in $(\Cart_{/K}^\tau)_*$:
$$
({}^\op\Pic\mathscr C,\mathscr O)\xrightarrow{\otimes E}
({}^\op\mathscr C_\dgen^\simeq,E)\xrightarrow{\End}
({}^\op\Az\mathscr C,\End E).
$$
(3) We have a long exact sequence of homotopy sheaves in $\mathrm h\,\mathscr X$:
\begin{multline*}
\dotsb\to\pi_2({}^\op\Az\mathscr C,\End E)
\to\pi_1({}^\op\Pic\mathscr C,\mathscr O)
\to\pi_1({}^\op\mathscr C_\dgen,E)
\to\pi_1({}^\op\Az\mathscr C,\End E)\to\\
\to\pi_0({}^\op\Pic\mathscr C)
\to\pi_0({}^\op\mathscr C_\dgen)
\to\pi_0({}^\op\Az\mathscr C)
\to\pi_0({}^\op\LTens^{\Az}\mathscr C)
\to0.
\end{multline*}
 
\end{thm}

\begin{proof} I) Consider the homotopy commutative diagram
\begin{center}

\begin{tikzpicture}[xscale=4,yscale=2]
\node (A) at (0,1) {$\mathscr C_\dgen^\simeq$};
\node (A') at (0,0) {$K$};
\node at (1,1.4) {$\Uparrow_{v_*}$};
\node at (0.5,0.5) {$\Swarrow_v$};
\node (B) at (1,1) {$\Az\mathscr C$};
\node (B') at (1,0) {$\LTens^{\Az}\mathscr C$};
\node at (1.5,0) {$=$};
\node (C) at (2,1) {$\LTens^{\Az}_\dgen\mathscr C$};
\node (C') at (2,0) {$\LTens^{\Az}\mathscr C$};
\node at (2.5,1) {$\subseteq$};
\node at (2.5,0) {$\subseteq$};
\node (D) at (3,1) {$\LTens_\dgen\mathscr C$};
\node (D') at (3,0) {$\LTens\mathscr C$};
\path[->,font=\scriptsize,>=angle 90]
(A) edge [bend left=60] node [above] {$i_*$} (C)
(A) edge node [above] {$\End$} (B)
(A) edge (A')
(A') edge node [above] {$\mathbf c_{\mathscr C}$} (B')
(B) edge node [above] {$\Mod_*$} (C)
(B) edge node [right] {$\Mod$} (B')
(C) edge (C')
(D) edge (D');
\end{tikzpicture}

\end{center}
where
\begin{itemize}
 \item $i_*$ is the inclusion $C\mapsto(\mathscr C,C)$,
 \item $\Mod_*(A)=(\Mod_A,A)$,
 \item $\mathbf c_{\mathscr C}$ is a Cartesian section with $\mathbf c_{\mathscr C}(S)=\mathscr C_S$, 
 \item $v(C)$ is the equivalence $\Mod_{\End C}\xrightarrow{\otimes_{\End C}C}\mathscr C$ as in Corollary \ref{cor:dgen equiv}, and
 \item $v_*(C)$ is its pointed version $(\Mod_{\End C},\End C)\to(\mathscr C,C)$.
\end{itemize}
The forgetful map $\LTens^\desc_*\mathscr C\to\LTens^\desc\mathscr C$ is a left fibration, thus so is its restriction to connected components $\LTens^{\Az}_\dgen\mathscr C\to\LTens^{\Az}\mathscr C$. Therefore, the square under $i_*$, which is strict Cartesian by construction, is homotopy Cartesian. We claim that $\Az\mathscr C\xrightarrow{\Mod_*}\LTens^{\Az}_\dgen\mathscr C$ is an equivalence. This will imply that the square under $\Mod_*$ is homotopy Cartesian, and thus by the pasting lemma, the square under $\End$ is homotopy Cartesian as well, as we needed to show.

By definition of dualizable generators and Theorem \ref{prop:LTen^desc is a stack} (2), we know that $\Alg\mathscr C\xrightarrow{\Mod_*}\LTens^\desc_\dgen$ is an equivalence. Therefore, it is enough to show that, for any $U\in K$, an algebra $A\in\Alg\mathscr C(U)$ is Azumaya if and only if $\Mod_A\in\LTens^\desc\mathscr C(U)$ is locally trivial. This in turn follows from Lemma \ref{lem:Az alg mod}.

II) Consider now the diagram
\begin{center}

\begin{tikzpicture}[xscale=4,yscale=2]
\node (A) at (0,1) {$\Pic\mathscr C$};
\node (A') at (0,0) {$\mathscr C_\dgen^\simeq$};
\node (B) at (1,1) {$K$};
\node (B') at (1,0) {$\Az\mathscr C,$};
\path[->,font=\scriptsize,>=angle 90]
(A) edge (B)
(A) edge node [right] {$\otimes E$} (A')
(B) edge node [right] {$\mathbf c_{\End E}$} (B')
(A') edge node [above] {$\End$} (B');
\end{tikzpicture}

\end{center}
where $\mathbf c_{\End E}$ is a Cartesian section such that $\mathbf c_{\End E}(S)=\End E$. We claim that the diagram is homotopy Cartesian. By the pasting lemma, and that by the homotopical Eilenberg--Watts theorem, we have $\Omega(\LTens\mathscr C,\mathscr C)=\Pic\mathscr C$, it is enough to show that the square is homotopy commutative. For that, as the map $\Az\mathscr C\xrightarrow{\Mod_*}\LTens^\desc_\dgen\mathscr C$ is an equivalence, it will be enough to show that the diagram
\begin{center}

\begin{tikzpicture}[xscale=4,yscale=2]
\node (A) at (0,1) {$\Pic\mathscr C$};
\node (A') at (0,0) {$\mathscr C_\dgen^\simeq$};
\node (B) at (1,1) {$K$};
\node (B') at (1,0) {$\LTens^\desc_\dgen\mathscr C$};
\node at (0.5,0) {$\subset$};
\path[->,font=\scriptsize,>=angle 90]
(A) edge (B)
(A) edge node [right] {$\otimes E$} (A')
(B) edge node [right] {$\mathbf c_{(\mathscr C,E)}$} (B');
\end{tikzpicture}

\end{center}
is homotopy commutative. That is, we need to supply a map
$$
\Pic\mathscr C\times\Delta^1\xrightarrow f\LTens_\dgen\mathscr C
$$
such that $f|(\Pic\mathscr C\times\Delta^{\{0\}})$ is the constant map $c$ with value $(\mathscr C,E)$, and $f|(\Pic\mathscr C\times\Delta^{\{1\}})$ is the map $L\mapsto(\mathscr C,L\otimes E)$. The homotopical Eilenberg--Watts theorem supplies the map $C\xrightarrow{F\mapsto\otimes F}\Fun_{\mathscr C}(\mathscr C,\mathscr C)$. This gives the map $\mathrm{EW}$ in the lifting problem
\begin{center}

\begin{tikzpicture}[xscale=4,yscale=2]
\node (A) at (0,1) {$\Pic\mathscr C\times\Delta^{\{0\}}$};
\node at (0,0.5) {$\cap$};
\node (A') at (0,0) {$\Pic\mathscr C\times\Delta^1$};
\node (AB) at (0.5,1.5) {$K$};
\node (AB') at (0.5,-0.5) {$\mathscr C\times\Delta^1,$};
\node (B) at (1,1) {$\LTens^\desc_*\mathscr C$};
\node (B') at (1,0) {$\LTens^\desc\mathscr C$};
\path[->,font=\scriptsize,>=angle 90]
(A) edge (B)
(A) edge (AB)
(AB) edge node [above right] {$\mathscr c_{(\mathscr C,E)}$} (B)
(B) edge node [right] {$r$} (B')
(A') edge [dashed] node [below right] {$f$} (B)
(A') edge (B')
(A') edge node [below left] {$\otimes E$} (AB')
(AB') edge node [below right] {$\mathrm{EW}$} (B');
\end{tikzpicture}

\end{center}
where the restriction map $r$ is a left fibration, thus we have a solution $f$, which as $E$ is a dualizable generator, maps into $\LTens^\desc_\dgen\mathscr C$, as required.

III) From the fibre sequences, we readily get a long exact sequence
\begin{multline*}
\dotsb\to\pi_2({}^\op\Az\mathscr C,\End E)
\to\pi_1({}^\op\Pic\mathscr C,\mathscr O)
\to\pi_1({}^\op\mathscr C_\dgen,E)
\to\pi_1({}^\op\Az\mathscr C,\End E)\to\\
\to\pi_0({}^\op\Pic\mathscr C)
\to\pi_0({}^\op\mathscr C_\dgen)
\to\pi_0({}^\op\Az\mathscr C)
\to\pi_0({}^\op\LTens^{\Az}\mathscr C).
\end{multline*}
Since we have $\Az\mathscr C\simeq\LTens^{\Az}_\dgen\mathscr C$, by definition of dualizable generators, we get that the map of homotopy sheaves $\pi_0({}^\op\Az\mathscr C)\to\pi_0({}^\op\LTens^{\Az}\mathscr C)$ is surjective.
 
\end{proof}

\begin{lem}\label{lem:Az alg mod} Let $\mathscr C$ be a monoidal quasi-category, and $A\in\Alg\mathscr C$. Then there exists $E\in\mathscr C_\dgen$ and $A\simeq\End E$ if and only if there exists $\Mod_A\simeq\mathscr C$ in $\LTens\mathscr C$.

\end{lem}

\begin{proof} $\Rightarrow$: Suppose that there exists $E\in\mathscr C_\dgen$ and an equivalence $A\xrightarrow{\phi}\End E$ in $\Alg\mathscr C$. Then we get equivalences in $\LTens\mathscr C$
$$
\Mod_A\xrightarrow{\Mod(\phi)}\Mod_{\End E}\xrightarrow{\otimes_{\End E}E}\mathscr C
$$
as needed.

$\Leftarrow$: Suppose that there exists an equivalence $\Mod_A\xrightarrow\phi\mathscr C$ in $\LTens\mathscr C$. By the homotopical Eilenberg--Watts theorem, $\phi$ is of the form $\otimes_AE$ for some $E\in{}_A\Mod$. It will be enough to show that the action map $A\otimes E\xrightarrow\alpha E$ exhibits $A$ as an endomorphism object of $E\in\mathscr C$. Let $C\in\mathscr C$ and consider the diagram
\begin{center}

\begin{tikzpicture}[xscale=4,yscale=2]
\node (A) at (0,1) {$\Map_{\mathscr C}(C,A)$};
\node (AB) at (1,0) {$\Map_A(C\otimes A,A\otimes A)$};
\node (B) at (1,1) {$\Map_{\mathscr C}(C\otimes E,A\otimes E)$};
\node (BC) at (2,0) {$\Map_A(C\otimes A,A),$};
\node (C) at (2,1) {$\Map_{\mathscr C}(C\otimes E,E)$};
\path[->,font=\scriptsize,>=angle 90]
(A) edge node [above] {$\otimes E$} (B)
(A) edge node [below left] {$\otimes A$} (AB)
(AB) edge node [right] {$\otimes_AE$} (B)
(AB) edge node [above] {$\mu\circ$} (BC)
(B) edge node [above] {$\alpha\circ$} (C)
(BC) edge node [right] {$\otimes_AE$} (C);
\end{tikzpicture}

\end{center}
where $A\otimes A\xrightarrow\mu A$ is the multiplication map. We know that the left triangle is homotopy commutative. We claim that the right square is homotopy commutative. That will conclude the proof as $(\mu\circ)\circ(\otimes A)$ is an equivalence as $\mu$ induces the counit of the induction $\smalladjoints{\mathscr C}{\Mod_A}{\otimes A}{\otimes_AA}$, and $\otimes_AE$ is an equivalence by assumption.

The claim follows from that the diagram
\begin{center}

\begin{tikzpicture}[xscale=4,yscale=2]
\node (A) at (0,1) {${}_A\Mod_A$};
\node (A') at (0,0) {${}_A\Mod$};
\node (B) at (1,1) {$\Mod_A$};
\node (B') at (1,0) {$\mathscr C,$};
\path[->,font=\scriptsize,>=angle 90]
(A) edge node [above] {$\otimes_AE$} (B)
(A) edge (A')
(B) edge (B')
(A') edge node [above] {$\otimes_AE$} (B');
\end{tikzpicture}

\end{center}
where the vertical maps are restriction maps, is commutative, and therefore the functor $\otimes_AE$ takes $A\in{}_A\Mod_A$ to $E\in{}_A\Mod$.

\end{proof}

\section{Applications}

In this section, we list a number of interesting families of monoidal quasi-categories we can apply the Homotopical Skolem--Noether Theorem to. In Subsection 1 we recall the the symmetric monoidal quasi-category of stable presentable quasi-categories \cite{lurie2014higher}*{\S4.8.2} as all the presentable monoidal quasi-categories in our examples are stable. In Subsection 2 we recall the monoidal structure on the underlying quasi-category of a monoidal model category \cite{lurie2014higher}*{\S4.1.7} as that will serve as our main device of getting examples. In Subsection 3 we apply our main result to Algebraic Geometry. This includes proving 1-descent for pre-generalized Azumaya algebras over a quasi-compact and quasi-separated scheme. In Subsection 4 we apply our main result to Derived and Spectral Algebraic Geometry. In Subsection 5 we apply our main result to ind-coherent sheaves and crystals.

\subsection{The symmetric monoidal quasi-category of stable presentable quasi-categories}

\subsubsection{Stable quasi-categories}

\begin{defn} Let $\mathscr C$ be a quasi-category. Then an object $0\in\mathscr C$ is a \emph{zero object} if it is both an initial object and a final object. We say that \emph{the quasi-category $\mathscr C$ is pointed} if it has a zero object. We say that \emph{the quasi-category $\mathscr C$ is stable} if the following assertions hold:
\begin{enumerate}
 \item The quasi-category $\mathscr C$ is pointed.
 \item The quasi-category $\mathscr C$ has finite limits and colimits.
 \item A square
\begin{center}

\begin{tikzpicture}[xscale=2,yscale=2]
\node (A) at (0,1) {$C'$};
\node (A') at (0,0) {$D'$};
\node (B) at (1,1) {$C$};
\node (B') at (1,0) {$D$};
\path[->,font=\scriptsize,>=angle 90]
(A) edge (B)
(A) edge (A')
(B) edge (B')
(A') edge (B');
\end{tikzpicture}

\end{center}
in $\mathscr C$ is a pushout diagram if and only if it is a pullback diagram.

\end{enumerate}

\end{defn}

\begin{defn} Let $\mathscr C$ be a stable quasi-category. For a nonnegative integer $n\ge0$ we let $\mathscr C\xrightarrow{C\mapsto C[n]}\mathscr C$ denote the $n$-th power of the suspension functor $\mathscr C\xrightarrow\Sigma\mathscr C$. For a nonpositive integer $n\le0$ we let $\mathscr C\xrightarrow{C\mapsto C[n]}\mathscr C$ denote the $(-n)$-th power of the loop object functor $\mathscr C\xrightarrow\Omega\mathscr C$. We refer to these functors as \emph{translation functors}.

Consider a diagram
$$
C\xrightarrow{\Bar f}D\xrightarrow{\Bar g}E\xrightarrow{\Bar h}C[1]
$$
in the homotopy category $\Ho\mathscr C$. Then we say that it is a \emph{distinguished triangle} if there exists a diagram
\begin{center}

\begin{tikzpicture}[xscale=2,yscale=2]
\node (A) at (0,1) {$C$};
\node (A') at (0,0) {$0'$};
\node (B) at (1,1) {$D$};
\node (B') at (1,0) {$E$};
\node (C) at (2,1) {$0$};
\node (C') at (2,0) {$F$};
\path[->,font=\scriptsize,>=angle 90]
(A) edge node [above] {$f$} (B)
(A) edge (A')
(B) edge (C)
(B) edge node [right] {$g$} (B')
(C) edge (C')
(A') edge (B')
(B') edge node [above] {$h$} (C');
\end{tikzpicture}

\end{center}
in $\mathscr C$ such that the following assertions hold:
\begin{enumerate}
 \item The squares are pushout diagrams in $\mathscr C$.
 \item The object $0,0'\in\mathscr C$ are zero objects.
 \item The map $f$ represents $\Bar f$ and the map $g$ represents $\Bar g$.
 \item Since $F$ is a suspension of $C$, there exists an equivalence $F\xrightarrow{h'}C[1]$ well-defined up to homotopy. Then the composite $h'\circ h$ represents $\Bar h$.
\end{enumerate}

\end{defn}

\begin{thm}\label{thm:stable gives triangulated}\cite{lurie2014higher}*{Theorem 1.1.2.14} Let $\mathscr C$ be a stable quasi-category. Then the homotopy category $\Ho\mathscr C$ can be endowed with a triangulated category structure as follows:
\begin{enumerate}
 \item Since $\mathscr C$ is pointed, for objects $C,D\in\mathscr C$ the zero map $C\to0\to D$ is a natural base point for the mapping space $\Map(C,D)$.
 \item The isomorphism $\pi_0\Map(C,D)\cong\pi_2(\Sigma^2C,D)$ equips the set $\Hom_{\Ho\mathscr C}(C,D)\cong\pi_0\Map_{\mathscr C}(C,D)$ with an abelian group structure.
 \item We have the translation functors and distinguished triangles defined above.
\end{enumerate}

\end{thm}

\begin{defn} Let $\mathscr C$ be a stable quasi-category. Then we say that \emph{$\mathscr C$ is compactly generated} if the triangulated category $\Ho\mathscr C$ is compactly generated.

\end{defn}

\subsubsection{Spectrum objects and the quasi-category of spectra}

\begin{defn} Let $\mathscr C\xrightarrow F\mathscr D$ be a functor between quasi-categories.
\begin{enumerate}
 \item Suppose that $\mathscr C$ has pushouts. Then we say that \emph{$F$ is excisive} if it takes pushout diagrams to pullback diagrams.
 \item Suppose that $\mathscr C$ has a final object $*\in\mathscr C$. Then we say that \emph{$F$ is reduced} if $F(*)\in\mathscr D$ is a final object.
\end{enumerate}
Suppose that $\mathscr C$ has pushouts and a final object. Then we denote by $\Exc_*(\mathscr C,\mathscr D)\subseteq\Fun(\mathscr C,\mathscr D)$ the full subcategory on excisive reduced functors.

Let $\mathscr S^\fin\subseteq\mathscr S$ denote the smallest full subcategory that contains the final object $*\in\mathscr S$ and it is closed under small colimits. Let $\mathscr S^\fin_*\subseteq\mathscr S_*$ denote the full subcategory on pointed objects of $\mathscr S^\fin$. Let $\mathscr C$ be a quasi-category that has finite limits. Then the \emph{quasi-category of spectrum objects of $\mathscr C$} is $\Sp(\mathscr C)=\Exc_*(\mathscr S^\fin,\mathscr C)$. The \emph{quasi-category of spectra} is $\Sp=\Sp(\mathscr S)$. We let $\Sp(\mathscr C)\xrightarrow{\Omega^\infty}\mathscr C$ denote the functor given by substitution at $S^0\in\mathscr S^\fin_*$.

\end{defn}

\begin{prop}\label{prop:Sigma infinity}\cite{lurie2014higher}*{Proposition 1.4.4.4 and Corollary 1.4.4.5} Let $\mathscr C$ be a presentable quasi-category. Then the following assertions hold:
\begin{enumerate}
 \item The quasi-category $\Sp(\mathscr C)$ of spectrum objects of $\mathscr C$ is presentable.
 \item The functor $\Sp(\mathscr C)\xrightarrow{\Omega^\infty}\mathscr C$ has a left adjoint $\Sigma^\infty_+$.
 \item Let $\mathscr D$ be a presentable stable quasi-category. Then the precomposition map
 $$
 \LFun(\Sp(\mathscr C),\mathscr D)\xrightarrow{\circ\Sigma^\infty_+}\LFun(\mathscr C,\mathscr D)
 $$
 is an equivalence.
\end{enumerate}

\end{prop}

\begin{defn} The \emph{sphere spectrum} is $S=\Sigma^\infty_+(*)\in\Sp$. 

\end{defn}

\subsubsection{Idempotent objects and the symmetric monoidal structure on the quasi-category $\PrSt$ of presentable stable quasi-categories}

\begin{defn} Let $\mathscr C$ be a monoidal quasi-category. Then an \emph{idempotent object in $\mathscr C$} is a morphism $\boldsymbol1\xrightarrow cC$ such that the maps in $\mathscr C$:
$$
C\simeq C\otimes\boldsymbol1\xrightarrow{\id_C\otimes c}C\otimes C\text{ and }C\simeq\boldsymbol1\otimes C\xrightarrow{c\otimes\id_C}C\otimes C
$$
are equivalences.

\end{defn}

\begin{prop}\label{prop:idempotent object}\cite{lurie2014higher}*{Proposition 4.8.2.4} Let $\mathscr C$ be a monoidal quasi-category and $\boldsymbol1\xrightarrow cC$. Then the following assertions are equivalent:
\begin{enumerate}
 \item The map $c$ is an idempotent object of $\mathscr C$.
 \item Consider the endofunctor $\mathscr C\xrightarrow{C\otimes}\mathscr C$. Then the natural transformation $\id_C\to(C\otimes)$ induced by $c$ exhibits $(C\otimes)$ as a localization functor.
\end{enumerate}

\end{prop}

\begin{prop}\cite{lurie2014higher}*{Proposition 4.8.2.7} Let $\mathscr C$ be a symmetric monoidal quasi-category. Let $\boldsymbol1\xrightarrow cC$ be an idempotent object of $\mathscr C$. Let $L$ denote the endofunctor $\mathscr C\xrightarrow{C\otimes}\mathscr C$. Let $L\mathscr C^\otimes\subseteq\mathscr C^\otimes$ be the full subcategory on objects of the form $C_1\oplus\dotsb\oplus C_n$ where each $C_i$ is in $L\mathscr C$. Then the composite $L\mathscr C^\otimes\to\mathscr C^\otimes\to\Fin_*$ of the inclusion map and the structure map gives $L\mathscr C^\otimes$ a symmetric monoidal structure.

\end{prop}

\begin{defn} Let $A\in\CAlg\mathscr C$ be a commutative algebra object in a symmetric monoidal quasi-category. Then it is \emph{idempotent} if the multiplication map $A\otimes A\to A$ is an equivalence. Let $\CAlg^\idem\mathscr C\subseteq\CAlg\mathscr C$ denote the full subcategory on idempotent commutative algebra objects.

\end{defn}

\begin{prop}\cite{lurie2014higher}*{Proposition 4.8.2.9}\label{prop:idempotent object and commutative algebra} Let $\mathscr C^\otimes$ be a symmetric monoidal quasi-category. Then the composite of canonical maps
$$
\CAlg^\idem\mathscr C\to\CAlg\mathscr C\simeq\CAlg(\mathscr C)_{\boldsymbol1/}\to\mathscr C_{\boldsymbol1/}
$$
is fully faithful with essential image the full subcategory on idempotent objects.

\end{prop}

\begin{defn} Let $(\mathscr C,C)$ be a pair of a presentable quasi-category $\mathscr C$ and an object $C\in\mathscr C$. We say that \emph{the pair $(\mathscr C,C)$ is idempotent} if the colimit-preserving map $\mathscr S\xrightarrow F\mathscr C$ such that $F(*)=C$ is an idempotent object of $\PrL$. In this case, Proposition \ref{prop:idempotent object and commutative algebra} equips $\mathscr C$ with a symmetric monoidal structure with $C$ the unit object.

\end{defn}

\begin{prop}\label{prop:Sp tensor}\cite{lurie2014higher}*{Example 4.8.1.23} Let $\mathscr C\in\PrL$ be a presentable quasi-category. Then we have an equivalence of presentable quasi-categories $\Sp\otimes\mathscr C\simeq\Sp(\mathscr C)$.

\end{prop}

\begin{cor}\label{cor:PrSt} The pair $(\Sp,S)$ is idempotent.

\end{cor}

\begin{proof} By Proposition \ref{prop:idempotent object} it is enough to show that the endofunctor $\PrL\xrightarrow{L(\mathscr C)=\Sp\otimes\mathscr C}\PrL$ is a localization functor. By the Proposition it is equivalent to the endofunctor $\PrL\xrightarrow{L'(\mathscr C)=\Sp(\mathscr C)}\PrL$. By Proposition \ref{prop:Sigma infinity} (3) the map $\mathscr S\xrightarrow{\Sigma^\infty_+}\Sp$ induces a natural transformation $\id_{\PrL}\to L'$ which exhibits $L'$ as a localization functor.
 
\end{proof}

\begin{prop}\label{prop:stable iff Sp}\cite{lurie2014higher}*{1.4.2.21} Let $\mathscr C$ be a quasi-category. Then it is stable if and only if the map $\Sp\mathscr C\xrightarrow{\Omega^\infty}\mathscr C$ is an equivalence.

\end{prop}

\begin{defn} The \emph{quasi-category of presentable stable quasi-categories} is the full subcategory $\PrSt\subseteq\PrL$ on presentable stable quasi-categories. By Propositions \ref{prop:Sp tensor} and \ref{prop:stable iff Sp} it is the essential image of the localization functor $\PrL\xrightarrow{L(\mathscr C)=\Sp\otimes\mathscr C}\PrL$. By Corollary \ref{cor:PrSt} it admits a symmetric model structure such that the following assertions hold:
\begin{enumerate}
 \item The quasi-category $\Sp$ of spectra is a unit object of the symmetric monoidal quasi-category $(\PrSt)^\otimes$.
 \item The natural inclusion map $(\PrSt)^\otimes\to(\PrL)^\otimes$ is symmetric monoidal.
\end{enumerate}
By Proposition \ref{prop:idempotent object and commutative algebra} the quasi-category $\Sp$ of spectra admits a symmetric monoidal structure for which the sphere spectrum $S$ is a unit object. We refer to the tensor product as the \emph{smash product of spectra}.

\end{defn}

\begin{prop}\label{prop:compactly generated is dualizable}\cite{lurie2018spectral}*{Proposition D.7.2.3} Let $\mathscr C$ be a compactly generated stable quasi-category. Then it is a dualizable object of the symmetric monoidal quasi-category $\PrSt$.

\end{prop}

\begin{rem} Since the natural inclusion map $\PrSt\to\PrL$ is symmetric monoidal, if $\mathscr C\in\PrSt$ is dualizable, then it is also a dualizable object of $\PrL$.

\end{rem}

\subsection{The underlying quasi-category of a monoidal model category}

Let $X$ be a quasi-category. A \emph{system} in $X$ is a collection $W\subseteq X_1$ of morphisms which contains all equivalences, and it is stable under homotopy and composition. The collection of all systems on $X$ forms a poset $\Sys X$. Therefore, we get a map $\Cat_\infty^\op\xrightarrow{X\mapsto N\Sys X}\Cat_\infty$, which classifies a Cartesian fibration $\WCat_\infty\xrightarrow q\Cat_\infty$.
The objects of the quasi-category $\WCat_\infty$ are pairs $(X,W)$ where $X$ is a quasi-category, and $W\in\Sys X$ is a system on $X$. A mapping space $\Map_{\WCat_\infty}((X,W),(X',W'))$ can be identified with the full subcategory of the mapping space $\Map_{\Cat_\infty}(X,X')$ on functors $X\xrightarrow fX'$ such that $f(W)\subseteq f(W')$.

The Cartesian fibration $q$ is the forgetful map $(X,W)\mapsto X$. It admits a section $\Cat_\infty\xrightarrow G\WCat_\infty$ sending a quasi-category $X$ to the pair $(X,W)$, where $W\subseteq X_1$ is the collection of weak equivalences in $X$. The functor $G$ admits a left adjoint $\WCat_\infty\xrightarrow{F(C,W)=C[W^{-1}]}\Cat_\infty$, which moreover commutes with finite products \cite{lurie2014higher}*{Proposition 4.1.3.2}. We get the following universal property.

\begin{prop} Let $(X,W)\in\WCat_\infty$ and let $(X,W)\xrightarrow uGX[W^{-1}]$ denote the unit map. Then for any quasi-category $Y$, the precomposition by $u$ map
$$
\Fun(X[W^{-1}],Y)\xrightarrow{\circ u}\Fun(X,Y)
$$
is fully faithful with essential image the collection of functors $X\xrightarrow fY$ such that $f$ takes all morphisms in $W$ to equivalences in $Y$.

\end{prop}

\begin{proof} Since the right adjoint $G$ is fully faithful, the precomposition by $u$ map
$$
\Map_{\Cat_\infty}(X[W^{-1}],Y)\xrightarrow{\circ u}\Map_{\Cat_\infty}(X,Y)
$$
is fully faithful with essential image the collection of functors $X\xrightarrow fY$ such that $f$ takes all morphisms in $W$ to equivalences in $Y$. Since for a quasi-category $Z$ the mapping space $\Map_{\Cat_\infty}(Z,Y)$ is equivalent to the largest Kan complex in the functor quasi-category $\Fun(Z,Y)$, we get the statement about the essential image. Fully faithfulness follows from \cite{lurie2009higher}*{Proposition 3.1.3.3}.
 
\end{proof}

Let $\mathbf A$ be a model category. Then the collection $W$ of weak equivalences between cofibrant objects gives a system in the nerve $N\mathbf A^c$ of the full subcategory on cofibrant objects. The \emph{underlying quasi-category of $\mathbf A$} is the localization $N\mathbf A^c[W^{-1}]$.

Suppose that the model structure on $\mathbf A$ is monoidal. Then there is an induced model structure on the full subcategory $\mathbf A^c$ on cofibrant objects. We get a monoidal structure on the nerve $N\mathbf A^c$, which in turn is classified by a monoid object in $M'\in\Mon_{\Assoc}\Cat_\infty$. Since the model structure on $\mathbf A$ is monoidal, the weak equivalences between cofibrant objects are preserved by tensor product, therefore the monoid object can be lifted to $M\in\Mon_{\Assoc}\WCat_\infty$. Since the left adjoint $F$ preserves finite products, it preserves monoid objects. Therefore, the composite $FM$ is a monoid object $FM\in\Mon_{\Assoc}\Cat_\infty$. This equips the underlying quasi-category $N\mathbf A^c[W^{-1}]$ with a monoidal structure. We refer to this as the \emph{underlying monoidal quasi-category of the monoidal model category $\mathbf A$}. It comes equipped with a monoidal functor $N(\mathbf A^c)^\otimes\xrightarrow uN\mathbf A^c[W^{-1}]^\otimes$ which satisfies the following universal property: for every monoidal quasi-category $\mathscr D^\otimes$, the precomposition by $u$ functor
$$
\Fun^\otimes(N\mathbf A^c[W^{-1}]^\otimes,\mathscr D^\otimes)\xrightarrow{\circ u}\Fun^\otimes(N(\mathbf A^c)^\otimes,\mathscr D^\otimes)
$$
is fully faithful with essential image the full subcategory on monoidal functors $N(\mathbf A^c)^\otimes\xrightarrow{f^\otimes}\mathscr D^\otimes$ which take morphisms in $W$ to equivalences in $\mathscr D$ \cite{lurie2014higher}*{Proposition 4.1.7.4}.

\begin{cons}[The underlying presheaf of monoidal quasi-categories of a presheaf of monoidal categories with a system, and its extension via gluing] \label{cons:underlying presheaf} (1) Let $K$ be a category, and let 
$$
K\xrightarrow{k\mapsto(\mathbf C(k)^\otimes,W(k))}(\Mon_{\Assoc}\WCat_\infty)^\op
$$
be a presheaf of monoidal categories with systems. Postcomposition with the opposite of the underlying quasi-category functor $\Mon_{\Assoc}\WCat_\infty\xrightarrow F\Mon_{\Assoc}\Cat_\infty$ yields a presheaf of monoidal quasi-categories 
$$
K\xrightarrow{k\mapsto\mathbf C(k)^\otimes[W(k)^{-1}]}(\Mon_{\Assoc}\Cat_\infty)^\op=(\Cat_\infty^{\Mon})^\op.
$$

(2) This functor can be extended to a colimit-preserving functor $\mathscr P(K)\xrightarrow{\mathscr C^\otimes}(\Cat_\infty^{\Mon})^\op$ in a way that is unique up to homotopy \cite{lurie2009higher}*{Theorem 5.1.5.6}. This means the following. Take a presheaf $X\in\mathscr P(K)$. Then we have $X\simeq\hocolim_{i\in I} h_{k_i}$ for some diagram $I\xrightarrow{k_i}K$ over some small simplicial set $I$ \cite{lurie2009higher}*{Corollary 5.1.5.8}. The presheaf $\mathscr C^\otimes$ satisfies
$$
\mathscr C(X)^\otimes=\holim_{i\in I}\mathscr C(k_i)^\otimes
$$
in $\Cat_\infty^{\Mon}$.

\end{cons}

\subsection{Algebraic Geometry}

\begin{cons} Let $S$ be a scheme. We will now apply Construction \ref{cons:underlying presheaf} to construct the coCartesian family $\QC_S^\otimes$ of the monoidal quasi-categories of unbounded complexes of quasicoherent sheaves on the opposite of the quasi-category $\St_S$ of $\infty$-stacks on the big fppf site $S_\fppf$. 

The quasi-category $\St_S$ is the localization at fppf-local equivalences of the quasi-category $\mathscr P(\Aff_S)$ of presheaves of spaces on the category $\Aff_S$ on affine $S$-schemes. For an affine $S$-scheme $\Spec A=T\in\Aff_S$, the category $\mathbf C(A)$ of unbounded complexes of $A$-modules can be equipped by the projective model structure, which is a combinatorial and monoidal model structure \cite{lurie2014higher}*{Propositions 7.1.2.8 and 7.1.2.11}. Moreover, for a morphism of $S$-algebras $A\to B$ and a quasi-isomorphism $f$ of dg-projective complexes of $A$-modules, the pullback $f\otimes_AB$ is the derived pullback, and therefore it is also a quasi-isomorphism. This shows that we get a functor
$$
(\Aff_S)^\op\xrightarrow{A\mapsto[\mathbf C(A)_\dgproj^\otimes,\qis],\,(A\to B)\mapsto\otimes_AB}\Mon_{\Assoc}\WCat_\infty.
$$
Since the projective model structure is combinatorial and monoidal, applying Construction \ref{cons:underlying presheaf} we get a functor $\mathscr P(\Aff_S)\xrightarrow{\QC_S^\otimes}\Pr^{\Mon}$ satisfying the following properties.

(1) Let $T=\Spec A$ be an affine $S$-scheme. Then we have $\QC_S(A)^\otimes\simeq\mathbf C(A)_\dgproj[\qis^{-1}]^\otimes$.

(2) Let $T\in\St_S$ be an $\infty$-stack over $S$. Then it is some homotopy colimit $T=\hocolim T_i$ of affine $S$-schemes. We have $\QC_S(T)^\otimes=\holim\QC_S(T_i)^\otimes$ in the quasi-category $\Cat_\infty^{\Mon}$ of monoidal quasi-categories.

\end{cons}

\begin{prop}\label{prop:C for AG} Let $S$ be a quasi-compact and quasi-separated scheme. Then the coCartesian family of presentable monoidal quasi-categories $\QC_S^\otimes\xrightarrow p\St_S^\op\times\Assoc^\otimes$ has fppf descent.

\end{prop}

\begin{proof} The Cartesian fibration of unbounded complexes of quasi-coherent sheaves is compactly generated \cite{bondal2003generators}*{Theorem 3.1.1} satisfies fppf descent \cite{lurie2011spectral}*{Corollary 6.13}, \cite{lurie2014higher}*{Theorem 7.1.2.13}. Therefore, Proposition \ref{prop:compactly generated is dualizable} and Corollary \ref{cor:dualizable commutes with descent} show that $p$ has fppf descent.

\end{proof}

\begin{cor}[Homotopical Skolem--Noether Theorem for schemes]\label{cor:Skolem--Noether for schemes} Let $S$ be a quasi-compact and quasi-separated scheme. Let $\Cart_S^\fppf$ denote the quasi-category of Cartesian fibrations on $\St_S$ which satisfy fppf descent.

(1) Let ${}^\op\TPerf_S:={}^\op(\QC_S)_\dgen$ denote the Cartesian fibration of totally supported perfect complexes on $S$, ${}^\op\Deraz_S:={}^\op\Az\QC_S$ the Cartesian fibration of derived Azumaya algebras on $S$ and ${}^\op\Dg^{\Az}_S:={}^\op\LTens^{\Az}\QC_S$ the Cartesian fibration of locally trivial presentable quasi-categories left-tensored over $\QC_S^\otimes$. Then the sequence in $(\Cart_S^\fppf)_*$:
$$
({}^\op\TPerf_S,\mathscr O)\xrightarrow{\End}({}^\op\Deraz_S,\mathscr O)\xrightarrow{\Mod}({}^\op\Dg^{\Az}_S,\mathscr D)
$$
is a homotopy fibre sequence.

(2) Let $E\in\TPerf(S)$ be a totally supported perfect complex on $S$. Then the sequence in $(\Cart_S^\fppf)_*$:
$$
(\B\mathbf G_m\times\mathbf Z,\mathscr O)\xrightarrow{\otimes E}({}^\op\TPerf_S,E)\xrightarrow{\End}({}^\op\Deraz_S,\End E)
$$
is a homotopy fibre sequence.

(3) We have isomorphisms of sheaves of groups
$$
\pi_i\Omega({}^\op\TPerf_S,E)\cong\pi_i\Omega({}^\op {\Deraz_S},\REnd E)
$$
for $i>0$, a short exact sequence of sheaves of groups
$$
1\to\mathbf G_m\xrightarrow{a\mapsto a\cdot}\Aut_{\Perf}E\xrightarrow{\Ad}\Aut_{\Deraz}(\REnd E)\to1,
$$
and an exact sequence of pointed sheaves of sets
$$
*\to\pi_0(\B\mathbf G_m\times\mathbf Z)\xrightarrow{\otimes E}\pi_0\TPerf_S\xrightarrow{\REnd}\pi_0\Deraz_S\xrightarrow{\Mod}\pi_0\Dg^{\Az}_S\to*.
$$

\end{cor}

\begin{proof} By Proposition \ref{prop:C for AG}, the family $\QC_S^\otimes$ satisfies the assumptions of Theorem \ref{thm:HSN}. Let $T$ be an $S$-scheme and $E\in\Perf(T)$ a perfect complex on $T$. Then $T$ is a generator if and only if it is totally supported \cite{thomason1997classification}*{Lemma 3.14}, \cite{lurie2014higher}*{Corollary 1.4.4.2}. Since being a dualizable generator is a local property, this shows that the full subcategories $\TPerf_S$ and $(\QC_S)_\dgen$ of $\QC_S$ agree. Moreover, $E$ is an invertible element of $\QC(T)$ if and only if $E$ is of the form $\mathscr L[n]$ for an invertible sheaf $\mathscr L$ and an integer $n\in\mathbf Z$. This shows that the inclusion $\B\mathbf G_m\times\mathbf Z\to{}^\op\Pic\QC_S$ is an equivalence. We have $\pi_i(\B\mathbf G_m\times\mathbf Z)=0$ for $i>1$, and if $E$ is totally supported, then the map $\mathbf G_m\xrightarrow{a\mapsto a\cdot}\Aut_{\Perf}E$ is injective. Finally, the map $\pi_0(\B\mathbf G_m\times\mathbf Z)\xrightarrow{\otimes E}\pi_0\TPerf_S$ is injective, because the sheaf $\pi_0\B\mathbf G_m$ is trivial, and for an integer $n\in\mathbf Z$, if we have $E\simeq E[n]$, then we get $n=0$. 

\end{proof}

\begin{rem} Let $E=\mathscr O_S^{\oplus n}$. Then $E$ is totally supported. The short exact sequence
$$
1\to\mathbf G_m\to\Aut E\to\Aut\End E\to1
$$
is the one in the classical Skolem--Noether Theorem \cite{giraud1971cohomologie}*{V, Lemme 4.1}:
$$
1\to\mathbf G_m\to\GL_n\to\PGL_n\to1.
$$
 
\end{rem}

\begin{rem} Let's show how this result implies Lieblich's Derived Skolem--Noether Theorem \cite{lieblich2009compactified}*{Theorem 5.1.5}. Let $T$ be an $S$-scheme, and $E,F$ two nonzero perfect complexes on $T$. The \emph{annihilator $\Ann E$} of $E$ is the kernel of the scalar multiplication map $\mathscr O_T\to\End E$. It is the ideal sheaf of the support of $E$. Note that this shows $\Supp E=\Supp\REnd E$. We let $\mathscr O_E=\mathscr O_T/\Ann E$. 

We need to show that there exists a unique integer $n\in\mathbf Z$ such that the map of sheaves
$$
\pi_0\Isom_{\Perf(T)}(E[n],F)\to\Isom_{\Deraz(T)}(\REnd E,\REnd F)
$$
is surjective, with each fibre being an $\mathscr O_E^\times$-torsor, which in case $E=F$ is split.

I) Suppose first that $E$ is totally supported. Take a zigzag of weak equivalences of algebras $\phi:\REnd E\simeq\REnd F$. Then it determines a $T$-point of the homotopy fibre product and thus we get a dashed arrow in the following diagram:
\begin{center}

\begin{tikzpicture}[scale=2]
\node (T) at (-1, 2) {$T$};
\node (A) at (0,1) {$\B\mathbf G_m\times\mathbf Z$};
\node (A') at (0,0) {$T$};
\node (B) at (1,1) {$\TPerf_T$};
\node (B') at (1,0) {$\Deraz_T.$};
\node at (.5,.5) {$\lrcorner^{\mathrm h}$};
\path[->,font=\scriptsize,>=angle 90]
(T) edge [bend left] node [above right] {$F$} (B)
(T) edge [bend right] (A')
(T) edge [dashed] node [above right] {$\mathscr L[n]$} (A)
(A) edge node [above] {$\otimes E$} (B)
(A) edge (A')
(B) edge node [right] {$\REnd$} (B')
(A') edge node [above] {$\REnd E$} (B');
\end{tikzpicture}

\end{center}
That is, the equivalence of algebras $\phi$ is homotopical to the image by $\REnd$ of an equivalence of perfect complexes $E\otimes\mathscr L[n]\simeq F$. We get a preimage of the required form if we restrict to a trivializing cover of the invertible sheaf $\mathscr L$. The unicity of $n$ follows from the injectivity of the map
$$
\B\mathbf G_m\times\mathbf Z\to\pi_0\TPerf_T.
$$
The rest of the statement follows from the short exact sequence
$$
1\to\mathbf G_m\to\pi_0\Aut_{\Perf_T}(E)\to\pi_0\Aut_{\Deraz_T}(\REnd E)\to1.
$$

II) In the general case, we can push forward from $\Supp E=\Supp\REnd E=\Supp\REnd F=\Supp F$.

\end{rem}

\begin{app}\label{app:generalized Azumaya}
Let $X\xrightarrow fS$ be a proper and smooth morphism of algebraic spaces. In \cite{lieblich2009compactified}, Lieblich compactifies the stack $f_*B\PGL_n$ of families of principal $\PGL_n$-bundles the following way. Using the version of the Skolem--Noether theorem
$$
1\to\boldsymbol\mu_n\to\SL_n\to\PGL_n\to1,
$$
we get that the natural map $B\SL_n\fatslash\boldsymbol\mu_n\to B\PGL_n$ is an equivalence. Here, $B\SL_n\fatslash\boldsymbol\mu_n$ is the \emph{rigidification}, that is the target of the universal morphism $B\SL_n\to B\SL_n\fatslash\boldsymbol\mu_n$ which is invariant with respect to the $\boldsymbol\mu_n$-action on $B\SL_n$ given by scalar multiplication \cite{abramovich2003twisted}*{\S5.1}. Let $\mathscr T^{\mathscr O}_{X/S}(n)$ denote the stack of totally supported sheaves with trivialized determinant and rank $n$ at every maximal point. Then one can show that the stack $f_*(\mathscr T^{\mathscr O}_{X/S}(n)\fatslash\boldsymbol\mu_n)$ is a quasi-proper Artin stack \cite{lieblich2009compactified}*{Lemma 4.2.2} such that the natural map $f_*B\PGL_n\to f_*(\mathscr T^{\mathscr O}_{X/S}(n)\fatslash\boldsymbol\mu_n)$ is an open immersion \cite{lieblich2009compactified}*{Lemma 4.2.3}.

To give another description of the objects classified by $f_*(\mathscr T^{\mathscr O}_{X/S}(n)\fatslash\boldsymbol\mu_n)$, Lieblich introduces the notion of \emph{pre-generalized Azumaya algebras}. These are perfect algebra objects $A$ of the derived category $D(X)=\Ho\mathscr D(X)$ such that there exists a covering $U\to X$ and a totally supported perfect sheaf $F$ on $U$ such that $A|U\simeq\REnd(F)$. Then he considers the category fibred in groupoids $\mathscr{PR}$ of pre-generalized algebras, where the isomorphisms are the weak algebra isomorphisms of $\Alg D(X)$. Since working in this truncated setting he can't keep track of all the higher descent data, he needs to make the stack of \emph{generalized Azumaya algebras} $\mathscr G$ the stackification of $\mathscr{PR}$. Therefore, although he  can show that the objects of $\mathscr G$ are the weak algebras of the form $\mathbf R\pi_*\REnd(F)$ where $\mathscr X\xrightarrow\pi X$ is a $\mathbf G_m$-gerbe and $F$ is a totally supported perfect $\mathscr X$-twisted sheaf \cite{lieblich2009compactified}*{Proposition 5.2.1.12}, he can only give a somewhat implicit description of the isomorphisms in $\mathscr G$ \cite{lieblich2009compactified}. The derived Skolem--Noether theorem he proves \cite{lieblich2009compactified}*{Theorem 5.1.5} can be viewed as the 1-truncation of our result. It implies that the natural map $\mathscr T^{\mathscr O}_X(n)\xrightarrow{\REnd}\mathscr G_X$ induces an equivalence $\mathscr T^{\mathscr O}_X(n)\simeq\mathscr G_X$.

He shows that stackification is not needed in case $X\xrightarrow fS$ is a smooth projective relative surface \cite{lieblich2009compactified}*{Proposition 6.4.1}. Our result implies that this holds in general.
\end{app}

\begin{prop}\label{prop:pre-generalized Azumaya is stack}
Let $X$ be a quasi-compact and quasi-separated scheme. Then the category fibred in groupoids of pre-generalized Azumaya algebras $\mathscr{PR}_X$ is a 1-stack.
\end{prop}

\begin{proof}
Let $\mathscr{PR}^\infty_X\subset\Alg\Perf_X$ denote the full substack of pre-generalized Azumaya algebras. We claim that $\mathscr{PR}^\infty_X$ is a 1-stack. This will show that $\mathscr{PR}^\infty_X\simeq\mathscr{PR}_X$ is a 1-stack.

Let $A$ be a pre-generalized Azumaya algebra on $X$. We need to show that $B\Aut A\subset\mathscr{PR}_X$ is 1-truncated. Letting $\mathscr X=\mathscr X(A)\in B^2\mathbf G_m(X)$ denote the gerbe of trivializations of $A$, there exists an $\mathscr X$-twisted totally supported perfect sheaf $F$ such that $A\simeq\REnd(F)$ \cite{lieblich2009compactified}*{Lemma 5.2.1.1}. Consider the sequence of canonical maps
$$
B\mathbf G_m\to B\Aut F\xrightarrow{\REnd}B\Aut A.
$$
Since being a fibration sequence of pointed $\infty$-stacks is local, and the twisted sheaf $F$ is locally isomorphic to a sheaf, our result implies that this is a fibration sequence. Therefore, we have an exact sequence
$$
\pi_2B\Aut F\to\pi_2B\Aut A\to\pi_1B\mathbf G_m\to\pi_1B\Aut F.
$$
As $B\Aut F$ is a 1-stack, we have $\pi_2B\Aut F=0$. Moreover, the map $\pi_1B\mathbf G_m\to\pi_1B\Aut F$ is the scalar action $\mathbf G_m\to\Aut F$, which is injective. These two facts imply $\pi_2B\Aut A=0$. For $i>2$, we have an exact sequence
$$
0=\pi_iB\Aut F\to\pi_iB\Aut A\to\pi_iB\mathbf G_m=0,
$$
thus $\pi_iB\Aut A=0$. The claim is proven.

\end{proof}

\subsection{Homotopical Algebraic Geometry}

\subsubsection{Derived homotopical algebraic context}

\begin{defn}

Let $k$ be a commutative ring. Then the category $\bfMod_k^\Delta$ of simplicial $k$-modules admits a symmetric monoidal model structure as follows \cite{goerss2009simplicial}*{II, Example 6.2}:
\begin{enumerate}
 \item Tensor product is defined levelwise: we have $(A\otimes_kB)_n=A_n\otimes_kB_n$.
 \item The forgetful functor $\bfMod_k^\Delta\xrightarrow U\Set_\Delta$ is a right Quillen adjoint where $\Set_\Delta$ is equipped with the Quillen model structure.
\end{enumerate}
We shall call this the \emph{Quillen model structure on simplicial $k$-modules}. We get an induced model structure on the category $\bfCAlg_k^\Delta$ of simplicial commutative $k$-algebras.

The \emph{quasi-category $\CAlg_k^\Delta$ of simplicial commutative $k$-algebras} is the localization $\bfCAlg_k^\Delta[\weq^{-1}]$ of the category of simplicial commutative $k$-algebras at the system of weak equivalences. The \emph{quasi-category $\DAff_k$ of derived affine schemes over $k$} is the opposite quasi-category $(\CAlg_k^\Delta)^\op$.

\end{defn}

Let $A$ be a commutative dg $k$-algebra. Then the category $\bfMod_A^\dg$ of dg $A$-modules admits a symmetric monoidal model structure as follows \cite{barthel2014six}*{Theorem 3.3}:
\begin{enumerate}
 \item Tensor product is induced by the tensor product on complexes.
 \item A morphism of complexes is a weak equivalence if and only if it is a quasi-isomorphism.
 \item A morphism of complexes is a fibration if and only if it is a degree-wise surjection.

\end{enumerate}
We shall call this the \emph{projective model structure on dg $A$-modules}. A dg $k$-module $M\in\bfMod_k^\dg$ is cofibrant if and only if it is \emph{dg-projective} \cite{hovey2002cotorsion}*{Example 3.3}, that is:
\begin{enumerate}
 \item The $k$-modules $M_n$ are projective and
 \item For all exact complexes $E\in\bfMod_k^\dg$, the Hom complex $\BHom(M,E)$ is also exact.
\end{enumerate}
A bounded below complex of projective $k$-modules is dg-projective \cite{hovey1999model}*{Lemma 2.3.6}.

Let $M$ be a simplicial $k$-module. Then its \emph{Moore complex} is the dg $k$-module
$$
(CM)_n=M_n,\quad d=\sum_{i=0}^n(-1)^id_i.
$$
We denote by $DM\le CM$ the subcomplex of degenerate simplices. The \emph{normalized Moore complex} is the dg $k$-module
$$
NM=CM/DM.
$$
Let $N$ be another simplicial $k$-module. Then the \emph{shuffle map}
$$
CM\otimes_kCN\xrightarrow{\nabla}C(M\otimes_kN)
$$
takes $m\otimes n\in CM_p\otimes_kNC_q$ to
$$
\nabla(m\otimes n)=\sum_{(\mu,\nu)}\sign(\mu,\nu)(s_\nu m)\otimes(s_\mu n)
$$
where the summation is over \emph{$(p,q)$-shuffles}, that is permutations 
$$
(\mu,\nu)=(\mu_1\,\dotsc\mu_p,\nu_1,\dotsc,\nu_q)\in\Sigma_{p+q}
$$
where we have
$$
\mu_1<\dotsb<\mu_p\text{ and }\nu_1<\dotsb<\nu_q
$$
and the associated degeneracy maps are
$$
s_\mu=s_{\mu_p}\dotsb s_{\mu_1}\text{ and }s_\nu=s_{\nu_q}\dotsb s_{\nu_1}.
$$
Then the normalized Moore complex functor
$$
\bfMod_k^\Delta\xrightarrow N\bfMod_k^\dg
$$
and the shuffle map give a lax monoidal right Quillen equivalence which is moreover lax symmetric mononoidal \cite{schwede2003equivalences}*{\S4.2}. Let $A\in\CAlg_k^\Delta$ be a commutative simplicial $k$-algebra. We denote by $\pi_\bullet A$ the graded $k$-algebra induced by the commutative dg-algebra $NA$.

\begin{cons} Let $\bfMod_k^\dg\xrightarrow P\bfMod_k^\dg$ denote a dg-projective replacement functor. Consider the functor
$$
\DAff_k^\op\xrightarrow{A\mapsto((\bfMod_{PNA}^\dg)^c,\qis)}_{}\Mon_{\Assoc}\WCat_\infty.
$$
By Construction \ref{cons:underlying presheaf} we get a functor $\mathscr P(\DAff_k)\xrightarrow{QC}\Cat_\infty^{\Mon}$.

\end{cons}

\subsubsection{Spectral homotopical algebraic context}

\begin{defn} An \emph{$\mathbf E_\infty$-ring} is a commutative algebra object of the symmetric monoidal quasi-category $\Sp$ of spectra. The \emph{quasi-category $\SAff$ of affine spectral schemes} is the opposite quasi-category $(\CAlg\Sp)^\op$ of the quasi-category of $\mathbf E_\infty$-rings. Let $R\in\CAlg\Sp$ be an $\mathbf E_\infty$-ring. Then we denote by $\Spec R\in\SAff$ the object corresponding to $R$. We denote by $\SAff_R$ the overcategory $\SAff_{/\Spec R}$.

\end{defn}

\begin{defn} An \emph{$\mathbf E_1$-ring} is an algebra object of the symmetric monoidal quasi-category $\Sp$ of spectra. Let $R\in\Alg\Sp$ be an $\mathbf E_1$-ring. For an integer $n\in\mathbf Z$, let $\pi_nR=\pi_0\Map_{\Sp}(S[n],R)$. These objects admit a natural abelian group structure by Theorem \ref{thm:stable gives triangulated}. Since the smash product on spectra commutes with colimits in each variable, we get equivalences $S[n+m]\xrightarrow\alpha S[n]\otimes S[m]$. Let $R\otimes R\xrightarrow\mu R$ be the multiplication map. We get maps
$$
\Map_{\Sp}(S[n],R)\times\Map_{\Sp}(S[m],R)\to\Map_{\Sp}(S[n]\otimes S[m],R\otimes R)\xrightarrow{\mu\circ\ \circ\alpha}\Map_{\Sp}(S[n+m],R)
$$
endowing $\pi_\bullet R=\oplus_n\pi_nR$ with a graded ring structure.

\end{defn}

\begin{prop}\cite{lurie2014higher}*{Lemma 1.1.2.10} Let 
\begin{center}

\begin{tikzpicture}[xscale=2,yscale=2]
\node (A) at (0,1) {$C$};
\node (A') at (0,0) {$0'$};
\node (B) at (1,1) {$0$};
\node (B') at (1,0) {$D$};
\path[->,font=\scriptsize,>=angle 90]
(A) edge node [above] {$f$} (B)
(A) edge node [right] {$g$} (A')
(B) edge (B')
(A') edge (B');
\end{tikzpicture}

\end{center}
be a diagram in a stable quasi-category $\mathscr C$ representing an element $\theta\in\Hom_{\Ho\Sp}(C[1],D)$. Then the inverse $-\theta\in\Hom_{\Ho\Sp}(C[1],D)$ is represented by the transposed diagram
\begin{center}

\begin{tikzpicture}[xscale=2,yscale=2]
\node (A) at (0,1) {$C$};
\node (A') at (0,0) {$0$};
\node (B) at (1,1) {$0'$};
\node (B') at (1,0) {$D.$};
\path[->,font=\scriptsize,>=angle 90]
(A) edge node [above] {$g$} (B)
(A) edge node [right] {$f$} (A')
(B) edge (B')
(A') edge (B');
\end{tikzpicture}

\end{center}

\end{prop}

\begin{cor} Let $R\in\CAlg\Sp$ be an $\mathbf E_\infty$-ring. Then the graded ring $\pi_\bullet R$ is graded commutative.

\end{cor}

\begin{defn} Let $X\in \Sp$ be a spectrum. Then we say it is \emph{discrete} if we have $\pi_nX=0$ for all $n\ne0$. The functor $\Sp\xrightarrow{\pi_0}\Ab$ restricts to an equivalence on the full subcategory $\Disc\Sp\subseteq\Sp$ of discrete spectra \cite{lurie2014higher}*{Proposition 1.4.3.6 (3)}. Since the map $\mathscr S^\fin_*\to\mathscr S$ with constant value the point is a zero object and the smash product $\Sp\otimes\Sp\to\Sp$ respects colimits in each variable, the equivalence $\Disc\Sp\xrightarrow{\pi_0}\Ab$ is symmetric monoidal. In particular, we get a natural embedding $\CAlg\Ab\to\CAlg\Sp$ of the category of discrete commutative rings into the quasi-category of $\mathbf E_\infty$-rings.

\end{defn}

\begin{rem} Note that the unit object of $\CAlg\Sp$ is the sphere spectrum $S$ which is not discrete. In particular, this is not $\mathbf Z$ and thus $\SAff\simeq\SAff_S\not\simeq\SAff_{\mathbf Z}$.

\end{rem}

\begin{cons}\label{cons:QC} Since the smash product on $\Sp$ gives a symmetric monoidal structure, for an $\mathbf E_\infty$-ring $R$, relative tensor product equips the module category $\Mod_R$ with a symmetric monoidal structure. Moreover the underlying quasi-category $\Mod_R$ is stable \cite{lurie2014higher}*{Corollary 7.1.1.5} Therefore the Morita functor can be enhanced to give a functor $\CAlg\Sp\xrightarrow{R\mapsto\Mod_R^\otimes}\Mon_{\CAlg}\PrSt$ \cite{lurie2014higher}*{Corollary 4.8.5.22}. This can be extended to a colimit-preserving functor $\mathscr P(\SAff)\xrightarrow{\QC}(\Mon_{\CAlg}\PrSt)^\op$.

\end{cons}

\subsubsection{Homotopical Skolem--Noether theorem in homotopical algebraic geometry}

In this subsubsection, scheme will mean either a derived or spectral scheme and stack will mean either a derived or spectral stack. Perfect stacks, introduced in \cite{benzvi2010integral}, constitute a broad class of stacks to which we can apply Theorem \ref{thm:HSN}.

\begin{defn} Let $A\xrightarrow fB$ be a morphism of commutative simplicial or $\mathbf E_\infty$ rings. Then we say that it is a \emph{flat morphism} if the following conditions hold:
\begin{enumerate}
 \item The morphism $f$ induces an isomorphism of graded rings
 $$
 \pi_0B\otimes_{\pi_0A}\pi_\bullet A\to\pi_\bullet B
 $$
 \item The morphism $\pi_0(A)\xrightarrow{\pi_0f}\pi_0(B)$ is a flat morphism of rings.
\end{enumerate}
A flat morphism is an \emph{\'etale morphism} if moreover the morphism $\pi_0(A)\to\pi_0(B)$ is an \'etale morphism of rings.

Let $\{A\xrightarrow{f_i}A_i:i\in I\}$ be a collection of morphisms of commutative simplicial or $\mathbf E_\infty$ rings. Then we say that it is a \emph{flat covering} if there exists a finite subset $J\subseteq I$ such that the following conditions hold:
\begin{enumerate}
 \item For each index $j\in J$, the morphism $f_j$ is a flat morphism.
 \item The morphism
 $$
 \pi_0(A)\xrightarrow{\oplus_j\pi_0(f_j)}\oplus_k\pi_0(A_j)
 $$
 is faithfully flat.
\end{enumerate}
A flat covering is an \emph{\'etale covering} if there exists a finite subset $J\subseteq I$ such that moreover for each $j\in J$ the morphism $f_j$ is an \'etale morphisms.

\end{defn}

These definitions define the \emph{fpqc} resp.~ \emph{\'etale topology}. Thus we can talk about \emph{fpqc resp.~\'etale stacks}. One can show that $QC$ satisfies fpqc descent both in the derived \cite{toen2007moduli}*{\S3.1} and spectral \cite{lurie2011quasi}*{Proposition 2.7.14} context. We let $S$ be an affine scheme and $\St_S$ the quasi-category of \'etale stacks.

\begin{cor}[Homotopical Skolem--Noether theorem for derived and spectral algebraic geometry]\label{thm:HSN for DAG} Let $S$ be an affine scheme. Let $\Cart_S^\fpqc$ denote the quasi-category of Cartesian fibrations on $\St_S$ which satisfy fpqc descent.

(1) Let ${}^\op(\Perf^\simeq_\gen)_S={}^\op(\QC_S)_\dgen^\simeq$ denote the right fibration of perfect generator complexes on $S$, ${}^\op\Deraz_S:={}^\op\Az\QC_S$ the Cartesian fibration of derived Azumaya algebras on $S$ and ${}^\op\Dg^{\Az}_S:={}^\op\LTens^{\Az}\QC_S$ the Cartesian fibration of locally trivial presentable quasi-categories left-tensored over $\QC_S^\otimes$. Then the sequence in $(\Cart_S^\fpqc)_*$:
$$
({}^\op(\Perf^\simeq_\gen)_S,\mathscr O)\xrightarrow{\End}({}^\op\Deraz_S,\mathscr O)\xrightarrow{\Mod}({}^\op\Dg^{\Az}_S,\mathscr D)
$$
is a homotopy fibre sequence.

(2) Let $E\in\Perf_\gen(S)$ be a perfect generator complex on $S$. Then the sequence in $(\Cart_S^\fpqc)_*$:
$$
({}^\op\Pic\QC_S,\mathscr O)\xrightarrow{\otimes E}({}^\op(\Perf^\simeq_\gen)_S,E)\xrightarrow{\End}({}^\op\Deraz_S,\End E)
$$
is a homotopy fibre sequence.

\end{cor}

\begin{proof} Let $T$ be a qcqs scheme. Then $\QC(T)$ is compactly generated \cite{benzvi2010integral}*{Proposition 3.19}, thus dualizable by Proposition \ref{prop:compactly generated is dualizable}. Therefore by Proposition \ref{prop:dualizable commutes with descent} the family of monoidal quasi-categories $\QC^\otimes$ has fpqc descent, so we can apply Theorem \ref{thm:HSN}.

\end{proof}

\begin{rem}\label{rem:DgAz is B2Gm x BZ} In case we are in the derived or the connective spectral case, we have $\Dg_S^{\Az}\simeq\B^2\mathbf G_m\times\B\mathbf Z$ \cite{toen2012derived}*{Corollary 2.12}, \cite{antieau2014brauer}*{Corollary 7.10}.

\end{rem}

\begin{rem} For $\mathbf E_\infty$-rings this result has previously appeared in \cite{gepner2016brauer}*{Propotision 5.15}. Moreover, in \cite{gepner2016brauer}*{Theorem 3.15} they show how the long exact sequence splits for algebraic Azumaya algebras, that is derived Azumaya algebras for which the associated graded algebras are also Azumaya.

\end{rem}

\subsection{Ind-coherent sheaves and crystals} In this subsection, we apply the Homotopical Skolem--Noether Theorem to the co-families of symmetric monoidal quasi-categories $\Ind\Coh$ and $\Crys^r$, which we introduce following \cite{IndCoh}, \cite{drinfeld2013finiteness}, \cite{grCrystals} and \cite{grDAGI}. Let $k$ be a field of characteristic 0.

\subsubsection{Finiteness conditions on prestacks}

\begin{defn} The quasi-categories of $\mathbf E_\infty$- and commutative dg $k$-algebras are equivalent \cite{lurie2014higher}*{Proposition 7.1.4.11}. We let $\CAlg_k=\CAlg\mathscr D(k)$. Following this equivalence, we will refer to elements of $\CAlg_k$ as commutative dg $k$-algebras and for $i\in\mathbf Z$ and $A\in\CAlg_k$ we will write $H^iA=\pi_{-i}A$. We will say that a commutative dg $k$-algebra $A$ is \emph{connective} if we have $H^iA=0$ for $i>0$. We let $\CAlg_k^{\le0}\le\CAlg_k$ denote the full subcategory of connective commutative dg $k$-algebras. Then the quasi-category $\CAlg_k^\Delta$ of simplicial commutative $k$-algebras is equivalent to $\CAlg_k^{\le0}$ \cite{lurie2018spectral}*{Proposition 25.1.2.2}. The \emph{quasi-category of (derived) affine schemes} is $\Aff=\Aff_k=(\CAlg_k^{\le0})^\op$. A connective commutative dg $k$-algebra $A$ corresponds to the affine scheme $\Spec A\in\Aff$. The \emph{quasi-category of prestacks} is $\PreStk=\PreStk_k=\mathscr P(\Aff_k)$.

\end{defn}

\begin{defn} Let $S=\Spec A$ be an affine scheme and $n\in\mathbf Z_{\ge0}$. Then we say that $S$ is \emph{$n$-coconnective} if we have $H^iA=0$ for $i<-n$. We denote by $\Aff_{\le n}\subseteq\Aff$ the full subcategory on $n$-coconnective affine schemes. In particular, we say that $S$ is a \emph{classical affine scheme} if it is 0-coconnective. We let $\Aff_\cl=\Aff_{\le0}$. We say that $S$ is \emph{eventually coconnective} if it is $n$-coconnective for some $n\ge0$. We let $\Aff_{<\infty}=\cup_{n\ge0}\Aff_{\le n}$. We say that $S$ is \emph{almost of finite type} if the following assertions hold:
\begin{enumerate}
 \item The commutative $k$-algebra $H^0A$ is finitely generated.
 \item For any $i\in\mathbf Z_{\ge0}$, the $H^0A$-module $H^iA$ is finitely generated.
\end{enumerate}
We denote by $\Aff_{\aft}\subseteq\Aff$ the full subcategory on affine scheme almost of finite type. We say that $S$ is \emph{of finite type} if the following assertions hold:
\begin{enumerate}
 \item The affine scheme $S$ is eventually coconnective.
 \item The affine scheme $S$ is almost of finite type.
\end{enumerate}
We let $\Aff_\ft=\Aff_{<\infty}\cap\Aff_{\aft}$.

\end{defn}

\begin{defn} Take $n\in\mathbf Z_{\ge0}$. Let ${}_{\le n}\PreStk=\mathscr P({}_{\le n}\Aff)$. Let $\mathscr Y$ be a prestack. Then we denote its restriction to ${}_{\le n}\Aff^\op$ by ${}_{\le n}\mathscr Y$. The left Kan extension functor ${}_{\le n}\PreStk\xrightarrow{\LKE}\PreStk$ is a fully faithful left adjoint to the restriction functor. We say that $\mathscr Y$ is \emph{$n$-coconnective} if it is in the essential image of $\LKE$. We let $\tau_{\le n}\mathscr Y=\LKE({}_{\le n}\mathscr Y)$.
\begin{center}

\begin{tikzpicture}[xscale=3,yscale=2]
\node (A) at (0,0) {$\PreStk$};
\node (B) at (1,0) {${}_{\le n}\PreStk$};
\node at (1.5,0) {$\perp$};
\node (C) at (2,0) {$\PreStk$};
\path[->,font=\scriptsize,>=angle 90]
(B) edge [bend left] node [above] {$\LKE$} (C)
(A) edge [bend left=60] node [above] {$\tau_{\le n}$} (C)
(C) edge [bend left] node [below] {$\res$} (B)
(A) edge node [above] {$\res$} (B);
\end{tikzpicture}

\end{center}
In particular, we say that $\mathscr Y$ is \emph{classical} if it is 0-coconnective. We let ${}_\cl\PreStk={}_{\le0}\PreStk$, ${}_\cl\mathscr Y={}_{\le0}\mathscr Y$ and $\tau_\cl\mathscr Y=\tau_{\le0}\mathscr Y$. We say that $\mathscr Y$ is \emph{eventually coconnective} if it is $n$-coconnective for some $n\ge0$. We let $\PreStk_{<\infty}=\cup_{n\ge0}\PreStk_{\le n}$. We say that $\mathscr Y$ is \emph{convergent} if it is the right Kan extension of its restriction to $\Aff_{<\infty}^\op$. We say that $\mathscr Y$ is \emph{locally almost of finite type} if it is the right Kan extension along the inclusion $\Aff_{<\infty}^\op\subset\Aff^\op$ of the left Kan extension along the inclusion $\Aff_{\ft}^\op\subset\Aff_{<\infty}^\op$ of its restriction to $\Aff_\ft^\op$. We let $\PreStk_\laft\subset\PreStk$ denote the full subcategory on prestacks locally almost of finite type. We say that $\mathscr Y$ is \emph{of finite type} if it is the left Kan extension of its restriction to $\Aff_\ft^\op$. We let $\PreStk_\lft\subset\PreStk$ denote the full subcategory on prestacks locally of finite type.

\end{defn}

\subsubsection{Open embeddings and proper morphisms between (derived) schemes}

\begin{defn} Let $\mathscr X\xrightarrow f\mathscr Y$ be a morphism of prestacks. Then we say that $f$ is \emph{affine schematic} if for all affine schemes $S$ and morphisms of prestacks $S\to\mathscr Y$ the fibre product $\mathscr X\times_{\mathscr Y}S$ is representable by an affine scheme.

Suppose that $f$ is affine schematic. We say that $f$ is \emph{flat (resp.~\'etale, Zariski, an open embedding)} if for all $S\in\Aff_{/\mathscr Y}$ the morphism of affine schemes $\mathscr X\times_{\mathscr Y}S\to S$ is flat (resp.~\'etale, Zariski, an open embedding). We say that $f$ is a \emph{closed embedding} if for all $S\in\Aff_{/\mathscr Y}$ the morphism of classical affine schemes ${}_\cl(\mathscr X\times_{\mathscr Y}S)\to {}_\cl S$ is a closed embedding.

Let $(\mathscr X_i\xrightarrow{f_i}\mathscr Y:i\in I)$ be a collection of affine schematic morphisms of prestacks. Then we say that $(f_i)_{i\in I}$ is a \emph{covering} if for all $S\in\Aff_{/\mathscr Y}$ the collection of morphisms of classical affine schemes ${}_\cl(\mathscr X\times_{\mathscr Y}S)\to {}_\cl S$ is a covering.

\end{defn}

\begin{defn} Let $\mathscr Y$ be a prestack. Then we say that it is a \emph{stack} if it satisfies \'etale descent. We denote by $\Stk\subset\PreStk$ the full subcategory on stacks.

\end{defn}

\begin{defn} Let $Z$ be a stack. Then we say that it is a \emph{(derived) scheme} if it satisfies the following assumptions:
\begin{enumerate}
 \item The diagonal map $Z\xrightarrow\Delta Z\times Z$ is a closed embedding.
 \item There exists a covering $(S_i\xrightarrow{f_i}Z:i\in I)$ of open embeddings of affine schemes.
\end{enumerate}
We refer to a collection $(f_i:i\in I)$ such as in (2) as an \emph{affine atlas}. We denote by $\Sch\subset\Stk$ the full subcategory on schemes.

Let $Z$ be a scheme. We say that it is \emph{quasi-compact} if the classical scheme ${}_\cl Z$ is quasi-compact. Suppose that $Z$ is quasi-compact. Then we say that it is \emph{almost of finite type} if it is locally almost of finite type. We denote by $\Sch_\aft\subset\Sch$ the full subcategory on schemes almost of finite type.

Let $X\xrightarrow fY$ be a morphism of schemes almost of finite type. We say that $f$ is \emph{proper} if the morphism of classical schemes ${}_\cl X\xrightarrow{{}_\cl f}{}_\cl Y$ is proper.

\end{defn}

\subsubsection{Ind-coherent sheaves}

\begin{defn} Let $\mathscr C$ be a quasi-category. Then we say that $\mathscr C$ is \emph{filtered} if for all finite simplicial sets $K$ every map $K\to\mathscr C$ has an extension along the inclusion $K\to K^\vartriangleright$.

Let $L$ be a simplicial set. Then we say that $L$ is \emph{filtered} if there exists a filtered quasi-category $\mathscr C$ and a categorical equivalence $L\to\mathscr C$.

Let $\mathscr C$ be a filtered quasi-category and $\mathscr C\xrightarrow F\mathscr D$ a functor between quasi-categories. Then we say that $F$ is \emph{continuous} if it commutes with filtered colimits. We denote by $\Fun_\cont(\mathscr C,\mathscr D)\subseteq\Fun(\mathscr C,\mathscr D)$ the full subcategory on continuous functors.

\end{defn}

\begin{prop}\cite{lurie2009higher}*{Proposition 5.3.5.12} Let $\mathscr C$ be a small quasi-category. Let $\Ind\mathscr C\subseteq\mathscr P(\mathscr C)$ denote the full subcategory on presheaves $\mathscr C^\op\to\mathscr S$ which classify right fibrations $\Tilde{\mathscr C}\to\mathscr C$ such that the quasi-category $\Tilde{\mathscr C}$ is filtered. Then the Yoneda embedding $\mathscr C\to\mathscr P(\mathscr C)$ factors through the inclusion $\Ind\mathscr C\to\mathscr P(\mathscr C)$. Let $\mathscr D$ be a quasi-category such that it has filtered colimits. Then the precomposition with the Yoneda embedding map
$$
\Fun_\cont(\Ind\mathscr C,\mathscr D)\to\Fun(\mathscr C,\mathscr D)
$$
is an equivalence.

\end{prop}

\begin{defn} Let $\mathscr C$ be a quasi-category. We call $\Ind\mathscr C$ the \emph{ind-completion} of $\mathscr C$.

\end{defn}

\begin{defn} Let $X$ be a scheme almost of finite type. Then we denote by $\Coh X\subset\QC X$ the full subcategory on complexes with bounded coherent cohomology. The inclusion map $\Coh X\to\QC X$ induces a map $\Ind\Coh X\xrightarrow{\Psi_X}\QC X$. The map $\QC X\otimes\QC X\xrightarrow{E\boxtimes F=(\pr_0^*E)\otimes(\pr_1^*F)}\QC(X\times X)$ induces a map $\Ind\Coh X\otimes\Ind\Coh X\xrightarrow\boxtimes\Ind\Coh(X\times X)$. Let $X\xrightarrow fY$ be a morphism of schemes almost of finite type. Then the direct image functor $\QC X\xrightarrow{f_*}\QC Y$ induces a morphism $\Ind\Coh X\xrightarrow{f_*}\Ind\Coh Y$.

\end{defn}

\begin{thm}\label{thm:IndCoh} There exists a coCartesian family of presentable monoidal quasi-categories $\Ind\Coh^\otimes\xrightarrow p\PreStk_\laft^\op\times\Assoc^\otimes$ with the following properties:
\begin{enumerate}
 \item The family $p$ is classified by the right Kan extension of the map $\Sch_\aft^\op\to\Pr^{\Mon}$ classifying its restriction.
 \item Let $X\xrightarrow fY$ be a morphism of schemes almost of finite type. Then the Cartesian map $\Ind\Coh Y\xrightarrow{f^!}\Ind\Coh X$ is a composite $f_1^*f_2^!$ where
 \begin{enumerate}
  \item $X\xrightarrow{f_1}\Bar X\xrightarrow{f_2}Y$ is a decomposition of $f$ by an open embedding followed by a proper morphism.
  \item The map $f_1^*$ is a left adjoint of $(f_1)_*$.
  \item The map $f_2^!$ is a right adjoint of $(f_2)_*$ \cite{grDAGI}*{II, Theorem 5.2.1.4}.
 \end{enumerate}
 \item The underlying Cartesian fibration ${}^\op\Ind\Coh\to\Sch_\aft$ has fppf descent \cite{grDAGI}*{II, Corollary 5.3.3.7}.
 \item The monoidal structure \cite{grDAGI}*{II, Theorem 5.4.1.2} can be given as the composite
 $$
 \otimes^!:\Ind\Coh X\otimes\Ind\Coh X\xrightarrow\boxtimes\Ind\Coh(X\times X)\xrightarrow{\Delta^!}\Ind\Coh X.
 $$
 \item Let $X$ be a scheme almost of finite type. Then the presentable quasi-category $\Ind\Coh X$ is self-dual \cite{grDAGI}*{II, Theorem 5.4.2.5}.
\end{enumerate}

\end{thm}

\begin{defn} Let $\mathscr Y\xrightarrow p\Spec k$ be a stack locally almost of finite type. The unit object of the symmetric monoidal quasi-category $\Ind\Coh^{\otimes}(\mathscr Y)$ is the \emph{dualizing complex} $\omega_{\mathscr Y}:=p^!(k)$. The map $\Ind\Coh(\mathscr Y)\xrightarrow{\Psi_{\mathscr Y}}\QCoh(\mathscr Y)$ admits a symmetric monoidal left adjoint $\QCoh(\mathscr Y)\xrightarrow{\Upsilon_{\mathscr Y}}\Ind\Coh(\mathscr Y)$ \cite{grDAGI}*{II, \S 6.3.3}, which restricts to an equivalence on dualizable objects \cite{grDAGI}*{II, Lemma 6.3.3.7}. Let $F=E\otimes\omega_{\mathscr Y}\in\Ind\Coh(\mathscr Y)$ be a dualizable object. Then its \emph{support} $\Supp(F)$ is $\Supp(E)$. We let $\TCoh\subseteq\Coh$ denote the full substack on totally supported coherent complexes and we let $\Az\Ind\Coh\subseteq\Alg\Ind\Coh^{\otimes^!}$ denote the full substack on algebra objects locally equivalent to endomorphism algebras of totally supported coherent complexes. 

\end{defn}

\begin{cor}[Homotopical Skolem--Noether Theorem for $\Ind\Coh$]\label{thm:HSN for IndCoh} The following assertions hold:
\begin{enumerate}
 \item Then the following is a fibre sequence in $\Stk_{\lft}$:
 $$
 {}^\op\TCoh^\simeq\xrightarrow{\End}\Az\Ind\Coh^\simeq\xrightarrow{\Mod}{}^\op\LTens^{\Az}\Ind\Coh^\simeq\simeq(\B^2\mathbf G_m\times\B\mathbf Z).
 $$
 \item Let $\mathscr Y$ be a stack locally of finite type and $E\in\TCoh(\mathscr Y)$ a totally supported complex with bounded coherent cohomology sheaves on $\mathscr Y$. Then the following is a fibre sequence in $(\Stk_{\lft})_{/\mathscr Y}$:
 $$
 ((\B\mathbf G_m)\times\mathbf Z)_{\mathscr Y}\simeq{}^\op\Pic\Ind\Coh_{\mathscr Y}\xrightarrow{\otimes E}{}^\op\TCoh_\mathscr Y^\simeq\xrightarrow{\End}{}^\op\Az\Ind\Coh_\mathscr Y^\simeq.
 $$
\end{enumerate}
 
\end{cor}

\begin{proof} By Theorem \ref{thm:IndCoh}, we can apply the abstract Homotopical Skolem--Noether Theorem, Theorem \ref{thm:HSN} to the co-family $\Ind\Coh^{\otimes^!}$. The rest of the Proof consists of identifying the output of the Homotopical Skolem--Noether Theorem.

Let $X\xrightarrow p\Spec k$ be an affine scheme of finite type. Then by Remark \ref{rem:DgAz is B2Gm x BZ} we have $(\B\mathbf G_m\times\mathbf Z)(X)\simeq\Pic\QCoh(X)$. The equivalence on dualizable objects $\QCoh(X)^d\xrightarrow{\Upsilon_X(E)=E\otimes\omega_X}\Ind\Coh(X)^d$ restricts to an equivalence $\Pic\QCoh(X)\simeq\Pic\Ind\Coh(X)$. The full subcategory $\Coh(X)\subseteq\Ind\Coh(X)$ is precisely the full subcategory of compact objects \cite{IndCoh}*{Corollary 1.2.6}. Moreover, the dualizing complex $\omega_X\in\Ind\Coh(X)$ is compact \cite{IndCoh}*{Corollary 9.6.4}, therefore every dualizable object of $\Ind\Coh^{\otimes^!}$ is compact. By definition, a coherent complex $E\otimes\omega_X\in\Coh X$ is totally supported if and only if the perfect complex $E\in\Perf X$ is totally supported.

\end{proof}

\begin{rem} Let $\mathscr Y$ be a stack locally of finite type. Note that via the equivalence $\Perf(\mathscr Y)^\otimes\xrightarrow{\Upsilon_{\mathscr Y}}\Coh(\mathscr Y)^{\otimes^!}$ the result of Corollary \ref{thm:HSN for IndCoh} is equivalent to Corollary \ref{thm:HSN for DAG} for Derived Algebraic Geometry. But it might offer a new look at derived Azumaya algebras, in particular in cases where $\omega_{\mathscr Y}\in\Ind\Coh(\mathscr Y)$ is compact but $\mathscr O_{\mathscr Y}\in\QC(\mathscr Y)$ is not.

\end{rem}

\subsubsection{Crystals and D-modules}

\begin{defn} Let $S$ be an affine scheme. The we say that it is a \emph{reduced scheme} if it is a reduced classical scheme. We denote by ${}_\red\Aff\subset{}_\cl\Aff$ the full subcategory on reduced affine schemes and we let ${}_\red\PreStk=\mathscr P({}_\red\Aff)$.

The \emph{de Rham functor} is the composite $\dR:\PreStk\xrightarrow\res{}_\red\PreStk\xrightarrow\RKE\PreStk$ of the restriction and the right Kan extension functors. Let $\mathscr Y\in\PreStk$ be a prestack. Then the \emph{de Rham prestack $\mathscr Y_{\dR}$} of $\mathscr Y$ is the image $\dR(\mathscr Y)$.

\end{defn}

\begin{prop}\cite{grCrystals}*{Proposition 1.3.3} Let $\mathscr Y\in\PreStk_\laft$ be a prestack locally almost of finite type. Then the de Rham prestack $\mathscr Y_{\dR}$ is a classical prestack locally almost of finite type.

\end{prop}

\begin{defn} The \emph{left crystals functor} is the composite $\Crys^l:\PreStk\xrightarrow{\dR}\PreStk\xrightarrow{\QC}\Pr^{\Mon}$. The \emph{right crystals functor} is the composite $\Crys^r:\PreStk_\laft\xrightarrow{\dR}\PreStk_\laft\xrightarrow{\Ind\Coh}\Pr^{\Mon}$. They are related by a symmetric monoidal natural equivalence $(\Crys^l|\PreStk_\laft)\xrightarrow\Upsilon\Crys^r$ \cite{grDAGI}*{II, \S6.3.2}, \cite{grCrystals}*{Proposition 2.4.4}.

\end{defn}

\begin{rem} Let $X$ be a classical scheme of finite type. Then there exist equivalences $\Crys^lX\simeq\DMod^lX$ and $\Crys^rX\simeq\DMod^rX$ \cite{grCrystals}*{\S5.5} where $\DMod^lX$ (resp.~$\DMod^rX$) are the stable quasi-categories of left (resp.~right) D-modules on $X$.

\end{rem}

\begin{cor}[Homotopical Skolem--Noether Theorem for $\Crys$]\label{thm:HSN for Crys} The following assertions hold:
\begin{enumerate}
 \item Then the following is a fibre sequence in $\Stk_{\lft}$:
 $$
 {}^\op(\Crys_{\Coh}^{r,\gen})^\simeq\xrightarrow{\End}(\Az\Crys^r)^\simeq\xrightarrow{\Mod}{}^\op\LTens^{\Az}(\Az\Crys^r)^\simeq\simeq(\B^2\mathbf G_m\times\B\mathbf Z)_{\dR}.
 $$
 \item Let $\mathscr Y$ be a stack locally of finite type and $E\in\Crys_{\Coh}^{r,\dgen}(\mathscr Y)$ a generator complex with bounded coherent cohomology sheaves on $\mathscr Y$. Then the following is a fibre sequence in $(\Stk_{\lft})_{/\mathscr Y_{\dR}}$:
 $$
 ((\B\mathbf G_m)\times\mathbf Z)_{\mathscr Y_{\dR}}\simeq{}^\op\Pic\Crys^r_{\mathscr Y}\xrightarrow{\otimes E}{}^\op(\Crys_{\Coh}^{r,\gen})_\mathscr Y^\simeq\xrightarrow{\End}{}^\op(\Az\Crys^r)_\mathscr Y^\simeq.
 $$
\end{enumerate}
 
\end{cor}

\begin{proof} The underlying Cartesian fibration of $\Crys^r$ satisfies fppf descent \cite{grCrystals}*{Corollary 3.2.4}. For a scheme $X$ locally almost of finite type, the stable quasi-category $\Crys^r(X)$ is compactly generated \cite{grCrystals}*{Corollary 3.3.3}, therefore it is dualizable \cite{lurie2018spectral}*{Proposition D.7.2.3}. This shows by Proposition \ref{prop:dualizable commutes with descent} that we can apply the abstract Homotopical Skolem--Noether Theorem. To finish, we need to identify its output.

Let $X$ be a scheme of finite type. Then just as in Corollary \ref{thm:HSN for IndCoh}, the dualizable objects in $\Crys^r(X)$ coincide with the compact objects, which in turn coincide with the D-modules with coherent underlying complex \cite{drinfeld2013finiteness}*{\S 5.1.17}. Moreover the equivalence $\Crys^l(X)\to\Crys^r(X)$ restricts to the equivalence $((\B\mathbf G_m)\times\mathbf Z)(X)\simeq\Pic\Crys^l(X)\simeq\Pic\Crys^r(X)$.

\end{proof}

\begin{defn} Let $\id_{\PreStk}\xrightarrow u\dR$ denote the unit of the adjunction $\smalladjoints{\PreStk}{{}_\red\PreStk}{\res}{\RKE}$. Let $\mathscr Y$ be a prestack. Then a \emph{twisting} on $\mathscr Y$ is a $\mathbf G_m$-gerbe $T$ on $\mathscr Y_{\dR}$ equipped with a trivialization of the pullback $T|\mathscr Y$ along $u$.

Let $T$ be a twisting on $\mathscr Y$. Then via the symmetric monoidal structure on $\QC$, the $\mathbf G_m$-gerbe $T$ acts on the quasi-category $\Crys^l(\mathscr Y)$ of left crystals. Therefore, we can form the quasi-category $\Crys^{T,l}(\mathscr Y)$ of \emph{$T$-twisted left crystals} on $\mathscr Y$.

\end{defn}

\begin{cor}\label{cor:twisted crystals as deraz} Let $\mathscr Y$ be a prestack and $T$ a twisting on $\mathscr Y$. Then the quasi-category ${}^\dgen\Crys^{T,l}(\mathscr Y)$ of $T$-twisted left crystals on $\mathscr Y$ that are dualizable generators is equivalent to the quasi-category $\Deraz^T(\mathscr Y_{\dR})$ of derived Azumaya algebras on $\mathscr Y_{\dR}$ with Brauer class $T$.

\end{cor}

\begin{proof} This is a formal consequence of the Homotopical Skolem--Noether Theorem for Derived Algebraic Geometry, Corollary \ref{thm:HSN for DAG}.

\end{proof}

\end{document}